\documentclass[11pt,openany,leqno]{article}
\usepackage{amsmath,amsthm,amsfonts,amssymb,amscd,url}
\usepackage{graphics}
\usepackage[latin1]{inputenc}

\usepackage{enumerate}
\usepackage{hyperref}
\usepackage{epsfig} 
\input{xy} \xyoption{all}
\usepackage{makeidx}
\makeindex
\begin{document}
\def\Diff{\text{Diff}}
\def\Max{\text{max}}
\def\Log{\text{log}}
\def\loc{\text{loc}}
\def\inta{\text{int }}
\def\det{\text{det}}
\def\exp{\text{exp}}
\def\Re{\text{Re}}
\def\lip{\text{Lip}}
\def\leb{\mathrm{Leb}}
\def\dom{\mathrm{Dom}}
\def\diam{\text{diam}\:}
\def\supp{\text{supp}\:}
\def \ola{\overleftarrow}

\def\sA{{\mathfrak A}}
\def\sB{{\mathfrak B}}
\def\sC{{\mathfrak C}}
\def\sD{{\mathfrak D}}
\def\sE{{\mathfrak E}}
\def\sG{{\mathfrak G}}
\def\sL{{\mathfrak L}}
\def\sM{{\mathfrak M}}
\def\sN{{\mathfrak N}}
\def\sP{{\mathfrak P}}
\def\sR{{\mathfrak R}}
\def\sS{{\mathfrak S}}
\def\sY{{\mathfrak Y}}
\def\sT{{\mathfrak T}}

\def\sa{{\mathfrak a}}
\def\sb{{\mathfrak b}}
\def\sc{{\mathfrak c}}
\def\sd{{\mathfrak d}}
\def\sg{{\mathfrak g}}
\def\sd{{\mathfrak d}}
\def\ss{{\mathfrak s}}
\def\se{{\mathfrak e}}
\def\sp{{\mathfrak p}}
\def\sm{{\mathfrak m}}
\def\sn{{\mathfrak n}}
\def\st{{\mathfrak t}}

\def\R{{\mathbb R}}
\def\Z{{\mathbb Z}}
\def\N{{\mathbb N}}
\newcommand{\ovfork}{{\overline{\pitchfork}}}
\newcommand{\ovforki}{{\overline{\pitchfork}_{I}}}
\newcommand{\Tfork}{{\cap\!\!\!\!^\mathrm{T}}}
\newcommand{\whforki}{{\widehat{\pitchfork}_{I}}}

\newtheorem{thm}{Theorem}
\renewcommand{\thethm}{\Alph{thm}}

\theoremstyle{plain}
\newtheorem{theo}{\bf Theorem}[section]
\newtheorem{lemm}[theo]{\bf Lemma}
\newtheorem{ques}[theo]{\bf Question}
\newtheorem{sublemm}[theo]{\bf Sublemma}
\newtheorem{IH}[theo]{\bf Extra induction hypothesis}
\newtheorem{prop}[theo]{\bf Proposition}
\newtheorem{coro}[theo]{\bf Corollary}
\newtheorem{Property}[theo]{\bf Property}
\newtheorem{Claim}[theo]{\bf Claim}
\theoremstyle{remark}
\newtheorem{rema}[theo]{\bf Remark}
\newtheorem{remas}[theo]{\bf Remarks}

\newtheorem{exem}[theo]{\bf Example}
\newtheorem{Examples}[theo]{\bf Examples}
\newtheorem{defi}[theo]{\bf Definition}

\newcommand\relatif{{\rm \rlap Z\kern 3pt Z}}
\makeatletter
\renewcommand\theequation{\thesection.\arabic{equation}}
\@addtoreset{equation}{section}
\makeatother

\title{Properties of the maximal entropy measure and geometry of Hénon attractors}

\author{Pierre Berger\\ CNRS-LAGA Université Paris 13\\
berger@math.univ-paris13.fr}
\date{\today}
\maketitle

\begin{abstract} 

We consider an abundant class of non-uniformly hyperbolic $C^2$-Hénon like diffeomorphisms called strongly regular and which corresponds to Benedicks-Carleson parameters. We prove the existence of $m>0$ such that for any such diffeomorphism $f$, every invariant probability measure of $f$ has a Lyapunov exponent greater than $m$, answering a question of L. Carleson. Moreover, we show the existence and uniqueness of a measure of maximal entropy, this answers a question of M. Lyubich and Y. Pesin.  
We also prove that the maximal entropy measure is equi-distributed on the periodic points and is finitarily Bernoulli, which gives an answer to a question of J.P. Thouvenot. Finally, we show that the maximal entropy measure is exponentially mixing and satisfies the central limit Theorem. The proof is based on a new construction of Young tower for which the first return time coincides with the symbolic return time, and whose orbit is conjugated to a strongly positive recurrent Markov shift.
\end{abstract}
\tableofcontents

\section*{Introduction}
The theory of uniformly hyperbolic dynamical systems is based on several paradigmatic examples which are the doubling angle map of the circle, the Smale solenoid, the Smale horseshoe and the Anosov map, all which are uniformly hyperbolic, transitive, locally maximal compact sets, called \emph{basic piece}. 

An invariant compact set $\Lambda$ for a diffeomorphism $f$ is \emph{(uniformly) hyperbolic} if there exists a $Df$-invariant splitting $E^s\oplus E^u$ of the tangent space restricted to $\Lambda$, and there exist
 $C>0, \lambda>1$ such that 
 for every $x\in \Lambda$, for any unit vectors $u\in E^s(x),$ $v\in E^u(x)$, it holds:
\[ \|D_xf^n(u)\|\le C\lambda^n\quad \mathrm{and}\quad 
\|D_xf^n(v)\| \ge \lambda^n/C
\quad \forall n\ge 0.\]
An invariant compact set $\Lambda$  is \emph{locally maximal} if there exists a neighborhood $U$ of $\Lambda$ so that $\cap_{n\in \mathbb Z} f^n(U)=\Lambda$.  It is an \emph{attractor} if $\cap_{n\ge 0} f^n(U)=\Lambda$.

They satisfy the following properties:
\paragraph{Persistence} For any perturbations $f'$ of $f$, there exists a unique basic set $\Lambda'$ for $f'$ which is close to $\Lambda$ (they are even homeomorphic and the dynamics $f|\Lambda$ and $f'|\Lambda'$ are conjugated).
\paragraph{SRB measure} If $\Lambda$ is an attractor, then there exists an ergodic measure which is SRB (its conditional measure $\mu$ with respect to the unstable manifolds of $\Lambda$ is absolutely continuous). Moreover Lebesgue almost every point $z$ in the neighborhood of $\Lambda$ belongs to the \emph{basin} of $\mu$: its Birkhoff sum converges to $\mu$:
\[\frac 1N\sum_{j=0}^{N-1}\delta_{f^j(z)} \rightharpoonup\mu\] 
\paragraph{Coding} The set $\Lambda$   admits a (finite) Markov partition. This implies that its dynamics is semi-conjugated with a subshift of finite type. The semi-conjugacy is 1-1 on a generic set. Its lack of injectivity is itself coded by subshifts of finite type of smaller topological entropy. This enables one to study efficiently all the invariant measures of $\Lambda$, to show the existence and uniqueness of the maximal entropy measure $\nu$, and to show the equidistribution of the periodic points w.r.t. $\nu$:
\[\frac1{Card\; Fix\; f^n}\sum_{z\in Fix\,f^n}\delta_z \rightharpoonup\nu\;.\]
\bigskip

Let us recall the definitions of \emph{entropy}.  For two covers $\mathcal O$ and $\mathcal O'$ of $M$, the family of intersections of a set from $\mathcal O$ with a set from $\mathcal O'$ forms a covering $\mathcal O \vee \mathcal O'$, and similarly for multiple covers.
For any finite open cover $\mathcal O$ of $M$, let $H(\mathcal O)$ be the logarithm of the smallest number of elements of $\mathcal O$ that cover $M$.
 The following limit exists:
 \[   H(\mathcal O,f) = \lim_{n\to\infty} \frac{1}{n} H(\mathcal O\vee f^{-1}\mathcal O\vee \cdots\vee f^{-n}\mathcal O).\] 
The \emph{topological entropy $h(f)$ of $f$} is the supremum of $H(\mathcal O, f)$ over all finite covers $\mathcal O$ of $M$.

Given a measure $\mu$, the \emph{entropy of $\mu$} is defined similarly. For a finite partition $\mathcal O$, put:
 \[   H_\mu(\mathcal O,f) = \lim_{n\to\infty} \frac{1}{n} \sum_{E\in 
 \mathcal O\vee f^{-1}\mathcal O\vee \cdots\vee f^{-n}\mathcal O} -\mu (E) \log \mu (E).\]
 Then the \emph{entropy $h_\mu$  of $\mu$} is the supremum of $H_\mu(\mathcal O, f)$ over all possible finite partitions $\mathcal O$ of $M$. 
 
From the Variational Principle, the topological entropy is the supremum of entropies of invariant probability measures:
\[h(f)=\sup\{h_\mu(f):\mu \;\text{probability } f\text{-invariant}\}.\] 
Therefore the topological entropy is an \emph{ergodic invariant}, {\it i.e.} it is invariant by bi-measurable conjugacy. 

A  probability $\mu$ has \emph{maximal entropy} if $h(f)=h_\mu(f)$. 
\bigskip

The non-uniformly hyperbolic theory is still in construction. It should involve the notion of \emph{hyperbolic invariant measure}. Let us recall that given a $C^{1+\alpha}$-diffeomorphism $f$ and an invariant, ergodic probability measure $\mu$, the Oseledets multiplicative ergodic theorem produces
a $\mu$-a.e $Df$-invariant splitting of the tangent bundle into  subbundles $E^c$, $E^s$ and $E^u$ and Lyaponuv exponents $\lambda_u>0>\lambda_s$ so that for $\mu$ a.e. $z$:
\[\lim_{n\to \pm\infty} \frac1n \log\|D_zf^n|E^c\|=0, \;\lim_{n\to \pm\infty} \frac1n \log\|D_zf^n|E^u\|\ge\lambda_u\quad \mathrm {and} \quad 
\lim_{n\to \pm\infty} \frac1n \log\|D_zf^n|E^s\|\le\lambda_s\]
The measure is called \emph{hyperbolic} if $E^c_z=0$ for $\mu$ a.e. $z$.

The non-uniformly hyperbolic theory is based on a few paradigmatic examples which are the attractor of a Collet-Eckmann quadratic map, the attractor of a Hénon like map of Benedicks-Carleson type, and the locally maximal non-uniformly hyperbolic horseshoes. They are the non-uniformly hyperbolic correspondents to respectively the doubling angle map of the circle, the Smale solenoid and the Smale horseshoe. 

All these compact sets  are  transitive and locally maximal sets although they are \emph{not} uniformly hyperbolic.  They are all persistent in the following sense : 
\paragraph{Abundance} For an open set of deformations $(f_a)_a$ of the dynamics $f_0=f$, for a set of parameters $a$ of Lebesgue measure positive, there exists compact set $\Lambda_a$ for $f_a$, which is transitive, locally maximal and endowed with an abstract structure which presents similar properties to those provided by the uniformly hyperbolic theory.  The abundance of Collet-Eckmann quadratic maps, is the well known Jacobson's Theorem \cite{Ja81, BC1,Ts93, Y}. The abundance of transitive H\'enon attractor is the celebrated Benedicks-Carleson \cite{BC2, MV93, YW, Ta11, berhen}. The abundance of non-uniformly hyperbolic horseshoes has been introduced by Palis-Yoccoz \cite{PY09, Ta12}.

\paragraph{SRB measure for attractors} 
Every Collet-Eckmann map preserves a unique absolutely invariant measure (SRB) \cite{CE83}.  The existence of the SRB measure for Benedick-Carleson parameters was proved by Benedicks-Young \cite{BY}. In \cite{BV2}, Benedicks-Viana  proved that the basin of the SRB contains Lebesgue a.e. point in the neighborhood of the attractor.   The paper \cite{Yo98} provides a general setting where appropriate hyperbolicity hypotheses allow to construct hyperbolic SRB measures with nice statistical properties.

\paragraph{Coding} 
In \cite{H81}, a coding is given to prove the existence and the uniqueness of the maximal entropy measure for unimodal maps of positive entropy (such as Collet-Eckmann maps).

In \cite{YW}, a certain coding is given in order to prove the existence of a maximal entropy measure for H\'enon attractors of Benedicks-Carleson type, but the formalism does not seem to imply easily its uniqueness.
 
In \cite{PY09}, a certain Markovian coding is given on the maximal invariant set, but it is not easy to see if this implies the uniqueness of the maximal entropy measures.


In finite regularity, a measure of maximal entropy needs not exist \cite{Gu1969}.
Nevertheless, a famous theorem of Newhouse states the existence of a maximal entropy measure for every smooth diffeomorphism \cite{NH}.  

Also given any $C^{1+\alpha}$-diffeomorphism of a compact surface of positive topological entropy greater than $\chi  > 0$, Sarig constructed a countable Markov chain for an invariant set which has full measure w.r.t. any ergodic invariant measure with metric entropy $> \chi$ \cite{Sa10}. The semi-conjugacy associated to this Markov partition is finite-to-one. 

\bigskip

In this work we study the ergodic properties of the unique paradigmatic example of non-uniformly hyperbolic attractor for surface diffeomorphisms: the Benedicks-Carleson attractors for Hénon-like maps.  It is given by diffeomorphisms $f_{a\; B}$ of $\R^2$ which are of the form: 
\[f_{a\; B}:\; (x,y)\mapsto (x^2+a+y,0)+B(x,y,a),\]
where $B\in C^2(\mathbb R^2\times \mathbb R, \mathbb R^2)$ is uniformly $C^2$-close to $0$. We denote by $b$ an upper bound of the uniform $C^2$-norm of $B|[-3,3]^2$. In \cite{berhen}, the following
 analogous to Benedicks-Carleson  Theorem is shown:
\begin{theo}\label{Mainold}
For any $\eta>0$, for any $a_0$ greater but sufficiently close to $-2$, there exists $b>0$ such that for any $B|[-3,3]^2\times \R$ with $C^2$-norm less than  ${b}$, there exists a subset $\Omega_B \subset [-2,a_0]$  such that $\frac{\leb \, \Omega_B}{\leb \, [-2,a_0]}>1-\eta$ and for every $a\in \Omega_B$, the map $f_{a\; B}$ is strongly regular.
\end{theo}
The definition of strong regularity is recalled in section \ref{rappeldesdef}. We showed in \cite{berhen} that this implies, for each $a\in \Omega_B$, $f_{a,B}$ leaves invariant a unique physical, ergodic, SRB probability measure.

We notice that the Jacobean $det\; Df_{a\; B}$ is small. We assume it smaller than $b$. Hence every  ergodic probability measure has a negative Lyapunov exponent. The first theorem is an answer a question  of L. Carleson (as related by S. Newhouse during the first Palis-Balzan conference):
\begin{thm}\label{MainNew}
For every strongly regular Hénon-like map $f$,  there exists $m>0$ so that for every invariant, ergodic probability measure $\mu$ has a Lyapunov exponent greater than $m$.
\end{thm}
The same conclusion has been recently proved for non-uniformly hyperbolic horseshoe which appears as perturbations of the first bifurcation of Hénon-like maps \cite{Ta13}. 

A second result is the following. 
\begin{thm}[Main result]\label{Main}
Every strongly regular regular Hénon-like diffeomorphisms $f$ leaves invariant a unique probability of maximal entropy $\nu$. Moreover $\nu$ is equi-distributed on the periodic points of $f$, finitarily Bernoulli, exponentially mixing and it satisfies the central limit Theorem. 
\end{thm}

A \emph{Bernouilli shift} is the shift dynamics of $\Sigma_N:=\{1,\dots, N\}^\mathbb Z$ endowed with the product probability $p^\mathbb Z$ spanned by a probability $p=(p_i)_{i=1}^N$ on $\{1,\dots, N\}$. The entropy of the probability $p^\mathbb Z$ is $h_{p}=-\sum_i p_i\log p_i$.
 By Ornstein and Kean-Smorodinsky isomorphism Theorems, any two Bernouilli shifts $(\Sigma_N, p^\mathbb Z)$ and $(\Sigma_{N'}, p'^{\mathbb Z})$ with the same entropy $h_{p}=h_{p'}$  are \emph{finitarily isomorphic} \cite{KS79}.  A bi-measurable isomorphism is \emph{finitary} if it and its inverse send open sets to open sets, modulo null sets.
%
 
To be \emph{finitarily Bernoulli} means that the dynamics, with respect to the maximal entropy measure, is  finitarily isomorphic  to a Bernouilli shift.


The \emph{central limit Theorem} is that for every  Hölder function $\psi$ of $\nu$-mean $0$, such that $\psi\not= \phi-\phi\circ f$ for any $\phi$ continuous, there exists $\sigma>0$ such that $\frac 1{\sqrt n} \sum_{i=1}^n \Psi\circ f^i$ converges in distribution (w.r.t. $\nu$)  to the normal distribution with mean zero and standard deviation $\sigma$.
    
The measure  $\nu$ is \emph{exponentially mixing} if there exists $0<\kappa<1$ such that for every pair of functions  of the plane $g\in L^\infty(\nu)$ and $h$ Hölder continuous, there is $C(g,h)>0$ satisfying for every $n\ge 0$:
\[Cov_\nu(g,h\circ f^n)<C(g,h) \kappa^n,\text{\; with $Cov$ the covariance.}\]

\thanks{ This work has been Partially supported by the Balzan Research Project of J. Palis and the project BRNUH of  Sorbonne Paris Cité university. I am very grateful to M. Lyubich for presenting me his problem, and its geometric vision of it. I am very thankful to O. Sarig for many explanations on the concept of entropy in symbolic dynamics. I would like also to acknowledge  M. Benedicks, J.-P. Thouvenot, F. Ledrappier, J. Buzzi, Y. Pesin, S. Senti and M. Viana for helpful discussions.}

\section*{Structure of the paper}
\paragraph{In Section \ref{rappeldesdef},} we explain the notion of strong regularity. In order to make   this concept transparent, we will state first Yoccoz' definition in the one dimensional case \cite{Y} and then the definition of \cite{berhen}. The definition of strong regularity involves a countable set of symbols $\sA$ and a certain algebraic structure on a subset of $\sA$-words that we call puzzle algebra.   
To each symbol $\sa\in \sA$ is associated 
two graph transforms and an integer $n_\sa$ called the order.

\paragraph{In Section \ref{defregularsection},} for every strongly regular Hénon like map $f$, we use the alphabet $\sA$ to encode some points $z$ in the neighborhood of the attractor as a sequence $\underline \sa(z)=(\sa_i)_i\in \sA^{\mathbb N}$.  Whenever the sequence $(\sa_i)_i$ has its orders which satisfy 	a certain linear bound from above, the sequence $\underline \sa(z)$ and the point $z$ are called \emph{regular}. This defines a subset of sequences $\tilde \sR \subset \sA^\mathbb N$ and a subset $\tilde {\mathcal R}$ in the neighborhood of the attractor. Actually $\tilde {\mathcal R}$ is a fibration by local stable manifolds $(W^s_{\underline \sa})_{\underline \sa\in \tilde \sR}$ as proved in Corollary \ref{geodesvarstable}. 
Moreover the curve $W^s_{\underline \sa} $ is $(1/|\log b|)^k$-contracted by $f^k$ for every $k\ge 0$, whereas  its normal vectors are expanded by a factor $m^k$, with $m>1$, for $k$ large enough.
 
 For $x\in W^s_{\underline \sb}$ with $\underline \sb\in \tilde \sR$, we put $\underline \sa(x):=\underline \sb$.
 
In Proposition \ref{eventually regu} we show that every invariant ergodic probability measure $\mu$ has its support either included in the orbit of $\tilde {\mathcal R}$, either in a certain uniformly hyperbolic set $\hat K_\square$, or in the fixed points $\{A,A'\}$. This implies Theorem \ref{MainNew}. 

Then, we consider the subset $\sR$ of $\tilde \sR$ made by the points which  return infinitely many times in $\tilde { \sR}$, by the shift map $\tilde \sigma$ of $\sA^{\mathbb N}$: 
\[\sR:=\{\sa=(\sa_i)_{i\ge 0}\in \tilde \sR:\; (\sa_{i+N})_{i\ge 0}\in \tilde \sR,\quad \text{for infinitely many }N\ge 0\}\]
It follows from the definition that the points of $\sR$ come back infinitely many times in $\sR$.

This split $\tilde {\mathcal R}=\cup_{\underline \sa\in \tilde \sR}W^s_{\underline \sa}$ into two subsets: 
\[\mathcal R=\cup_{\underline \sa\in \sR} W^s_{\underline \sa}\quad \text{and}\quad \mathcal E=\cup_{\underline \sa\in \tilde \sR\setminus \sR}W^s_{\underline \sa}\;.\] 

\paragraph{In Section \ref{sectionMarkov},} we define a Young tower on a subset $\Lambda$ whose obit supports the same probability measures as the orbit of $\mathcal R$, and so that the first return time of the dynamics in $\Lambda$ corresponds to the return time given by the tower structure.  The latter property does not appear in \cite{YW}, despite they construct an encoding which implies the existence (but not the uniqueness) of the maximal entropy measure. 

For this end, given $\underline \sa\in \sR$, we define $N_{\sR}(\underline \sa)$ as the first return time in $\sR$ of $\underline \sa$ by the shift dynamics $\tilde \sigma$ of $\sA^\N$. This defines  combinatorial return time $N_\sR(\underline \sa)=n_{\sa_0}+\cdots +n_{\sa_{N_{\sR}(\underline \sa)-1}}$ and a first return map $\tilde \sigma^\sR$ of $\sR$ with $\tilde \sigma^{\sR}(\underline \sa)= \tilde \sigma^{N_\sR(\underline \sa)}(\underline \sa)$. 

The map $\tilde \sigma^\sR$ is semi conjugated, via $x\in \mathcal R \mapsto \underline \sa(x)\in \sR$, with the  first combinatorial return map $f^{\mathcal R}$ of $x\in \mathcal R$ into $\mathcal R$, with $f^{\mathcal R}(x)= f^{N_\sR(\underline \sa(x))}(x)$. Then we put:
\[  R:=\bigcap_{n\ge 0} (f^{\mathcal R})^n(\mathcal R)\; .\]

In Proposition \ref{EgaliterdesR}, we state that  the orbits of $R$ and $\mathcal R$ support the same invariant probability measures.  In Proposition \ref{lift}, we prove that the first return map of $R$ into  $R$ is equal to the first combinatorial return map $f^{\mathcal R}$. 

The latter proposition is new and crucial since it implies that $R$ is in bijection with the inverse limit $\overleftarrow \sR$ of $\sR$ for the $\tilde \sigma^{\sR}$-dynamics. This enables a precise combinatorics study of the invariant measures in the orbit of $\mathcal R$. 

 In subsection  \ref{thesubsectionlambda}, we push forward the set $R$ to define a set $\Lambda$ which supports a structure of Young tower. This push forward corresponds to one iteration of the shift $\tilde \sigma$ of $\sA^\N$.  Thus $\Lambda$ is still bijectively encoded via a map $\overleftarrow \sb$  by a subset $\sL= \tilde \sigma(\overleftarrow \sR)\subset \sA^{\Z}$ (Prop. \ref{Bij2}). 
 

 In Proposition \ref{EgaliteracLambda}, we prove that every invariant ergodic probability measure is either supported by the orbit of $\{A,A'\}\cup \mathcal E\cup K_\square$, either it is supported by the orbit of $\Lambda$.

 Splitting the set $\sL:=\sL^u\cdot \sL^s$, with $\sL^s\subset \sA^\N$ and $\sL^u\subset \sA^{\mathbb Z^-}$, we can define canonical stable and unstable manifolds. 
 
More precisely, to every $\underline \sb\in \sL^s$ is associated a long stable curve $\gamma^s(\underline \sb)$ which is $(1/\log \; b)^k$  by $f^k$, for all  $k\ge 0$ by Claim \ref{gammas}.
 
Also for every $\overline \sb\in \sL^u$ is associated a long unstable curve $\gamma^u(\overline \sb)$ which is $m^{-k}$ contracted by $f^{-k}$, for all $k\ge 0$ by Claims \ref{gammau}, for a uniform $m>1$. 

The set of words $\sS:=\{\sa_1(z)\cdots \sa_{N_\sR(\underline \sa)}:\; \underline \sa=(\sa_i)_{i\ge 0}\in \sR\}$ is countable and is used to encode $\Lambda$. In Proposition \ref{Markov}, we prove that $\sL$ is equal to $\sS^\mathbb Z$ without the stable set of a fixed point:
 \[\sL=  \sS^\mathbb Z\setminus \sA^{(\Z)}\cdot (\ss_-)^\N\;, \; \quad \text{for }\ss_-\in \sA .\]

This implies that $\Lambda$ has a Markov partition given by the countable alphabet $\sS$.  In particular:
\[\bigsqcup_{\sL^u} \gamma^s \cap \bigsqcup_{\sL^u} \gamma^u=\Lambda\]

The sets $\Lambda^s_\sg:= \cup_{\underline \sa\in\sL^s_\sg} \gamma^s(\underline \sa)\cap \Lambda=i(\sL^u\cdot \sL^s_\sg)$, among  $\sg\in \sS$, defines a Markov partition  of $\Lambda$, and are sent by $f^\Lambda$ onto respectively $\Lambda^u_\sg:= \cup_{\overline  \sa\in\sL^u_\sg} \gamma^u(\underline \sa)\cap \Lambda$.
  
We show then in Claims \ref{Y1}, \ref{Y2}, \ref{Y5} that this Markov partition satisfies the Young tower properties $(Y_1)-(Y_2)-(Y_3)$ and $(Y_5)$ of \cite[\textsection 6]{Pe10}.

In Proposition \ref{propSPR}, we show that the cardinality of pieces of the Markov partition $(\sL_\sg^s)_{\sg\in \sS}$ with induced time equal to $m\ge 1$ is at most $2e^{2 m/\sqrt M}$.
 As the dynamics on $\cup_{k\ge 0} f^k (\Lambda)$ has entropy close to be $\ge \log 2$, we deduce that the Young tower induces a conjugacy between $f|\cup_{k\ge 0}f^k(\Lambda)$ and a strongly positive recurrent Markovian (mixing) shift $\sigma \colon \Omega_G\to \Omega_G$ up to the stable set of a periodic point $\hat A$.
 
With  $\Omega_G'= \Omega_G\setminus W^s(\hat A)$, the conjugacy $i\colon \Omega_G'\to \cup_{k\ge 0}f^k(\Lambda)$ is shown to be Hölder continuous in Claim \ref{Holder}. In Claim \ref{homeo} we show that $i$ is a homeomorphisms.

 By \cite{BBG06}, a strongly positive recurrent mixing Markov shift is finitarily Bernoulli (see Prop. 2.3 of \cite{BBG06} and \cite{Ru82}). Actually, as $W^s(\hat A)$ has measure zero for the maximal entropy measure, using the fact that $i$ is a homeomorphism, it comes that $f|\cup_m f^m(\Lambda)$ is finitarily Bernoulli. 
 
Also the maximal entropy measure of $f|\cup_m f^m(\Lambda)$ exists, is unique, exponentially mixing, and it satisfies the central limit theorem, by Hölder continuity of $i$ and the following:
\begin{theo}[Cyr-Sarig, Thm. 1.1-2.1 \cite{sarig2}]\label{cyr-sarig}
Let $\Omega_G$ be a topologically mixing countable Markov chain which is strongly positive recurrent and with finite topological entropy.
Then there exists a unique maximal entropy probability; this measure satisfies the central limit theorem and it is exponentially mixing.\end{theo}

 Moreover we get that the periodic points of $f|\cup_{m}f^m(\Lambda)$ are equidistributed w.r.t. the maximal entropy measure from the following classical result:
\begin{theo}[Thm D, \cite{V-J67}]
   If $\sigma$ is mixing, strongly positive recurrent, Markov shift then the following converges weekly to the maximal entropy measure, as $p\to \infty$:
   \[\frac{1}{\text{Card }Fix\,\sigma^p}\sum_{x\in Fix\,\sigma^p} \delta_{x}.\]
   Moreover, $\frac1p\log (\text{Card }Fix\,\sigma^p)$ converges to the topological entropy of $\sigma$.
\end{theo}

This implies that the restriction of $f$ to $\cup_{n\ge 0} f^n(\Lambda)$ satisfies the conclusion of Theorem \ref{Main}.

This section finishes by proving the first return time property (Propositions \ref{EgaliterdesR} and \ref{lift}) by using a combinatorial argument using basically the formalisms of Puzzle algebra (combinatorial division). It uses also an argument based on Pesin theory whose proof is postponed to Appendix \ref{prinvLyapExp}.
\bigskip

\paragraph{In Section \ref{exception},} we achieve the proof of Theorem \ref{Main}.
By Proposition \ref{EgaliteracLambda}, it suffices to prove that $\{A,A'\}\cup \cup_n f^n(\mathcal E)\cup \hat K_\square$ supports only invariant ergodic probability with negligible entropy w.r.t. the entropy of $f$, and that the number  of periodic points therein is negligible w.r.t. those in $\cup_n f^n(\Lambda)$. 

In other words, we prove the following to achieve the proof of Theorem \ref{Main}.
\begin{itemize} 
\item The ergodic probability measures of $f$ which are not contained in $\cup_{n\ge 0} f^n(R)$, and so contained in $\hat K_\square$, $\{A,A'\}$ or $\cup_{n\ge 0} f^n(\mathcal E)$ have small entropy. 
\item The number of fixed points of $f^n|\hat K_\square\cup \cup_{n\ge 0} f^n(\mathcal E)\cup \{A,A'\}$ is negligible w.r.t. the number of fixed points of $f^n|\cup_{n\ge 0} f^n(\Lambda)$. 
\end{itemize}

The bounds for $\hat K_\square$ are easily computed in Proposition \ref{hKsquare}, since $\hat K_\square$ is a mere uniformly hyperbolic horseshoe. The entropy of the measures in $\cup_{n\ge 0} f^n(\mathcal E)$ are bounded by using Ledrappier-Young entropy formula after we give an upper bound on the unstable Hausdorff dimension of its hyperbolic measure in Proposition \ref{HDexceptionel}.
The number of periodic points is bounded from above by using combinatorial tools introduced (such as the division $/$) in the latter section.

\paragraph{In Appendix \ref{prinvLyapExp},} we prove Propositions \ref{eventually regu} and \ref{EgaliterdesR} by using the fact that almost every point of an invariant probability measure has well defined Lyapunov exponents. 

\paragraph{In Appendix \ref{sectionpreuvegeoYg},} we prove the statements relative to the geometry of the partition and 
of the long stable leaves involved. This is done by looking at the expansion and the contraction of $(Df^k)_k$ at their points and then by using classical arguments of \cite{BC2, YW, berhen}. 


At the end, {\bf an index gives the notations and definitions}. 

\paragraph{Open questions}
This manuscript implies that by Theorem 3.1 of \cite{Pe10}, a strongly regular Hénon like map has a unique equilibrium state for many potentials. It is natural  to ask:
\begin{ques}[Pesin-Senti-Zhang] Does every strongly regular map enjoy a unique equilibrium state for potentials of the form $s\cdot \log \big|\det\, Tf {|W^u}\big|$?
\end{ques} 
In this work, we answer positively this question for $s=0$. To get other  values of $t$, it would suffice to extract from the Young tower $(\Lambda, (\Lambda_\sg)_\sg)$ another tower with similar properties and which satisfies moroever the distortion bound $(Y_4)$ of \cite{Pe10}. This seems possible by using Prop. 2.9 of \cite{berhen}.  

Another natural question is:
\begin{ques} What is the Hausdorff dimension of the Hénon attractor? 
\end{ques}
It is easy to show that the Lebesgue measure of the attractor is zero. We except that the dimension should be close to 1 for $b$ small. From this work, it remains basically to study the set of infinitely irregular points in the attractor.

\section{Strong regularity}
\label{rappeldesdef}
In this section we recall Yoccoz' proof of Jakobson's Theorem, and how it has been generalized in \cite{berhen} to prove Benedicks-Carleson's Theorem. 

\subsection{Strongly regular quadratic maps}

For $a$ greater but close to $-2$, the quadratic map $P\colon  x\mapsto x^2+a$ has two fixed points $-1\approx A_0< A_0'\approx 2$ which are hyperbolic. The segment $[-A_0', A_0']$ is sent into itself by $P$, and its boundary bounds the basin of infinity. All the points of $(-A_0',A_0')$ are sent by an iterate of $P_a$ into $\R_\se:=[A_0,- A_0]$.

Yoccoz' definition of strongly regular maps is based on the position of the critical value $a$ with respect to the preimages of $A_0$. To formalize this, he used his concept of puzzle pieces.

\subsubsection{Puzzle pieces}
\begin{defi}
A puzzle piece $\sa= (\R_\sa, n_\sa)$ is the pair of a segment $\R_\sa$ of $\R_\se$ and an integer $n_\sa$, so that $P^{n_\sa}|\R_\sa$ is a bijection from $\R_\sa$ onto $\R_\se:=[A_0,- A_0]$. 
\end{defi}
For instance $\se:= \{\R_\se, 0\}$ is a puzzle piece, called \emph{neutral}. 

To define the simple puzzle pieces, let us denote by $M$ the minimal integer such that $P^M(a)$ belongs to $[A_0,-A_0]$; $M$ is large since $a>-2$ is close to $-A_0'\approx -2$.

For $i\ge 0$, let $A_i:= -(P|\R^+)^{-i}(-A_0)$. Note that $(A_i)_{i\ge 0}$ is decreasing and converges to $-A_0'$. Also $[A_{i+1}, A_{i}]$ is sent bijectively by $P^{i+1}_a$ onto $\R_\se$. The same holds for $[ -A_i, -A_{i+1}]$.

By definition of $M$, the critical value $a$ belongs to  $[ A_{M}, A_{M-1}]$.
 Hence for $2\le i\le M$, there is a segment $\R_{\ss^{i}_-}\subset \R^-$ and a segment $\R_{\ss^{i}_+}\subset \R^+$ both sent bijectively by $P$ onto $[- A_{i-1},-A_{i-2}]$.

\begin{defi}[Simple puzzle piece]
The pairs of the form $(\R_{\ss^{i}_\pm}, i)$ for $2\le i\le M$ are puzzle pieces called \emph{simple}. There are $2(M-1)$ such pairs. The set of simple puzzle pieces is denoted by $\sY_0=\{\ss^i_\pm ; 2\le i\le M\}$.  \label{Psimple}\index{$\sY_0$}
\end{defi}

Puzzle pieces enjoy two fundamental properties:
\begin{enumerate}
\item Two puzzle pieces $\sa$ and $\sb$ are nested or disjoint: 
\[\R_\sa\subset \R_\sb\text{ or } \R_\sb\subset \R_\sa \text{ or }  int\; \R_\sb\cap int\; \R_\sb=\varnothing\;.\]
\item For every puzzle piece $\sa$, for every perturbation of the dynamics, the hyperbolic continuities of the relevant preimages of the fixed point $A_0$ define a puzzle piece for the perturbation.
\end{enumerate}
\subsubsection{Building puzzle pieces}

The first operation is the so-called \emph{simple product} $\star$:
\begin{defi}[$\star$-product]
Let $\sa = (\R_\sa,n_\sa)$ and $\sb = (\R_\sb, n_\sb)$ be two puzzle pieces so that $\R_\sb \subset \R_\se$. 
Then, the puzzle piece $\sa\star \sb$ with segment $\R_{\sa \star \sb} =  (P^{n_\sa}|\R_\sa)^{-1}(\R_\sb)\subset \R_\sa$ and integer $n_{\sa\star \sb} = n_\sa +n_\sb$ is a puzzle piece : the map $P^{n_{\sa\star \sb}}$ sends bijectively $\R_{\sa\star \sb}$ onto $\R_\se$. 
\end{defi}\index{Simple product $\star$}

Note that the simple operation $\star$ is associative. Indeed for any puzzle pieces $\sa,\sb,\sc$, it holds:
\[\sa\star (\sb\star \sc)= (\sa\star \sb)\star \sc=:\sa\star \sb\star \sc\; .\]

We need another operation to construct pieces in the closure  $\R_\square$ of the complement of the simple pieces union in $\R_\se$:
 \[\R_\square:= cl(\R_\se\setminus \cup_{\sa\in \sY_0} \R_\sa)=P^{-1}_a([ -A_{M}, -A_{M-1}])\]\index{$\R_\square$}
 
This is a neighborhood of $0$ of length dominated by $2^{-M}$ when $a$ is close to $-2$.

This second operation is the so-called \emph{parabolic product} $\square$. 

\begin{defi}[$\square$-product]\index{Parabolic product $\square$}
Let $\sa$ and $\sb$ be two puzzle pieces so that $\R_\sb\subsetneq \R_\sa$ and so that $\R_\sb$ intersects $P^{M+1}(\R_\square)$ at a non trivial segment.
We notice that $P^{M+1}|\R_\square$ has two inverse branches, one $g_+$ with image into $\R^+$ and the other $g_-$ with image into $\R^-$. 

We define the parabolic pieces:
\[\square_+(\sa-\sb):=\{g_+(cl(\R_\sa\setminus  \R_\sb)), M+1+n_\sa\}
\quad \text{and}\quad \square_-(\sa-\sb):=\{g_-(cl(\R_\sa\setminus  \R_\sb)),M+1+ n_\sa\}\]
\end{defi}
A parabolic piece $\sp= \square_\pm(\sa-\sb)$ is never a puzzle piece. Indeed, with:
\[\{\R_\sp, n_\sp\}:= \{g_\pm(\R_\sa- int\,  \R_\sb), M+1+n_\sa\},\]
the segment $\R_\sp$ is sent by $P^{n_\sp}$ onto a connected component of $cl(\R_\se\setminus P^{n_\sa}(\R_\sb))\subsetneq\R_\se$.

We notice that the $\star$-product extends canonically to the set of parabolic and puzzle pieces: we can make simple product between those pieces. 

\subsubsection{Yoccoz' definition of strong regularity}

The main ingredient of Yoccoz' definition, is to ask for the existence of a sequence of puzzle pieces $\sc = (\sa_i)_{i\ge 1}$ so that with $\sc_k= \sa_1\star \cdots\star \sa_k$ the first return $P^M(a)$ belongs to a nested intersection of puzzle pieces $\cap_{k\ge 1} \R_{\sc_k}$:
\begin{equation}\tag{$SR_1$}
P^{M+1}(0)\in \bigcap_{k\ge 1} \R_{\sc_k}\; ,
\end{equation}\index{SR$_1$}
and so that $(\sa_i)_{i\ge 1}$ satisfies: 

\begin{equation}\tag{$\star$}
\sum_{j\le i\; \sa_j\notin \sY_0} n_{\sa_j} \le e^{-\sqrt M} \sum_{j\le i-1} n_{\sa_j},\quad \forall j\le i.
 \end{equation}
Moreover Yoccoz asked that there is a neighborhood $\hat \R_\se$ of $\R_\se$ so that every involved segment  $\R_{\sa_i}$ has a neighborhood $\hat \R_{\sa_i}$ which is sent bijectively by $P^{n_{\sa_i}}$ onto $\hat \R_\se$. The negativity of the Schwarzian derivative of $P$ gives then a distortion bound for $P^{n_{\sa_i}}|\R_{\sa_i}$. 

Such a hypothesis is assumed in particular for all simple pieces in $\sY_0$. This implies the existence of $c>0$ such that every $x\in \R_\sa$, $\sa\in \sY_0$, it holds that:
 \[\|\partial_x P^{n_\sa}\|\ge e^{cn_\sa}.\] 
Then Equation ($\star$) and the distortion bound implies:
\begin{equation}\tag{$\mathcal {CE}$}
\liminf_{n\to \infty} \frac1n\log \|\partial_xP^{n}(a)\| \ge c^-:=(1-e^{-\sqrt M} )c\;.\end{equation}
In particular, strongly regular unimodal map satisfies the Collet-Eckmann condition.

\subsubsection{Alternative definition of strong regularity}
The existence of an interval $\hat \R_{\sa_i}$ extending $\R_{\sa_i}$ is replaced by two other conditions: $h$-times and $(\diamondsuit )$. 

%

\begin{defi} A puzzle piece or a parabolic piece $\sa=(\R_\sa,n_\sa)$ is \emph{hyperbolic} if it satisfies the following condition:
\begin{equation}\tag{$h-times$} 
\forall z\in \R_\sa\text{ and }l\le n_\sa: \\
|\partial_x P ^{n_\sa}(z)| \ge e^{\frac{c}{3} (n_\sa-l)} |\partial_x P^l(z)|\; ,
\end{equation}\index{h-times}
with $c:= \log 2/2$ \index{$c$}.
\end{defi}

It is straight forward to see that a $\star$-product of hyperbolic pieces is hyperbolic.

Suppose that the map $P$  satisfies $(SR_1)$ with $(\sc_k)_{k\ge 1}$. We define the following countable set of symbols  $\sA:= \sY_0\sqcup\{\square_\delta(\sc_k-\sc_{k+1}): \; k\ge 0,\; \delta\in \{+,-\}\}$.\index{$\sA$}
\begin{prop}
Every puzzle piece $\sa$ is a simple product of pieces  in $\sA$.
\end{prop}
\begin{proof}
We proceed by induction on $\sa$. As the puzzle pieces  are nested or disjoint, either $\R_\sa$ is included in a simple piece $\R_\ss$ either it is included in $\R_\square$.

 In the first case,
$(P^{n_\ss}(R_\sa), n_\sa-n_\ss)$ is still a puzzle piece and by induction it is a product of parabolic and simple piece $\sa_1\star \cdots \star \sa_k$. Hence $\sa= \ss\star \sa_1\star \cdots \star \sa_k$. 

In the second case, $\R_\sa$ is either included in $\R^-$ or in $\R^+$. Also its first return in $\R_\se$ is $f^{M+1}(\R_\sa)$. Note that
$(f^{M+1}(\R_\sa),n_\sa-M-1)$  is still a puzzle piece. Let $k\ge 0$ be the greater integer so that $f^{M+1}(\R_\sa)$ is included into $\R_{\sc_k}$. Then $\R_\sa$ is included in $\R_{\square_\pm (\sc_k-\sc_{k+1})}$. Also its image by $f^{n_{ \square_\pm (\sc_k-\sc_{k+1})}}$ is also a puzzle piece and so we can use the induction hypothesis as above to achieve the proof. \end{proof}
\begin{defi}\index{Prime}
A puzzle piece $ \sa$ is \emph{prime} if it is a simple puzzle piece or if there exist parabolic pieces $\sp_1,\dots, \sp_k\in \sA$ and a simple puzzle piece $\ss\in \sY_0$ so that:
\[\sa= \sp_1\star \sp_2\star\cdots\star \sp_k\star \ss.\]
\end{defi}

Hence to obtain the hyperbolicity of any puzzle piece, it suffices to give a  combinatorial condition on the critical orbit which implies the hyperbolicity of all the simple pieces and all the parabolic pieces in $\sA$.
 This is the case if $P$ satisfies $(SR_1)$ with a sequence $\sc=(\sa_i)_i$ so that $P^{M+n_{\sc_i}}(a)\in \R_\se$ does not belong to an exponentially small neighborhood of $\partial \R_\se=\{A_0,-A_0\}$.  
 
To make the notation less cluttered, we denote $\ss^2_-$ and $\ss^2_+$ by respectively $\ss_-$ and $\ss_+$. These two puzzle pieces have their segment which is a neighborhood of respectively $A_0$ and $-A_0$ in $\R_\se$.
  
Likewise, the segments of the pieces 
$\ss_-^{\star k}:=\ss_-\star  \cdots \star \ss_-$ and  $\ss_+^{\star k}:=\ss_+\star \ss_-^{\star k}$ are neighborhoods of respectively $A_0$ and $-A_0$ in $\R_\se$.

The condition we ask is the following:
\begin{equation}\tag{$\diamondsuit $}
P^{M+1+n_{\sc_i}}(0)\notin \R_{s_-^{\star \aleph(i)}}\sqcup \R_{s_+^{\star \aleph(i)}}\;,
\end{equation}\index{$\diamondsuit $}
with $\aleph(0):=\left[\frac{\log M}{6c^+}\right]$ and for $i>0$, 
$\aleph(i):=\left[\frac{c}{6c^+}(i+M)\right]$, where $c^+:=\log 5$.\index{$aleph$@$\aleph$}\index{$c^+$}
Such a condition implies that every parabolic pieces is hyperbolic (see Prop. \ref{Proph} below).

The condition $(\diamondsuit)$ does hold if the sequence $\sc=(\sa_i)$ involved in $(SR_1)$ is \emph{common}:

\begin{defi}\index{Common}
A \emph{common sequence} $\sc=(\sa_i)_i$ is a sequence of puzzle pieces which  satisfies $(\star)$ and so that for every $i\ge 0$:
\begin{itemize}
\item the piece $\sa_i$ is either in $\sY_0$ or $\R_{\sa_i}$ is included in $\R_\square$.
\item  $\sa_i\star \cdots\star \sa_{i+\aleph (i)}\notin \{s_-^{\star \aleph(i)},s_+^{\star \aleph(i)}\}$.
\end{itemize}
\end{defi}
%


\begin{defi} The quadratic map $P$ is \emph{strongly regular} if there exists a common sequence $\sc=(\sa_i)_{i\ge 1}$ so that:
\begin{equation}\tag{$SR_1$}
P^{M+1}(0)\in \R_{\sc_k},\quad \text{with } \sc_k=\sa_1\star \cdots \star \sa_k.\end{equation}
\begin{equation}\tag{$SR_2$}
\text{Every puzzle piece $\sa_k$ is prime.}\end{equation}
\end{defi}

As announced, we have:
\begin{prop}[Prop 1.3 and 4.1 \cite{berhen}]\label{Proph}
If $P$ is strongly regular, then every simple piece and parabolic piece is hyperbolic. 
\end{prop}
As every puzzle pieces is a $\star$-product of parabolic and simple pieces, it comes:
\begin{coro} 
If $P$ is strongly regular, then every puzzle piece is hyperbolic.
\end{coro}
As for Yoccoz definition, this implies:
\begin{coro}\label{CEP} 
If $P$ is strongly regular, then it satisfies the Collet-Eckmann  Condition $(\mathcal {CE})$. 
\end{coro}

\subsection{Strongly regular Hénon like endomorphisms}

We now consider a $C^2$-map $f:=f_{a\, B} \colon (x,y)\mapsto (P(x)+y, 0)+ B_a(x,y)$ satisfying that:
\begin{itemize} 
\item the parameter $a>-2$ is close to $-2$, so that the first return time $M$ of $a$ by $P$ in $\R_\se$ is large.
\item A real number $b>0$ small w.r.t. $|a+2|$ (and is even small w.r.t. $e^{-e^{e^M}}$), bound the $C^0$-norm of $\det\; Df$ and the $C^2$-norm of $(x,y,a)\in [-3,3]^2\times \R\mapsto  B_a(x,y)$.\index{$b$}
\end{itemize}
Put $\theta:= |\log\, b|^{-1}$. \index{$\theta$} We notice that $\theta$ is small w.r.t. $e^{-e^M}$.
 
We observe that $f$ is $b$-close to $\hat P:=(x,y)\mapsto (x^2+a+y,0)$ which preserves the line $\R\times \{0\}$ and whose restriction therein is equal to the quadratic map $P$.  Hence, for $b$ small,  the fixed point $(A_0,0)$ for $\hat P$ persists as a fixed point $A$ of $f$. \index{$A$}

The strong regularity condition is related to the topology of the homoclinic tangle of $W^s(A; f)\cup  W^u(A; f)$.

To formalize this we generalize the definition of puzzle pieces for \emph{flat curves} that we will define in the sequel.  

First let us notice that the (compact) local stable manifold $\{(x,y)\in\R\times [-1,\infty):  x^2+a= A_0\}$ persists as a local stable manifold $W^s_{loc}(A; f)$ for $f_{a\;B}$. With the line $\{y=2\theta\}$ and the line $\{y=-2\theta\}$, the local stable manifold  $W^s_{loc}(A; f)$ bounds a compact diffeomorphic to a filled square denoted by $Y_\se$ (see fig. \ref{geometricmodele}).

Let us denote by $\partial^s Y_\se:= Y_\se\cap W^s_{loc}(A; f)$ and 
$\partial^u Y_\se:= Y_\se\cap \{y=\pm 2\theta\} $. 

Both sets consists of two connected curves whose union is $\partial Y_\se$.

\begin{defi}[flat stretched curve]
A   curve $S\subset Y_\se$ is \emph{flat} if it is the graph of a $C^{1+Lip}$-function $\rho$ over an interval $I\subset \R$, with $C^{1+Lip}$-norm at most\footnote{Actually, in \cite{berhen}, we ask the flat stretched curves to be the image by a certain map $y_\se$ of a graph of a function satisfying such bounds. Nevertheless the map $y_\se$ has its $C^{1+Lip}$-norm bounded and its inverse has its $C^{1+Lip}$-norm bounded by $\theta^{-1}$. Moreover all bounds on the graph transforms will have sufficiently room so that this does not change the statement of the propositions involving the flat curves.}   $\theta$.
\[\|\rho\|_{C^0}\le\theta,\quad \|D\rho\|_{C^0}\le\theta,
\quad \|Lip(D\rho)\|_{C^0}\le\theta\;.\]
 
The flat curve $S$ is \emph{stretched} if it is included in $Y_\se$ and satisfies that $\partial S\subset \partial^s Y_\se$. 
\end{defi}\index{Flat stretched curve}

For instance  $\se(S):=\{S, 0\}$ is a puzzle piece called \emph{neutral}. 

\subsubsection{Puzzle pieces}
A puzzle piece is always associated to a flat stretched curve $S$. 
\begin{defi} 
A \emph{ puzzle piece} \index{Puzzle piece}
 $\sa(S)$ of $S$ is the data of:
\begin{itemize}
\item an integer $n_\sa$ called the \emph{order of a puzzle piece} \index{Puzzle piece2@Order of a puzzle piece}
 of $\sa(S)$,
\item a segment $S_\sa$ of $S$ sent by $f^{n_\sa}$ to a flat stretched curve $S^\sa$.\end{itemize} 
A piece $\sa(S)=(S_\sa,n_\sa)$ is \emph{hyperbolic} if the following conditions hold:
\paragraph{\emph{h-times}}\index{h-times} For every $z\in S_\sa$, $w \in T_z S_\sa$ and every $l\le n_\sa$: $\|D_{z}f^{n_\sa}(w)\|\ge e^{\frac{c}{3} (n_\sa-l)}\cdot \|D_zf^l({w})\|$.
\end{defi}\index{h-times} \index{Hyperbolic piece}
We recall that $c=\log\, 2/2$.

In order to define the simple puzzle piece, we assume that $P^{M+1}(0)$ does not belong to  $\R_{s_-^{\star \aleph(0)}}\sqcup 
\R_{s_+^{\star \aleph(0)}}$. Hence the following $P$-forward invariant compact set:
$$K:=\{A_0',-A_0'\}\cup \bigcup_{i\ge 0} \{A_i, -A_i \}\cup \bigcup_{\ss\in \sY_0} \partial \R_s$$ is at bounded distance from $0$ and so is uniformly expanding for $P$. 

We remark that the set $K\times\{0\}$ is uniformly hyperbolic for $\hat P$. For $z_0=(x_0,0)\in K\times \{0\}$, 
the component $W^s_{loc}(z_0;\hat P)$ containing $z_0$ of  
\[\{(x,y)\in\R\times [-1,\infty)\colon x^2+y=x_0^2\}\]
is a local stable manifold of $z_0$. We notice that $W^s_{loc}(z_0, \hat P)$ is an arc of parabola.
%
%

By hyperbolic continuity, for $b$ sufficiently small, the family of curves $(W^s_{loc}(z_0))_{z_0\in K\times \{0\}}$ persists as a family $(W^s_{loc}(z_0; f))_{x_0\in K\times \{0\}}$ so that:
\begin{equation}\label{lamination} f (W^s_{loc}(z_0; f))\subset W^s_{loc}(\hat P(z_0); f)\; .\end{equation}



Also for every $\sa\in \sY_0\cup\{e, \square\}$, the endpoints $(x_-,x_+)$ of $\R_\sa$ belong to $K$, and the curves $W^s((x_\pm,0); f)$ are sufficiently close to 
$W^s((x_\pm,0); \hat P)$ so that they stretch across the strip $\R\times [-2\theta,2\theta]$ to bound a compact set $Y_\sa$ close to  $S_\sa\times \{0\}$  and diffeomorphic to a filled square (see Fig. \ref{geometricmodele}).  The set $Y_\sa$ is called the \emph{box}\footnote{Also called simple extension in \cite{berhen}.} associate to $\sa$.\index{Box}\index{$Y_\square$}\index{$Y_\se$}

\begin{figure}[h]
    \centering
        \includegraphics{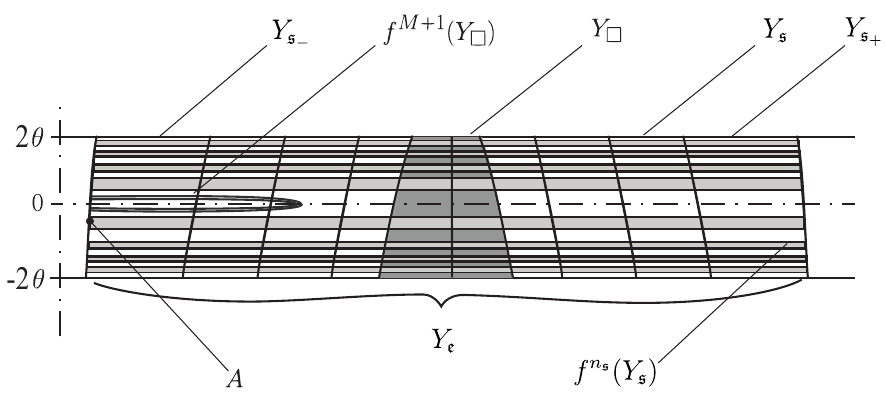}
    \caption{Geometric model for some parameters of the Hénon map. }
\label{geometricmodele}
\end{figure}

Let $\partial^uY_\sa:= Y_\sa\cap \{y=\pm 2\theta\}$ and let $\partial^s Y_\sa :=  Y_\sa \cap \cup_{\pm} W^s((x_\pm,0)  ; f)$. 

We notice that by (\ref{lamination}), it holds that $f^{n_\sa }( \partial^s Y_\sa) \subset  \partial^s Y_\se$, as depicted by Fig. \ref{geometricmodele}. \\

\begin{defi}[Simple pieces]
For every flat stretched curve $S$, for every $\ss\in \sY_0$, 
let $S_\ss:=S\cap Y_\ss$ and let $\ss(S):= \{ S_\ss,n_\ss\}$. Note that $S^\ss=f^{n_\ss}(S_\ss)$. \label{henonsimple}
\end{defi}

In \cite{berhen} Expl. 2.2, we show:
\begin{exem}[Simple pieces]
For any flat stretched curve $S$, each pair $\sa(S):= \big\{Y_{\sa}\cap S, n_\sa\big\}$, for $\sa\in \sY_0$ is a hyperbolic puzzle piece called \emph{simple}.
\end{exem}
\begin{exem}[Curves $S^{{t\! t}}$ and $(S^{t})_{t\in T_0^{\mathbb Z^-}}$]\label{T0}
The map $S\mapsto S^{s_-}$ from the space of flat stretched curves into itself is well defined and $C^1$-contracting in the space of flat stretched curves. 
With ${t\! t}:=(s_-)_{i\le 0}\in \sY_0^{\mathbb Z^-}$, we denote by $S^{t\!t}$ its fixed point. It is a half local unstable manifold of $A$. 
 We have  also $S^{{t\! t}} =\{ z_0\in Y_\se:\; \exists (z_i)_{i\le -1}\in Y_{s_-}^{\mathbb Z^-}, \; z_{i+1}= f^2(z_i)\}$.

Similarly for $t= (\sa_i)_{i\le-1}\in \sY_0^{\mathbb Z^-}$, the set:
\[S^{t} =\{ z_0\in Y_\se:\; \exists (z_i)_{i\le -1}\in \prod_i Y_{\sa_i}, \; z_{i+1}= f^{n_{\sa_i}}(z_i)\},\]
is a flat stretched curve. We put $T_0:=\sY_0^{\mathbb Z^-}$. They define the family of curves $(S^t)_{t\in T_0}$. 
 \end{exem}
 
Similarly, we can define the operation $\star$ on puzzle pieces of flat stretched curves.\index{$T_0$}
\subsubsection{Operation $\star $ on puzzle pieces}
\begin{defi}[{Operation $\star$ on puzzle pieces}]
Let $\sa(S):= \{S_\sa, n_\sa\}$ and $\sb(S^\sa)= \{S^\sa_\sb, n_\sb\}$ be two puzzle pieces of $S$ and $S^\sa:= f^{n_\sa}(S_\sa)$ respectively. We define the puzzle piece of $S$: 
\[\sa\star\sb(S):=\{f^{-n_\sa}(S^\sa_\sb)\cap S_\sa, n_\sa+n_\sb\}.\]
Indeed the map $f^{n_{\sa\star \sb}}|S_{\sa\star \sb}$ is a bijection onto $S^{\sa\star \sb}:= f^{n_{\sa\star \sb}}(S_{\sa\star \sb})$. 
\end{defi}
The pair of puzzle pieces $(\sa(S),\sb(S^\sa))$ is called \emph{suitable}. More generally, a sequence $(\sa^i(S^i))_{1\le i< k}$, for $k\in \N\cup\{\infty\}$  is called suitable if the pair of any two consecutive puzzle pieces is  suitable. 
We can now generalize condition $(\star)$ of Yoccoz' strong regularity definition. \index{Suitable}
\begin{defi}[Common sequence]
For $N\in [1,\infty]$, a \emph{common sequence} \index{Common sequence} $\sc$ is a suitable  sequence of  hyperbolic puzzle pieces $\sc := (\sa_i(S^i ))_{i= 1}^{N-1}$ from $S^1:=S^{t\!t}$ which satisfies the following properties:
\begin{equation}\tag{$\star$}
\sum_{j\le i\; \sa_j\notin \sY_0} n_{\sa_j} \le e^{-\sqrt M} \sum_{j\le i-1} n_{\sa_j},\quad i< N.
 \end{equation}
Moreover every pieces $\sa_i(S^i )$ is either simple or included in $Y_\square$, and for every $i\ge 0$, 
\begin{equation}\tag{$\diamondsuit $}
\sa_{i+1}\star\cdots\star \sa_{i+\aleph(i)}\notin \{  s_-^{\star \aleph(i)}, s_+^{\star \aleph(i)}\}
\end{equation}
\end{defi}

The product $\sc_i := \sa_1\star \sa_2\star \cdots \star\sa_{i-1}\star \sa_i$ 
is called a \emph{common product of depth $i$} \index{Common product of depth $i$}
 and it defines a pair $\sc_i(S^{t\!t})=:\{ S^{t\!t}_{\sc_i},n_{\sc_i}\}$ called a \emph{common piece}. \index{Common piece}

A \emph{common piece of depth $0$} is the pair equal to $\sc_0(S^{t\!t}):=\{ S^{t\!t}, 0\}= \se(S^{t\!t})$.

Not all the puzzle pieces have their endpoints with a nice local stable manifold. Nevertheless it is the case for the common piece:
\begin{prop}[\cite{berhen} Prop. 3.6 ]
Each endpoint $z_\pm$ of $S^{t\!t}_{\sc_i}$ has a local stable manifold $W^s_{loc} (z_\pm; f)$ which stretches across $Y_\se$ and is $\sqrt b $-$C^2$-close to an arc of curve of the form:
\[\{(x,y)\colon x^2+y = cst\}.\]
With the lines $\{y=\pm2 \theta\}$, this bounds a box of $Y_\se$ denoted by $Y_{\sc_i}$. Moreover, for every $z\in Y_{\sc_i}$, the vector $Df^{n_{\sc_i}} (1,0)$ is $\theta$-close to be horizontal and of norm at least $e^{c^- n_{\sc_i}}$, with:
\[c^-=c-\frac 1{\sqrt M}=\frac {\log 2}2-\frac 1{\sqrt M}\; .\]\index{$c^-$}
\end{prop}
We put $\partial^s Y_{ \sc_i}=  \cup_\pm W^s_{loc} (z_\pm; f)\cap Y_\se$ and $\partial^u Y_{ \sc_i} = \partial^u Y_\se \cap Y_{ \sc_i}$.

By the above Proposition, the width of $Y_{\sc_i}$ is smaller than $2e^{-c^- n_{\sc_i}}$ times the width of $Y_\se$ and so, if $N=\infty$, the following decreasing intersection:
\[W^s_{\sc} := \cap_{i\ge 0} Y_{ \sc_i}\;.\]
is a $C^{1+Lip}$-curve called \emph{common stable manifold}, which is $\sqrt b$-$C^{1+Lip}$-close to  an arc of a curve of the form:
\[\{(x,y)\colon x^2+y = cst\}.\]


\subsubsection{Tangency condition}

Every flat stretched curve $S$ intersects $Y_\square$ at a segment $S_\square=S\cap Y_\square$. This segment is sent by $f^{M+1}$ to
a curve $S^\square$ which is $C^2$-close to a folded curve $ \{(-Cst\cdot 4^M t^2+f^{M}_a(a), 0): t\in \R\}\cap Y_\se$.

The definition of strong regularity for Hénon-like maps supposes the existence of a family of curves $(S^t)_{t\in T^*}$ so that for each $t\in  T^*$, there exists a common sequence of puzzle pieces $\sc^t$ so that\index{$T^*$}

\paragraph{$(SR_1)$} $S^{t\square}= f^{M+1} (S^t_\square)$ is tangent to $W^s_{\sc^t}$.

As in dimension $1$, conditions are given on the puzzle pieces involved in the common sequences. In this two dimensional case, conditions are moreover given on the flat and stretched curves forming   $(S^t)_{t\in T^*}$.

\subsubsection{Parabolic operations from tangencies}
As in dimension $1$, if a flat stretched curve $S$ satisfies that $S^\square $ is tangent to a common stable manifold $W^s_{\sc^t}$, then we can define parabolic pieces.

Indeed, then for every $i$, 
$(f^{M+1}|S_\square)^{-1}cl(Y_{\sc_i}\setminus Y_{\sc_{i+1}})$ consists of zero or two segments. 

We denote by $S_{\square_-(\sc_i -\sc_{i+1})}$ the left hand side segment and by $S_{\square_+(\sc_i -\sc_{i+1})}$ the right hand side segment. 

Let $\sp$ be a symbol in $\{\square_+(\sc_i -\sc_{i+1}),\square_-(\sc_i -\sc_{i+1})\}$.
\begin{defi}[Symbolic identification]
The symbols $\square_+(\sc_i -\sc_{i+1})$ and $\square_-(\sc_i -\sc_{i+1})$ depend only on $Y_{\sc_i}$ and $Y_{\sc_{i+1}}$.

In particular if for $t\not=t'$ it holds $Y_{\sc_i^t}=Y_{\sc_i^{t'}}$ and 
$Y_{\sc_{i+1}^t}=Y_{\sc_{i+1}^{t'}}$, then the following identifications are done $\square_+(\sc^t_i -\sc^t_{i+1})=\square_+(\sc^{t'}_i -\sc^{t'}_{i+1})$ and $\square_-(\sc^t_i -\sc^t_{i+1})=\square_-(\sc^{t'}_i -\sc^{t'}_{i+1})$.
\end{defi}

With $n_{\sp}=M+1+n_{\sc_i}$, the pair $\sp(S):=\{S_{\sp}, n_{\sp}\}$ is called a \emph{parabolic piece}.

This pair $\sp(S)$ cannot be a puzzle piece since the curve $f^{n_\sp}(S_\sp)$ is not stretched (like in the one dimensional model).
 
However the curve $f^{n_\sp}(S_\sp)$ can be extended to a flat stretched curve $S^\sp$ by an algorithm given by Prop. 4.8 and 5.1 in \cite{berhen}. In particular $S^\sp\supsetneq f^{n\sp} (S_{\sp})$.

\begin{defi}[Set of symbols $\sA$]
Let $f$ which satisfies $(SR_1)$ with the flat stretched curves $(S^t)_{t\in T^*}$ and the common sequences $(\sc^t)_{t\in   T^*}$.

Let 
\[\sA:= \sY_0\cup \bigcup_{t\in  T^*}\bigcup_{i\ge 0} \{\square_+ (\sc^t_i -\sc^t_{i+1}), \square_- (\sc^t_i -\sc^t_{i+1}) \}\;.\]
\end{defi}
The above union over $t\in T^*$ is not disjoint by the above remark.
As they are countably many puzzle pieces of $S^{t\! t}$, they are countably many common pieces $\sc_i$ and boxes $Y_{\sc_i}$. Thus $\sA$ is countable.
\begin{prop}[Prop. 1.7 and 4.1 of \cite{berhen}]\label{hypersA}
For every $t\in T^*$, every parabolic or simple piece  $\sa(S^t)$,  with $\sa\in \sA$, is hyperbolic.
\end{prop}

\begin{defi}[Suitable chain]
Let $(S^{i})_{i=1}^n$ be a family of flat stretched curves and let $(W^s_{\sc^i})_{i=1}^n$ be a family of  common stable manifolds so that $S^{i\square }$ is tangent to  $W^s_{\sc^i}$.
 
For each $i$ let $\sp_i$ be a symbol either in $\sY_0$ , either parabolic obtained from $\sc^i$ (that is of the form $\square_\pm (\sc_j^i -\sc_{j+1}^i)$). 

The chain of symbols $(\sp_i)_{i=1}^n$ is called suitable from $S^{1}$ if:
\begin{enumerate}
\item $S^{{i+1}}=S^{{i}\cdot \sp_i}$ for every $i<n$,
\item The segment of the pair $\sp_1(S^{1})\star \cdots \star \sp_n(S^{n})$ is not trivial (it has cardinality $>1$).
\end{enumerate}
The chain of symbols is complete if $\sp_n$ belongs to $\sY_0$, and incomplete otherwise. 
The chain of symbols $(\sp_i)_i$ is \emph{prime} if $\sp_i\notin \sY_0$ for $i<n$.\index{Prime} \index{Complete}
\end{defi}
A corollary of Proposition \ref{hypersA} is:
\begin{coro}
If $(\sp_i)_{i=1}^n$ is suitable, then $\sp_1(S^{t_1})\star \cdots \star \sp_n(S^{t_n})$ is a hyperbolic piece of $S^{t_1}$. 
\end{coro}

\subsubsection{Puzzle algebra and strong regularity definition}

In Example \ref{T0}, we defined for every $t\in T_0:= \sY_0^{\mathbb Z^-}$ a flat stretched curve $S^t$. 

In \cite{berhen}, for a set of parameters $a\in P_B$ of Lebesgue measure positive,  we show the existence of a family of curves $(S^t)_{t\in T^*}$ and a family of common sequences $C= (\sc^t)_{t\in T^*}$ which are linked in the following way by the tangency condition and parabolic/simple operations. 

\begin{enumerate}[$(SR_1)$]
\item $S^{t\square}= f^{M+1} (S^t\cap Y_\square)$ is tangent to $W^s_{\sc^t}$.

\[ \text{Put}\quad \sA:= \sY_0\cup \bigcup_{t\in  T^*,\quad i\ge 0} \{\square_+ (\sc^t_i -\sc^t_{i+1}), \square_- (\sc^t_i -\sc^t_{i+1}) \}\; \text{modulo the symbolic identification}.\]
\item For every $t\in T^*$, every puzzle piece $\sa_i(S^i)$ involved in $\sc^t=(\sa_i(S^i))_i$ is given by suitable, complete and prime chain of symbols $\underline \sa_i$ in $\sA^{(\N)}$.  
\item The set $T^*$ is the subset of $\sA^{\mathbb Z^-}$ defined by 
$$T^*= \{t\cdot \sp_{-n} \cdots \sp_{-1}:\; t\in T_0,\;  n\ge 0,  \; (\sp_i)_{1\le i\le n} \in \sA^n\text{ is a suitable chain from }S^{t}\}.$$ 
For $t^*= t\cdot \sp_{-n} \cdots \sp_{-1}\in T^*$, we put  $S^{t^*}= (\cdots (S^{t})^{ \sp_{-n}} \cdots )^{\sp_{-1}}$. 

\end{enumerate}
\begin{rema}
In ($SR_2$) the element $t\cdot \sp_{-n} \cdots \sp_{-1}$ is equal to the presequence $(\sa_i)_{i\le -1}\in \sA^{\mathbb Z^-}$ defined by
 $\sa_{-i}:=\sp_{-i}$ if $1\le i\le n$, and, with $t=(\ss_i)_{i\le -1}\in T^0$,
  $\sa_{-i}:\ss_{-i+n}$ if   $i\ge n+1$.
\end{rema}

\begin{defi}
A map $f$ so that there exists a family of flat stretched curves $(S^t)_{t\in T^*}$ and a family of common sequences $(\sc^t)_{t\in T^*}$ satisfying  $(SR_1-SR_2-SR_3)$ 
 is called \emph{strongly regular}. 
 \end{defi}
\begin{defi}\index{$\sG$}
Let $\sG$ be the set of finite segments of sequences in $T^*\subset Y_0^{\mathbb Z^-}\times \sA^{(\mathbb N)}$. It has a structure of pseudo-semi-group for two operations: $\star$ and $\square$. The triplet $(\sG, \star, \square)$ is called a \emph{Puzzle Algebra}.
\footnote{In \cite{berhen}, the presentation of strong regularity is different: The  set $T^*$ is presented as the disjoint union of the sets $T$ and $T^\square$, formed by the presequences $t\star \sp_{-n}\star \cdots\star \sp_{-1}\in T^*$ which finish by  respectively a simple piece or a parabolic piece.
This splits the family of curves $(S^t)_{t\in T*}$ into two subfamilies  $\Sigma= (S^t)_{t\in T}$ and $\Sigma^\square= (S^t)_{t\in T^\square}$. 
Furthermore, the set of prime puzzle pieces of a curve $S^t$, with $t\in T$, is denoted therein by $\mathcal Y(t)$. We define also $\mathcal Y:= \sqcup_{t\in T} \mathcal Y(t)$. The quadruplet $(\Sigma,\Sigma^\square,C,\mathcal Y)$ is called a puzzle algebra. This is equivalent to the above definition.
}  

\end{defi}



The main result of \cite{berhen} (Theorem 0.1) is the following:
\begin{theo} Every strongly regular map leaves invariant an ergodic, physical SRB measure supported by 
a non uniformly hyperbolic attractor. Moreover, strongly regular maps are abundant in the following meaning:

For every $\epsilon>0$, there exists $b>0$, such that for every $B$ of $C^2$ norm less than $b$, there exist $\eta>0$ and a subset $\Pi_B\subset [-2,-2+\eta]$ with $\frac{\leb \Pi_B}{\leb [-2,-2+\eta]}\ge 1-\epsilon$  such that for every $a\in \Pi_B$, the map $f_{a\, B}$ is strongly regular.
\end{theo}
\begin{rema}\label{violent} To fix the idea, we will suppose  the following very rough inequalities:
 $M\ge 1000$ and $-\log b\le \exp\, \exp M$.  They are sufficient for the new analytic conditions given by this work. 
 \end{rema}
 
\section{Regular sets of strongly regular dynamics}
\label{defregularsection}

\subsection{Regular set for strongly regular quadratic maps}
Let $P$ be a strongly regular quadratic map satisfying $(SR_1)$-$(SR_2)$. We recall that:
\[\sA:=  \sY_0\sqcup \{\square_+(\sc_i-\sc_{i+1}),\square_-(\sc_i-\sc_{i+1})\}.\]


We recall that $\sa\in \sA$ defines a simple puzzle piece (i.e. belongs to $\sY_0$) if $n_{\sa}\le M$ and defines a parabolic piece if $n_{\sa}\ge M+1$.

Like in the H\'enon case, a chain $(\sa_i)_{i=1}^k\in \sA^k$ is \emph{suitable} iff $\R_{\sa_1\star \cdots \star \sa_k}$ is a non trivial segment. It is \emph{complete} if $\sa_k$ is belongs to $\sY_0$. We recall that it is prime if $\sa_i$ does not belong to $\sY_0$ for $i<k$.

The following will be useful for the Markov partition that we will define:
\begin{prop}\label{primecomple=puzzle}
For every suitable, complete chain $(\sp_i)_{i=1}^k\in \sA^k$, the product $\underline \sp:=\sp_1\star \cdots \star \sp_k$ is a puzzle piece.  
\end{prop}
\begin{proof}
We proceed by induction on $k$. For $k=1$,  $\underline \sp=\sp_1\in \sY_0$ which is indeed a puzzle piece.

Let $k\ge 2$ and assume by induction that $\sp_2\star \cdots \star \sp_k$ is a puzzle piece.  If  $\sp_1\in \sY_0$, then $\underline \sp$ is the product of two puzzle pieces $\sp_1$ and  $\sp_2\star \cdots \star \sp_k$, thus $\underline \sp$ is a puzzle piece. If  $\sp_1$ is of the form $\square_\pm(\sc_i-\sc_i\star \sa_{i+1})$, 
 with $\sa_{i+1}$ a prime puzzle piece. Then $P^{n_{\sp_1}}(\R_{\sp_1})$ is equal to one component of 
 $cl(\R_\se \setminus  \R_{\sa_{i+1}})$ and intersects at a non empty subset $int\; \R_{\sp_2\star \cdots \star \sp_k}$.  As the puzzle pieces are nested or disjoint, it comes that $\R_{\sp_2\star \cdots \star \sp_k}$ is included in  $cl(\R_\se \setminus \R_{\sa_{i+1}})$ and so in  $P^{n_{\sp_1}}(\mathbb R_{\sp_1})$. Consequently $P^{n_{\underline \sp}}(\mathbb R_{\underline \sp})$ is equal to $\R_\se$, and so $\underline \sp$ is a puzzle piece.
\end{proof}

It holds that for $\sa\not= \sa'\in \sA$, the intersection of the segments $\R_\sa$ and $\R_{\sa'}$ consists of at most one point which is in $\cup_{k\ge 0} P^{-k}(A_0)$. Hence the alphabet $\sA$ defines a cover $\{\R_{\sa}:\; \sa \in \sA\}$ of $\R_\se\setminus \{0\}$ which is a partition modulo the preimages of $A_0$.

Similarly, the suitable chain $\sg=(\sa_i)_{i=1}^N  \in \sA^N$ of length $N$, we associate the following segment:
$$\R_\sg := \R_{\sa_1\star \cdots \star \sa_N }.$$
The set of such segments covers $\R_\se\setminus \cup_{k\ge 1}P^{-k}(\{0\})$. As a consequence of the $h$-times property given by Proposition \ref{Proph} and the Collet-Eckmann condition given by Corollary \ref{CEP}, it holds:
\begin{prop}\label{lyapinvmeasure} Every invariant probability measure $\nu$ has a Lyapunov exponent at least $c/3$.
\end{prop}

This partition enables us to define combinatorially a certain Pesin set:
\begin{defi} \label{Rxi2}\index{Regular sequence}A suitable sequence of symbols $\sg=(\sa_i)_{i=1}^N\in \sA^N$  is \emph{regular}
 if the following inequality holds for every $i\le N$:  \index{$\dag$}
\begin{equation}\tag{\dag}
n_{\sa_{i}}\le M+\Xi\sum_{1\le j<i} n_{\sa_j}\;,\end{equation}
with $\Xi:= e^{\sqrt M}$.\index{$\Xi$} \index{Regular sequence}\index{$\Xi$}
\end{defi}

We recall that given two different suitable chains $\sg,\sg'\in \sA^n$ of the same length $n$ , it holds that $\R_\sg\cap \R_{\sg'}$ is empty or equal to a preimage of $A_0$.

\begin{defi}[Regular point]
For every $z\in \R_\se\setminus \cup_{k\ge 1}P^{-k}(\{A_0\})$, let $0\le p\le \infty$ and $(\sa_i(z))_{0\le i< p}$ be the maximal regular chain of symbols so that  $z$ belongs to $\R_{\sa_0(z)\star \cdots \star \sa_i(z)}$ for every $i< p$. The point $z$ is said \emph{$p$-regular}. If $p=\infty$, the point $z$ is called  \emph{regular}.\index{Regular point}

\end{defi}
We notice that the chain is empty when $p=0$. This occurs  iff $z$ belongs to $\R_\square$. Otherwise, $\sa_0(z)$ belongs to $\sY_0$.
\begin{rema} 
By definition, whenever $p<\infty$, a $p$-regular point is not $p-1$-regular nor $p+1$-regular.
\end{rema}

Clearly, if $z\in \R_\se\setminus \cup_{k\ge 1}P^{-k}(\{A_0\})$ is $p$ regular then, with $m:=n_{\sa_0(z)}+ \cdots +n_{ \sa_{p-1}(z)}$, the point  $P^{m}(z)$ is equal to $0$ or it belongs to a certain $\R_{\square_\pm(\sc_i-\sc_{i+1})}$ which satisfies:
\[n_{\square_\pm(\sc_i-\sc_{i+1})} >M+\Xi\sum_{0\le j<p} n_{\sa_i(z)}\; .\]
Then we put $\sa_p(z)=\square$, $z':=P^{m+M+1}(z)$ and we define  $\sa_{i+p+1}(z)= \sa_i(z')$ for every $i\ge 1$. 

Such a recursion defines a (full) sequence $\underline \sa(z)=(\sa_i(z))_{1\le i<\infty}$ for every $z\in \R_\se\setminus \cup_{n\ge 0} P^{-n}(A_0)$
\begin{defi}
A point $z$ is \emph{infinitely irregular} if $\sa_i(z)=\square$ for infinitely many integers $i$.
Otherwise $z$ is \emph{eventually regular}. 
\end{defi}
Indeed if $z$ is \emph{eventually regular}, with $N$ minimal so that $\sa_i(z)\not =\square$ for every $i\ge N+1$, with the convention $n_\square =M+1$, it holds that $P^{n_{\sa_1(z)}+\cdots +n_{\sa_N(z)}}(z)$ is regular.  
 
By looking at the Lyapunov exponent of the invariant probability measures, we will prove the following in Appendix \ref{prinvLyapExp}: 
\begin{prop}\label{evenregudim1}
For every invariant probability measure $\nu$ with support off $\{A_0,A_0'\}$, $\nu$-almost every point $z\in \R_\se$ is eventually regular or satisfies that 
$\sa_i(z)=\square$ for all $i$ large enough.
\end{prop}

\subsection{Regular set for strongly regular Hénon like maps}
For a strongly regular Hénon like map $f$ with structure $\{\sG, \star ,\square\}$, we are going to encode the dynamics thanks to a family of partitions $(\mathcal P(t))_{t\in   T^*}$. This encoding will define the regular and irregular sets.
  
By $(SR_2)$, every puzzle piece involved in the common sequences of $\sc^t$, $t\in T^*$ is given by a suitable, complete chain from $S^{t\!t}$. 

Conversely, a straight forward generalization of Proposition \ref{primecomple=puzzle} (it is shown as Lemma 7.12 in \cite{berhen}), states the following:
\begin{prop}\label{completearepuzzlepiece}
For every suitable, complete chain $(\sp_i)_{i=1}^k\in \sA^k $ from $S^{t\! t}$, the pair $\underline \sp(S^{t\!t}):=\sp_1\star \cdots \star \sp_k(S^{t\! t})$ is a puzzle piece of $S^{t\! t}$.  
\end{prop}

\subsubsection{Partition $\mathcal P(t)$ associated to $t\in    T^*$.} \label{partietiongeo}
Let $t\in    T^*$ and let $\sc=\sc^t$ be its associated common sequence  by $(SR_1)$. 
We recall that $(Y_{\sc_i})_i$ is a nested sequence of boxes the intersection of which is the curve $W^s_{\sc}$.

Therefore $\{Y_{\sc_{i}}\setminus Y_{\sc_{i+1}};\; i\ge 0\}\cup \{f^{-M-1}(W^s_{\sc})\cap Y_\square\}$ is a partition of $Y_\se$.   Put:
\[Y_{\square (\sc_i-\sc_{i+1})}=cl\Big(f^{-M-1} (Y_{\sc_i}\setminus Y_{\sc_{i+1}})\Big)\cap Y_\square,\quad Y_{\square c}:= f^{-M-1}(W^s_{\sc})\cap Y_\square.\]
\[\partial^s Y_{\square (\sc_i-\sc_{i+1})}=f^{-M-1} (\partial^s Y_{\sc_i}\cup \partial^s Y_{\sc_{i+1}})\cap Y_\square\; . \]

These sets have a very tame geometry:

\begin{prop}\label{prop des common extension}
The boundary of $Y_{\square (\sc_i-\sc_{i+1})}$  is formed 
by segments of the lines $\{y=\pm 2\theta\}$ and by arcs of curves $\sqrt{b}$-$C^2$ close to arcs of  parabolas of the form $\{(x,y): P(x)+y= cst\}$.

For every $z\in Y_{\square (\sc_i-\sc_{i+1})}$, every $n\ge 0$, the following inequality holds:
\begin{equation}\tag{$\mathcal{PCE}^{n_{ \sc_i}+M+1}$}
\|D_zf^{k}(0,1)\|\ge e^{-M c^{+} k},\quad \forall k\le n_{\square_\pm(\sc_i-\sc_{i+1})}:=M+1+n_{\sc_i}.\end{equation}
Moreover $D_zf^{n_{\sc_i}+M+1}(0,1)=:(u_x,u_y)$ satisfies  $|u_y|\le \theta |u_x|$ and  
\begin{equation}\tag{$\mathcal{CE}^{n_{ \sc_i}+M+1}$}\label{CEacertainmoment} \|D_zf^{n_{\sc_i}+M+1}(0,1)\|\ge e^{ c^{-}(n_{\sc_i}+M+1)}.\end{equation}
\end{prop}

\begin{proof}
Inequalities $(\mathcal{PCE}^{n_{ \sc_i}+M+1})$ and (\ref{CEacertainmoment}) are given by Proposition $14.2$ of \cite{berhen}. The statement about the geometry of $Y_{\square(\sc_i-\sc_{i+1})}$ will be generalized in Proposition \ref{geoYg} and proved afterward. 
%
%
\end{proof}

In Figure \ref{position}, we draw all the possible topological shapes for $Y_{\square (\sc_i-\sc_{i+1})}$. From this, we remark that $Y_{\square (\sc_i-\sc_{i+1})}$ has one, two or three components. There is at most one component disjoint from $S^t$. If such a component exists, we denote it by $Y_{\square_b (\sc_i-\sc_{i+1})}$. There are one or two components which intersect $S^t$. If there are two components which intersect $S^t$, then there is one component at the left hand side of the other. We denote this component by $Y_{\square_- (\sc_i-\sc_{i+1})}$. The component at the right hand side of the other is  denoted by  $Y_{\square_+ (\sc_i-\sc_{i+1})}$. If there is only one component which intersects $S^t$, we  denote it by   $Y_{\square_a (\sc_i-\sc_{i+1})}$. In this case, we shall split $Y_{\square_a(\sc_i-\sc_{i+1})}$ into two components $Y_{\square_\pm (\sc_i-\sc_{i+1})}$.

\begin{figure}[h!]
    \centering
        \includegraphics{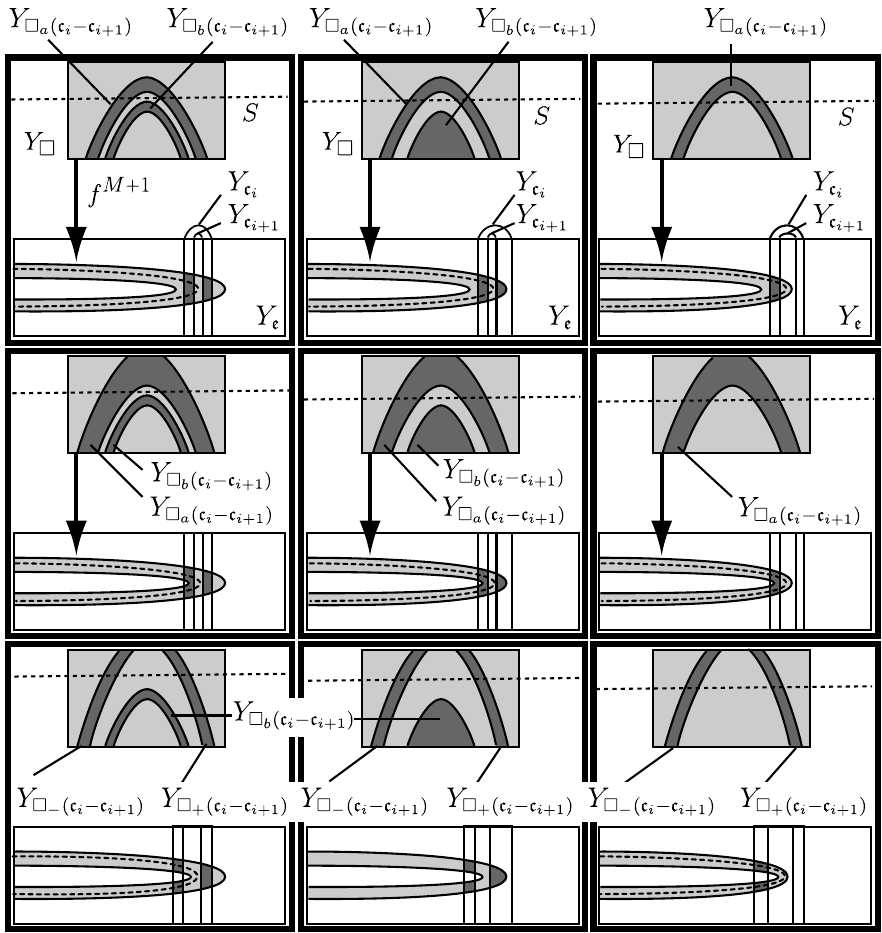}
    \caption{Possible shapes for $Y_{\square (\sc_i-\sc_{i+1})}$.}
   \label{position}
\end{figure}

Let $z_0$ be the point at the center of the segment $(f^{M+1}|S^t_\square)^{-1}(Y_{\sc_{i+1}})$ of $S^t$. 
Let $\Delta$ be the vertical line passing through $z_0$.  
 The line $\Delta$ splits the set $Y_{\square_a (\sc_i-\sc_{i+1})}$ into two components. The one at the left (resp. right) of the other is denoted by $Y_{\square_- (\sc_i-\sc_{i+1})}$ (resp. $Y_{\square_+ (\sc_i-\sc_{i+1})}$). We add $\Delta\cap Y_{\square_a (\sc_i-\sc_{i+1})}$ to $Y_{\square_- (\sc_i-\sc_{i+1})}$.  Figure \ref{partition} depicts this splitting.

\begin{figure}[h!]
    \centering
        \includegraphics{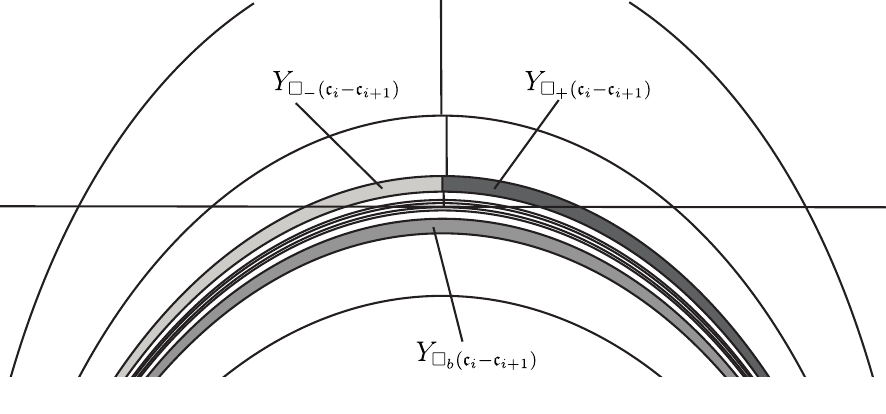}
    \caption{Partition of $Y_\square$. }
   \label{partition}
\end{figure}

For $\delta\in \{+,-,b\}$, we put 
$$\partial^sY_{\square_\delta(\sc_i-\sc_{i+1})}=\partial^sY_{\square (\sc_i-\sc_{i+1})}\cap  \partial Y_{\square_\delta (\sc_i-\sc_{i+1})}\quad \text{and}\quad 
\partial^uY_{\square_\delta(\sc_i-\sc_{i+1})}= cl(\partial Y_{\square_\delta (\sc_i-\sc_{i+1})}\setminus \partial^sY_{\square_\delta(\sc_i-\sc_{i+1})})
\; .$$ 

We remark that $\mathcal P(t):= \{Y_{\sa};\; \sa\in \sY_0\}\cup\big\{ Y_{\square_\delta (\sc^t_i-\sc^t_{i+1})};\;  i\ge 0,\; \delta\in \{+,-,b\}\big\}\cup Y_{\square c^t}$
 is a \emph{partition of} $Y_\se$  \emph{modulo} $W^s(A)$.\index{$\mathcal P(t)$}
This means that $\mathcal P(t)$ is a covering of $Y_\se$ and every pair of different elements of $\mathcal P(t)$ have their intersection in $W^s(A)$. 
 
The  partition $\mathcal P(t)$ depends on $t\in   T^*$, since $\sc_i:= \sc_i^t$ depends on $t$ and the lines $\Delta$ depend on $S^t$.

\index{$\sP$}  Let $\sP(t):= \sY_0\sqcup\{\square_\delta (\sc^t_i-\sc^t_{i+1}):\; i\in \mathbb N,\; \delta\in \{\pm, b\}\}\sqcup \{\square c^t\}$ be the set of symbols associated.
The set $\sP(t)$ is countable. If $\sa\in \sY_0$, we already defined an integer $n_\sa$. 
  
 Put $n_{\square_\delta (\sc^t_i-\sc^t_{i+1})}=M+1+n_{\sc^t_i}$ for $i\in \mathbb N$ and $\delta\in\{+,-,b\}$. Put $n_{ \square \sc^t}= \infty$. 

All the elements of $\sP(t)$ belong to $\sA$ but those of the form $\square c^t$ and  $\square_b (\sc^t_i-\sc^t_{i+1})$ for $i\ge 0$. 


\subsubsection{Regular points of Hénon like maps}
We recall that $t\!t:= (s_-)_{n\le -1}\in T_0$ is the exponent of the curve  $S^{t\!t}$ equal to a half local unstable manifold of $A$. This curve will be the starting point for the encoding.

Given a suitable sequence of symbols $\sg=(\sa_i)_{i=0}^n\in \sA^n$ from the curve $ S^{t\! t}$, the symbol 
$\sa_{i+1}$ belongs to $\sP(t\cdot \sa_0\cdots \sa_i)$ for every $i\in [0, n)$. This leads us to consider the set of points $z\in Y_\se$ such that $f^{n_{\sa_0}+\cdots +n_{\sa_i}}(z)$ belongs to $Y_{\sa_{i+1}}$ for every $i<n$. This set has in general a very wild geometry. However it is not the case if the sequence is regular. We recall that $\Xi:=e^{\sqrt M}$.\index{$\Xi$}

\begin{defi}\label{defregular} \index{Regular sequence}   A sequence of symbols $\sg=(\sa_i)_{i=0}^n\in \sA^n$  is \emph{regular}
 if $\sg$ is   suitable  from $ S^{t\! t}$ and the following inequality holds for every $i\le n$:  
\label{Rxi}
\begin{equation}\tag{$\dag$} n_{\sa_{i}}\le M+\Xi\sum_{0\le j<i} n_{\sa_j}\; .\end{equation}
\end{defi}
Also for $i=0$, the above equation gives $n_{\sa_0}\le M$ and so that $\sa_0$ is a simple piece: $\sa_0\in \sY_0$.

\begin{defi}\label{defweaklyregular} \index{Weakly regular sequence} A sequence of symbols $\sg=(\sb_i)_{i=0}^n\in \sA^n$  is \emph{weakly regular}
 if $\sg$ is   suitable  from $ S^{t\! t}$ and the following inequality holds for every $i\le n$:  
\label{Rxiweak}
\[n_{\sb_{i}}\le \Xi\cdot (M+\sum_{1\le j<i} n_{\sb_j})\;\]
\end{defi}

We notice that a common sequence is regular and that a regular sequence is  weakly regular.

\begin{prop}\label{geoYg}
For every weakly regular sequence $\sg=\sb_0 \cdots \sb_n$, the set 
\[Y_\sg:=\{z\in Y_\se:\;  f^{n_{\sb_0\cdots \sb_i}}(z)\in Y_{\sb_{i+1}},\; \forall i<n\}\]
is a box which satisfies the following properties:
\begin{enumerate}
\item  $\partial^s Y_{\sg}$ is formed by two segments of the stable manifold of $A$; both link  $\big\{y=-2\theta\big\}$ to $\big\{y=2\theta\big\}$ and are $\sqrt{b}$-$C^2$-close to an arc of a curve of the form $\{ P(x)+y=cst\}$. Its tangent vectors are $\theta^k$ contracted by $Df^k$ for every $k\ge 0$.
\item Both components of $\partial^s Y_{\sg}$ of are linked by two segments $\partial^u Y_{\sg}$ of 
$\big\{y=-2\theta\big\}$ and $\big\{y=2\theta\big\}$ respectively.

\item For every $z\in Y_{\sg}$, every vector $u=(u.x,u.y)$ so that $|u.y|\le \theta |u.x|$, it holds: 
\[\left\{\begin{array}{c} e^{\frac{c}3 (n_{\sg}-k)} \|D_zf^{k}(u)\|\le \|D_zf^{n_{g}}(u)\|\quad \forall k\le n_{\sg},\\
e^{-c^+(M+1)\Xi k} \le \|D_zf^{k}(u)\|\quad \forall k\le n_{\sg}.\end{array}\right.\]
\item Every $z\in Y_{\sg}$ belongs to a curve $\mathcal C\subset Y_\sg$, of length less than 1, intersecting every flat stretched curve and being $\theta^k$-contracted by $f^k$ for every $k\le n_\sg$.  
\item The set $Y^\sg:=f^{n_\sg}(Y_\sg)$ is a box such that
$\partial^u Y^\sg:= f^{n_\sg}(\partial^u Y_\sg)$ is the disjoint union of two flat curves and $\partial^s Y^\sg:= f^{n_\sg}(\partial^s Y_\sg)$ is made by two $\theta^{n_\sg}$-small segments of $W^s (A)$ passing trough the endpoints of $f^{n_\sg}(S^{t\! t}_\sg)$.
\end{enumerate}
\end{prop}
The proof of this proposition is done in \textsection \ref{sectionpreuvegeoYg}.
As for common sequences, given a weakly regular sequence $\underline \sb:= (\sb_i)_i$, we can define $W^s_{\underline \sb}:= \cap_{j\ge 0} Y_{\sb_0\cdots \sb_j}$. An immediate consequence of the above Proposition is the following:
\begin{coro}\label{geodesvarstable} The set $W^s_{\underline \sb}:= \cap_{j\ge 1} Y_{\sb_0\cdots \sb_0}$ is a connected curve with an endpoint in each of the lines $\big\{y=\pm 2\theta\big\}$. The curve $W^s_{\underline \sb}$ is $\sqrt{b}$-$C^{1+Lip}$-close to a segment of a curve of the form $\{ P(x)+y=cst\}$ and its tangent vectors are $\theta^k$-contracted by $Df^k$ for every $k\ge 0$.
\end{coro}

We are now ready to encode the dynamics, with respect to these regular sequences. 

For $z\in Y_\se\setminus W^s(A)$, let $\underline \sa(z):= (\sa_i(z))_{0\le i< p}$ be the maximal regular sequence of symbols in $\sA$ such that $z\in Y_{\sa_0\cdots \sa_{i}(z)}$, for every $i<p\in[0,\infty]$. Note that if $z\in Y_\square$, then  $p=0$ and the sequence $\underline \sa(z)$ is empty.  Otherwise $p\ge 1$ and $\sa_0(z)\in \sY_0$.

This sequence is uniquely defined. Indeed, as $z$ belongs to $Y_\se\setminus W^s(A)$, by induction on $i<p$, $f^{n_{\sa_0\cdots \sa_i}}(z)$ belongs to  $Y_{e}\setminus W^s(A)$. Therefore there exists a unique symbol $\sa_{i+1}\in \sP(t\! t\cdot \sa_0\cdots \sa_i)$ such that $z$ belongs to $Y_{\sa_{i+1}}$.

  \begin{defi} Such a point $z\in Y_\se\setminus W^s(A)$ is $p$-\emph{regular}. If $p=\infty$, the point $z$ is \emph{regular}. We denote $\tilde {\mathcal R}$ the set of regular points:\index{regular@$p$-regular points}
  \[\tilde {\mathcal R}=\{z\in Y_\se: \; z \; \text{is regular}\}\;.\]
	\index{$\tilde {\mathcal R}$}
  \end{defi}

We notice that every point of $Y_\se\setminus W^s(A)$ is at least $0$-regular. 

When $p=\infty$, by $\theta$-contraction of $W^s_{\underline \sa(z)}$, for every $i$, $f^{n_{\sa_0\cdots \sa_i}}(z)$ is $\theta^{{n_{\sa_0\cdots \sa_i}}}$-close to $S^{t\! t\cdot \sa_0\cdots \sa_i}$. Also by Condition $(\diamondsuit)$, the set 
$\partial^u Y_{\sa_{i+1}}$ is $\min(\theta, e^{-c^+n_{\sa_{i+1}} -2c^+\aleph(n_{\sa_{i+1}})})$-distant to $S^{t\! t\cdot \sa_0\cdots \sa_i}$. Using the fact that $n_{\sa_{i+1}}\ge M+\Xi n_{\sa_0\cdots \sa_i}$, it comes that   $f^{n_{\sa_0\cdots \sa_i}}(z)$ cannot belong to $\partial^u Y_{\sa_{i+1}}$.

Since the boundary of $Y_{\sa_{i+1}}$ is included in $W^s(A)\cup \partial^u Y_{\sa_{i+1}}$, and since $\tilde {\mathcal R}$ is disjoint to $W^s(A)$, it comes:
\begin{Claim}\label{pouraC0} \label{pour scs} 
The map $z\in \tilde {\mathcal R}\mapsto \underline \sa(z)\in \sA^{\mathbb N}$ is continuous, for $\sA^{\mathbb N}$ endowed with the discrete product topology.  Also the closure of $\tilde {\mathcal R}$ is included in $\tilde {\mathcal R}\cup \bigcup_{\underline \sa\; \text{regular chain}}\partial^s Y_{\underline \sa}$.
\end{Claim}

We show below the following important Proposition:
\begin{prop}\label{caracterisationdessequences} 
Let $z\in Y_\se\setminus W^s(A)$ be a $p$-regular point and let $(\sa_j)_{j\ge 0}:= \underline \sa(z)$.   

If $p<\infty$, then the symbol $d\in \sP(t\! t\cdot \sa_0\cdots \sa_{p-1}(z))$ such that $f^{n_{\sa_0\cdots \sa_{p-1}}}(z)$ belongs to $Y_d$ satisfies:
\[n_{d}> M+\Xi\sum_{0\le j<p } n_{\sa_j}.\] 
\end{prop}

In particular, this implies that $d$ is of the form  $\square c$ or $\square_\delta (\sc_i-\sc_{i+1})$. Then $z_p:=f^{n_{\sa_1\cdots \sa_{p-1}(z)}+M+1}(z)$ belongs to $W^s _\sc $ or $Y_{\sc_i}$, with $n_{\sc_i}\ge \Xi n_{\sa_0\cdots \sa_{p-1}}$. We recall  
that $\sc$ or $\sc_i$ is the $\star$-product of simple or parabolic pieces $(\sa_i')_{i=0}^m$ which form a regular sequence. Moreover, $m\ge n_{\sc_i}/(M+1)$ (see. \cite{berhen}).  This implies:
\begin{coro}\label{progressionexpo}
The point  $z_p$ is at least $\Xi n_{\sa_0\cdots \sa_{p-1}}/(M+1)$-regular. 
\end{coro}

This is why, for every $p$-regular point $z$, with $0\le p<\infty$, we complement the sequence $(\sa_i(z))_{i< p}$ of $z\in Y_\se\setminus W^s(A)$ by the following inductive way: 

Put  $\sa_p(z)=\square$ with $n_\square=M+1$ and then inductively $\sa_{p+i}(z):=\sa_{i}(z')$, for $i\ge 0$, with $z':= f^{m}(z)$ and $m=\sum_{i\le p}n_{\sa_i(z)}$.  


\begin{defi} 
A point $z$ is \emph{infinitely irregular} if the sequence $\underline \sa(z)$ takes infinitely many times the value $\square$. \index{Infinitely irregular}
Otherwise $z$ is \emph{eventually regular}: there exists $j$ such that $(\sa_{j+i}(z))_{i\ge 0}$ is regular.\index{regular@Eventually regular}
\end{defi} 

We notice that if $z$ is regular then $\underline \sa(z)$ belongs to $\sA^\mathbb N$. 
 
\begin{proof}[Proof of Proposition \ref{caracterisationdessequences}]  Let $\sg = \sa_0\star \cdots \star \sa_j(z)$ for any $j<p$. 

By Proposition \ref{geoYg}.4, the point $z':= f^{n_\sg}(z)$ belongs to the $\theta^{n_\sg}$-neighborhood of $S^{t\! t\cdot \sg}$.

Let us suppose that $z'$ belongs to a set of the form $Y_{\square_b (\sc_i-\sc_{i+1})}$. This implies that $f^{M+1}(z')$ belongs to the component of $Y_{\underline \sc_i}\setminus Y_{\underline \sc_{i+1}}$ which does not intersect $S^{t\! t\cdot \sg \square}$. From the tangency position,  the curve $S^{t\! t\cdot \sg \square}$ is $e^{-(2 \aleph(i+1)+1+n_{\sc_{i+1}})c^+}$-far from this other component. On the other hand, as $z'$ is $\theta^{n_\sg}$-close to $S^{t\! t \cdot \sg}$ and belongs to the convex $Y_\square$, hence the point $z'$ is $\theta^{n_\sg}$-close to $S^{t\! t \cdot \sg}_\square$. Thus $f^{M+1}(z')$ is $\theta^{n_\sg} e^{c^+ (M+1)}$-close to $S^{t\! t\cdot \sg \square}$. Therefore, $z'$ can belong to $Y_{\square_b (\sc_i-\sc_{i+1})}$ only if:
\begin{equation}\label{uneinegalitedetangence} e^{-(n_{\sc_{i+1}}+2 \aleph(i+1)+1)c^+}\le  \theta^{n_\sg} e^{c^+ (M+1)}\end{equation}  
 As $(n_{\sc_{i+1}}+ 2\aleph(i+1)+1)c^+\le (2c^+ +\frac {c}{3 } )n_{\square_b (\sc_i-\sc_{i+1})}$, inequality (\ref{uneinegalitedetangence}) implies:
\[-(2c^++\frac c3)n_{\square_b (\sc_i-\sc_{i+1})}\le n_\sg \log \theta+ c^+ n_{\square_b (\sc_i-\sc_{i+1})}.\]

As by Remark \ref{violent} $-\log \theta/ (4 c^+ + c/3)\ge \Xi +M$, 
  it comes that $n_{\square_b (\sc_i-\sc_{i+1})}$ is greater than $(M+\Xi) n_\sg$.

Suppose  for the sake of a contradiction, that $z'$ belongs to a box of the form $Y_{\sp}$ with $\sp=\square_\pm (\sc_i-\sc_{i+1})$ such that $n_\sp\le M+\Xi n_\sg$. We shall show that  $(\sa_j)_{p\ge j\ge 0}$ is suitable from $S^{t\! t}$ and so is a regular sequence, which is a contradiction with the definition of $p$. 
For this end, it suffices to show that $S^2:= S^{t\! t\cdot \sg}_\sp$ intersects $S^1:= f^{n_\sg}(S^{t\! t}_\sg)$. 

If $S^1\cap S^2= \varnothing$, then the local stable manifolds of the endpoints of $S^1$ and $S^2$ are disjoint. This means that 
$\partial^sY^\sg$ is disjoint from $\partial^sY_\sp$. By the same argument as for Claim \ref{pouraC0}, the set $Y^\sg$ is disjoint from
 $\partial^u Y_\sp$. Consequently, the boundary $\partial Y_\sp$ is disjoint from $\partial Y^\sg$. Thus $Y_\sp$ is either disjoint from $Y^\sg$ or $Y^\sg$ is included in $Y_\sp$. The first case cannot occur since $z'\in Y_\sp\cap Y^\sg$. Thus $Y^\sg\cap S^{t\! t\cdot \sg}=S^1$ is included in $Y_\sp\cap   S^{t\! t\cdot \sg}=S^2$ which is a contradiction.
 
If $S^1$ intersects $S^2$ only at an endpoint, then one curve of $\partial^sY_\sp$ contains one curve of $\partial^sY^\sg$. As $\partial^uY_\sp$ is disjoint from $Y^\sg$, the interiors of  $Y_\sp$ and $Y^\sg$ are disjoint. Consequently $Y^\sg\cap Y_\sp$ is included in $W^s(A)$. This is a contradiction with 
 $Y^\sg\cap Y_\sp\setminus W^s(A)\ni z'$. 

\end{proof}

\subsubsection{Lyapunov exponents of invariant probability measures}
We show here Theorem \ref{MainNew}. 

Let us define:
\[K_\square:= \bigcap_{N\ge 0} \bigcup_{n\ge N} f^n \big(\{z\in Y_\se; \underline \sa(z)=\square\cdots \square\cdots:\; x\in Y_\se\setminus W^s(A)\}\big).\]
The set $K_\square$ is equal to $\bigcap_{n\ge 0} f^{-n(M+1)}(Y_\square)= \bigcap_{n\ge 0} f^{-n(M+1)}(Y_{\square(\se - \sc_1^{t\!t})})$ (which is possibly empty).

For $b$ small enough w.r.t. $M$, this is the hyperbolic continuity of the uniformly hyperbolic compact set of $P^{M+1}$:
\[\bigcap_{n\ge 0} P^{-n(M+1)}(\R_{\square_-(\se-\sc_1)}\cup \R_{\square_+(\se-\sc_1)}),\]
whose expansion is more than $e^{c/3}$. Hence it comes: 

\begin{prop}\label{LyapKsquare}
If $K_\square$ is not empty,  any  of its invariant probability measure has a Lyapunov exponent at least $c/3$.
\end{prop}

\begin{lemm}\label{lyappos} There is no measure $\mu$ with both Lyapunov exponents negative.\end{lemm} 
\begin{proof}[Proof of Lemma \ref{lyappos}] 
For the sake of a contradiction, suppose that such a measure exists. 

First let us recall that by \cite{Katok-Haselblat}, coro S.5.2, p.694,  the support of $\mu$ contains a  periodic attractive orbit $(p_i)_i$.

Hence, either $\underline \sa (p_0)$ is eventually constantly equal to $\square$ or either, by Corollary \ref{progressionexpo}, the sequence $\underline \sa (p_0)$ contains regular segment of arbitrarily long length. 

The case where $\underline \sa (p_0)$ is eventually constantly equal to $\square$ corresponds to $p_0\in W^s(K_\square)$. The compact set $K_\square$ being hyperbolic, both Lyapunov exponents cannot be negative.

If $\sa(p_0)$ contains a regular segment $\sg=\sa_0\dots \sa_n$ of arbitrarily long length $n$, then by Proposition \ref{geoYg}.3, the point $p_0$ cannot belong to an attracting cycle.   
\end{proof}

This Lemma will enable us to prove  in \textsection \ref{sectionPropeventuallyregu}, the following two dimensional generalization of Proposition \ref{evenregudim1}. 
\begin{prop}\label{eventually regu}
For every invariant measure $\mu$ with support off $\{A,A'\}$, $\mu$-almost every point in $Y_\se$ is eventually regular or satisfies $\sa_i(z)=\square$ for all $i$ sufficiently large.
In particular $\mu(K_\square\cup \tilde{\mathcal R})>0$.
\end{prop}

We remark that Propositions \ref{LyapKsquare},  \ref{eventually regu} and \ref{geoYg}.3 imply the following scholium of  Theorem \ref{MainNew}:
\begin{theo}
Every invariant, probability measure $\mu$ for a strongly regular map $f$ has its Lyapunov exponent  at least $c/3$.
\end{theo}

By Proposition \ref{eventually regu},   an ergodic probability measure has its support either included in  $\{A,A'\}$, in $\hat K_\square:= \cup_{M\ge n\ge 0} f^n(K_\square)$ or in $\cup_{n\ge 0} f^n( \tilde {\mathcal R})$. 
We will see in section \ref{Ksquare} that the  entropy of  the measures  supported by $\hat K_\square$ is small. Hence we look at the measures supported by $\tilde {\mathcal R}$. To study their ergodic properties, we are going to split $\tilde {\mathcal R}$ into two subsets:
\begin{itemize} 
\item ${\mathcal R}$ which has a Markovian structure,
\item $\mathcal E$ which intersects only the support of probability measures with small entropy.  
\end{itemize}

We recall that we defined the set $\tilde {\mathcal R}$ thanks to the alphabet $\sA$ and the map:
\[z\mapsto \underline \sa(z)\in \sA^{\mathbb N}\; .\] 
Let $\tilde \sR$ be the image of this maps. Every point in $\tilde \sR$ is a sequence in $(\sa_i)_i\in \sA^{{\mathbb N}}$ which satisfies 
\[n_{\sa_i}\le M+1+\Xi \sum_{j=0}^{i-1} n_{\sa_j}\;.\] 
On $\sA^{\mathbb N}$, the shift dynamics $\tilde \sigma$ acts canonically. \index{$\tilde \sR$, $\sR$, $\sE$}
We observe that the set $\tilde \sR$ is not invariant by $\tilde \sigma$. Furthermore, not every point $(\sa_i)_i\in \tilde \sR$ comes back to $\tilde \sR$. 

Let $\sR$ be the points $(\sa_i)_{i\ge 0}\in \tilde \sR$ which return infinitely many times in $\tilde \sR$  by the shift dynamics.
Let $\sE$ be the complement of $\sR $ in $\tilde \sR$.
\[\sR:=\{\underline \sa := (\sa_i)_i\in \tilde \sR : \; \forall N\ge 0,\; \exists n\ge N\text{ s.t. } \tilde \sigma^n(\underline \sa)\in \sR\}\;.\]
\[\sE:= \tilde \sR \setminus \sR\; .\]

They define the sets:
\[\mathcal R := \{ z\in Y_\se\colon \underline \sa(z)\in \sR\}\quad \text{and}\quad \mathcal E := \{ z\in Y_\se\colon \underline \sa(z)\in \sE\}\;.\]\index{$\mathcal R$, $\mathcal E$}
which are called respectively \emph{infinitely regular set} and \emph{exceptional set}.

We observe that $\tilde{\mathcal R} = \mathcal R \cup \mathcal E$.

We will see in section \ref{setE}  that the entropy (and the Hausdorff dimension) of any ergodic probability measure supported by the orbit of $\mathcal E$ is small. Hence the interesting ergodic, probability measures  are contained in the orbit of $\mathcal R$. The next section is devoted to the study of this set, thanks to a Young tower. 

\section{Young Tower on $\Lambda$}\label{sectionMarkov}

In this section we deduce from the latter section a structure of Young tower on a subset $\Lambda$ satisfying properties $Y_1-Y_2-Y_3-Y_5$ of   Pesin-Senti-Zhang \cite{Pe10}. This allows to deduce the existence and uniqueness  of Gipps states and many of their properties. 

\subsection{The set $R$}
We recall that $\mathcal R=\cup_{\underline \sa\in \sR}W^s_{\underline \sa}$, where $\sR\subset \sA^{\mathbb N}$ is the set of regular sequences $\underline \sa\in \tilde \sR$ which come back infinitely many times in $\tilde \sR$ by the shift map  $\tilde \sigma\colon \sA^\mathbb N\circlearrowleft$. 

Hence every point in $\sR$  comes back infinitely many times in $ \sR$. 
For every $\underline \sa\in \sR$, let $N_\sR (\underline \sa)\in \mathbb N$ 
be the first return time of $\underline \sa$ in $\sR$. 

Also we put  $N_R(\underline\sa):= n_{ \sa_0}+\cdots +n_{\sa_{N_\sR(\sa)-1}}$. 

The return maps associated are the following:\index{$\tilde \sigma^\sR$, $f^{\mathcal R}$}
$$\tilde \sigma^\sR\colon \underline \sa\in \sR\mapsto \tilde \sigma^{N_\sR(\underline \sa)}(\underline \sa)\in \sR\quad \text{and}\quad f^{\mathcal R}\colon z\in \mathcal R\mapsto f^{N_{R}(\underline \sa(z))}(z)\in \mathcal R$$
Clearly the map $\underline \sa\colon z\in \mathcal R\mapsto \underline \sa(z)\in \sR$ semi-conjugates these dynamics:
\[\tilde \sigma^\sR\circ \underline \sa= \underline \sa\circ f^{\mathcal R}\]

\begin{rema} The map $\sigma^\sR$ is the first return map  in $\sR$, but $f^\mathcal R$ is in general NOT the first return map in $\mathcal R$.  Suppose there exists  $z\in \mathcal R$ such that $\underline \sa(z)$ is of the form $\sc_1\cdot \square_\pm (\sc_1-\sc_2)\cdot \sc_1\cdot \square_\pm (\sc_1-\sc_2)\cdots \sc_1\cdot \square_\pm (\sc_1-\sc_2)\cdots$.  Then  the first return time  of $z$ in ${\mathcal R}$ is $n_{\sc_1}+M+1$ and not $n_{\sc_1}+ M+1 +n_{\sc_1}$ as given by  $f^{\mathcal R}$.\end{rema} 

Hence the function $z\in \mathcal R\mapsto N(\underline \sa(z))$ is possibly not in $L^1(\mu)$ for some ergodic measure $\mu$.   

 A strategy would be to look at the combinatorial structure of $\mathcal R$  to show that it is always integrable. 
  Here we use a different strategy, we exhibit $R\subset \mathcal R$ so that the first return time of $x\in R$ by $f$ is also $N(\underline \sa (x))$. This leads us to consider: 
$$R:= \cap_{n\ge 0} (f^{\mathcal R})^n({\mathcal R})$$\index{$R$@$R$}
   
First of all it is important to know if the orbits of $R$ and $\mathcal R$ support the same measures.  By looking at the Lyapunov exponent we will show in Appendix \ref{prinvLyapExp}:
 \begin{prop}\label{EgaliterdesR} For every $f$-invariant probability $\mu$, the sets $\cup_{n\ge 0} f^n ( R)$ and $\cap_{N\ge 0} \cup_{n\ge N} f^n (\mathcal R)$ are equal $\mu$-almost everywhere. \end{prop} 
 
We will also show in section \ref{Proofof proplift} the following:
\begin{prop}\label{lift}
The map $f^{\mathcal R}$ is the first return map of $R$ into $\tilde {\mathcal R}$ induced by $f$:
\[\forall z\in R,\; \forall i\in (0,  N({\underline \sa(z)})), \quad f^i(z)\notin \tilde {\mathcal R}\; .\]
\end{prop}
This implies that $f^{\mathcal R}|R$ is a bijection. Hence the inverse limit $\ola R$ of $R$ for $f^{\mathcal R}$ is equal to $R$:
\[\ola R= R\; .\]

 On the other hand, the inverse limit $\ola\sR\subset \sA^\mathbb Z$ for $\tilde \sigma ^\sR$ which is formed by sequences $(\sa_i)_{i\in \mathbb Z}$ so that $(\sa_i)_{i\ge 0}$ is $\sR$ and for infinitely many $k\ge 0$, the sequence $(\sa_{i-k})_{i\ge 0}$ is in $\sR$.  The semi-conjugacy $\underline \sa$ lifts canonically to the inverse limit to produce a map $\ola \sa$:
 \[\begin{array}{ccccc}
&& \ola \sa&&\\
&R&\to &\ola \sR&\\
f^{\mathcal R}&\downarrow & &\downarrow &\tilde \sigma^\sR\\
&R &\to &\ola  \sR&\\
&&\ola \sa&&
\end{array}
\]
\begin{prop}\label{Bij}
The map $\ola\sa$ is a bijection from $R$ onto $\ola\sR$.
\end{prop}
\begin{proof}
%

Let $(\underline \sa^{-i})_{i\ge 0}$ be a $\tilde \sigma^\sR$- preorbit in $\ola \sR$ : 
$\tilde \sigma(\underline \sa^{-i})=\underline \sa^{-i+1}$. Since:
\[\tilde \sigma^\sR\circ \underline \sa= \underline \sa\circ f^{\mathcal R}\;,\]
 the following curves are nested:
\[(f^{\mathcal R})^i(W^s_{\underline \sa^{-i}})\subset (f^{\mathcal R})^{-i+1}(W^s_{\underline \sa^{-i+1}})\subset\cdots \subset  W^s_{\underline \sa^0}\;.\]
By Proposition \ref{geoYg}.4, the length of $(f^{\mathcal R})^i(W^s_{\underline \sa^{-i}})$ is smaller than $\theta^{i}$. Thus this nested sequence of compact curves converges to a unique point $z\in \cap_{i\ge 0} (f^{\mathcal R})^i(\mathcal R)=R$. 
This proves both the surjectivity and the injectivity of $\ola \sa$. 
\end{proof}

Every point in $\mathcal R$ has a nice stable manifold. Actually one can show that every point in $\mathcal R$ has also a flat local unstable manifold, but in general it does not have a flat local unstable manifold which stretches across $Y_\se$. Such a \textquotedblleft non-stretching across property\textquotedblright occurs for instance for a point $z$ so that $\sa_{2i}(z)=\ss\in \sY_0$ and $\sa_{2i+1}(z)=\square_+(\sc_0-\sc_1)$ for every $i\le 0$. 
We recall that $\sc_0=\se$ and that $\sc_1\in \sY_0$. Then a local unstable manifold is contained in the limit of the flat stretched curve $S^{t\!t\cdot \ss\cdot  \square_+(\sc_0-\sc_1)\cdots  \ss\cdot  \square_+(\sc_0-\sc_1)}$ but the flat segment of this local stable manifold stops at $Y_{\sc_1}$ (inside which  the local unstable manifold is folded) and so it is not stretched.

On the other hand, we will see  in the sequel that  $f^{n_\ss}(z)$ has not only a local unstable manifold which is a flat stretched curve, but also a stable manifold enjoys a nice geometry.

\subsection{The sets $\Lambda$ and $\sL$}\label{thesubsectionlambda}
This leads us to consider the following symbolic sets:
$$\sL:= \tilde \sigma(\ola \sR), \quad 
\sL^s:= \{(\sa_i)_{i\ge 0}: \exists  (\sa_i)_{i\in \mathbb Z}\in \sL\}= \tilde \sigma(\sR)\quad 
\sL^u:= \{(\sa_i)_{i\le -1}: \exists  (\sa_i)_{i\in \mathbb Z}\in \sL\} $$
which will be useful to define a Young tower on the following set:
$$\Lambda:= \cup_{\ss\in \sY_0} f^{n_\ss}(Y_\ss\cap R).$$

Actually the map $h\colon \sqcup_{\ss\in \sY_0} int(Y_\ss)\ni z\mapsto f^{n_\ss}(z)\in Y_\se$ is injective since it is the first return map of points in $Y_\se$ by  $f$ into $Y_\se$. Note that $h$ sends $R$ onto $\Lambda$. Thus we can define:
\[\ola\sb\colon z\in \tilde \sigma(\overleftarrow \sR)= \Lambda\mapsto \tilde \sigma \circ \ola\sa\circ h^{-1}(z)\in \sL\quad\text{and}\quad  \underline \sb \colon z\in \Lambda\mapsto \tilde \sigma \circ \underline \sa\circ h^{-1}(z)\in \tilde \sigma(\sR)=\sL^s\; . \]

Furthermore, the map $h$ is a homeomorphism onto its image (the stable manifold of $A$ which contains $\cup_{\ss\in \sY_0} \partial ^s Y_s$ is disjoint from $\mathcal R$ and so from $R$). Put:

\[f^\Lambda=h\circ f^{\mathcal R}\circ h^{-1}:\Lambda\to \Lambda\quad \text{and}\quad 
\tilde \sigma^\sL:=\tilde \sigma\circ \tilde \sigma^\sR\circ \tilde \sigma^{-1}\; .\]\index{$\tilde \sigma^\sL$, $f^\Lambda$}

Let $N_\sL\colon \sL\to \mathbb N$ be such that 
$\tilde \sigma^\sL= \tilde \sigma^{N_\sL}$. We remark 
that  $N_\sL( (b_i)_{i\in \Z})= N_\sR((b_{i-1})_{i\ge 0})$. 

Let $N_\Lambda \colon \Lambda\to \mathbb N$ be such that 
$f^\Lambda= f^{N_\Lambda}$. We remark 
that  $N_\Lambda(z)=  n_{\sb_0(z)\star \cdots \star \sb_{N_\sL(\ola \sb(z))-1}}$. 



Also the following diagram commutes:

 \[\begin{array}{ccccc}
&& \ola \sb&&\\
&\Lambda&\to &\ola \sL&\\
f^\Lambda&\downarrow & &\downarrow &\tilde \sigma^\sL\\
&\Lambda &\to &\ola  \sL&\\
&&\ola \sb&&
\end{array}
\]

From Proposition \ref{eventually regu}, for every invariant probability  measure $\mu$, it holds:
\[\mu( \tilde {\mathcal R}\cup K_\square\cup\{A,A'\})>0\]
As  $\tilde {\mathcal R}= \mathcal R\cup \mathcal E$ it holds:
\[\mu( \cup_{n\ge 0}  f^n({\mathcal R})\cup f^n(\mathcal E) \cup \hat K_\square\cup\{A,A'\})=1\]
Thus, if the measure $\mu$ is ergodic, it holds:
 \[\mu( \cup_{n\ge 0}  f^n({\mathcal R}))=1\quad \text{or} \quad \mu( f^n(\mathcal E))=1 \quad \text{or} \quad \mu(\hat K_\square)=1\quad \text{or} \quad \mu(\{A,A'\})=1\]
As $\cup_{n\ge 0}  f^n({\mathcal R})$ is equal to 
$\cap_{N\ge 0} \cup_{n\ge N}  f^n({\mathcal R})$ modulo a $\mu$-null set, by Proposition \ref{EgaliterdesR}, it is also equal to 
$\cup_{n\ge 0} f^n(R)$ and so to $\cup_{n\ge 0} f^n(\Lambda)$ modulo a $\mu$-null set.   

Hence it holds:
\begin{prop} \label{EgaliteracLambda}
Every ergodic probability measure $\mu$,  one of the following conditions holds:
\begin{itemize}
\item Either $\mu$ is supported by $\{A,A'\}$ or $\hat K_\square$ or  $ \cup_{n\ge 0} f^n(\mathcal E)$,
\item Either  $\mu$ is supported by   $\cup_{n\ge 0} f^n(\Lambda)$.  
\end{itemize}
\end{prop}

%

From  Propositions  \ref{lift}  and \ref{Bij}, and the fact that $h$ is a first return map, it holds:
\begin{prop}\label{Bij2}
The first return time of $z\in \Lambda$ is $N_\Lambda(z) $. The map $\ola \sb\colon \Lambda\to \sL$ is a bijection. 
\end{prop}

The following countable subset of $\sA^{(\N)}$ will index the Markov partition of $\Lambda$ and $\sL$:
\[\sS:= \{\sb_0(z)\cdots \sb_k(z): z\in \Lambda, \; n_{\sb_0(z)\cdots \sb_k(z)}=N_\Lambda(z)\}=\{\sa_1\cdots \sa_{N_\sR(\underline \sa)}:\;  \underline \sa=(\sa_i)_{i}\in \ola \sR\}.\]
\index{$\sS$}

Given $\underline \sb\in \sL^s$, there exists $z\in \Lambda$ such that $\underline \sb =(\sb_i)_{i\ge 0}$, 
with $\ola\sb(z)=(\sb_i)_{i\in \mathbb Z}$. Then $\sg_0(z):=(\sb_i)_{0\le i< N_\sL(\ola \sb(z))}$ is in $\sS$. 
Likewise, for every $k\ge 0$, the chain $\sg_k(z):=\sg_0((f^\Lambda)^k( z))$ belongs to $\sS$. We notice that the sequence $\underline \sb(z)$ is the concatenation of the chains $(\sg_k(z))_{k \ge 0}$. We denote such a concatenation by:
\[\underline \sb(z)= (\sb_i)_{i\ge 0}=\sg_0(z)\cdot \sg_1(z)\cdots \sg_k(z)\cdots\]
Hence it makes sense to write $\sL^s\subset \sS^\mathbb N$. 

Similarly we define for $k>0$, $\sg_{-k}(z)= \sg_0((f^{\Lambda})^{-k}(z))$. We notice that $(\sb_i)_{i<0}$ is the concatenation of the chains $(\sg_k(z))_{k < 0}$. We denote such a concatenation by:
\[\overline \sb(z)= (\sb_i)_{i<0}= 
\cdots \sg_{-k}(z)\cdots\sg_{-2}(z)\cdot \sg_{-1}(z)\]
Hence it makes sense to write $\sL^u\subset \sS^{\Z^-}$.
As every sequence in $\sL$ is the canonical concatenation of a presequence in $\sL^u$ with a sequence in $\sL^s$, it holds:
\[\sL \subset \sL^u\cdot \sL^s\subset \sS^{\Z^-}\cdot \sS^\mathbb N= \sS^\mathbb Z\; .\]

 \paragraph{Stable leaves}
Every $z\in \Lambda$ satisfies that $h^{-1}(z)$  belongs to 
$W^s_{\underline \sa\circ h^{-1}(z)}$. 
We recall that $\underline \sa\circ h^{-1}(z)$ is of the form $  \sa_0\cdot \underline \sb(z)$, with $\sa_0\in \sY_0$ and $\underline \sb(z)\in \sL^s$ weakly regular. Hence this curve is  sent by $h$ into   $W^s_{\underline \sb(z)}\ni z$. We put

$$\gamma^s(z)= \gamma^s(\underline \sb(z)):=W^s_{\underline \sb(z)}\;.$$ 
%
%

By Corollary  \ref{geodesvarstable} it comes immediately:
\begin{Claim}\label{gammas}
The curve $\gamma^s(z)$ is $C^{1+Lip}$ close to an arc of parabola.  Moreover its tangent vectors are $\theta^k$-contracted by $f^k$ for every $k\ge 0$. \end{Claim}

\paragraph{Unstable leaves}

We recall that for every $\overline \sb\in \sL^u$ is of the form 
$\overline \sb= (\sg_i)_{i< 0} \in \sS^{\mathbb Z^-} $. Also by  Proposition \ref{Bij2}, there exists  $z\in \Lambda$ which belongs to
$$\gamma^u(\overline \sb):= \bigcap_{ i\ge 1} (f^{\Lambda})^i(Y_{\sg_{-i}\cdots \sg_{-1}}).$$ 
We put $\gamma^u(z)= \gamma^u(\overline \sb(z))$. 
\begin{Claim} \label{gammau}
The curve $\gamma^u(z)$ is flat and stretched. Moreover its tangent vectors are $e^{-ck/3}$-contracted by $f^{-k}$, for every $k\ge0$.
\end{Claim}
\begin{proof}
We notice that $\underline \sa \circ h^{-1}\circ (f^\Lambda)^{-i}(z)$ is of the form  $\sa_{-i-1}\cdot \sg_{-i}\cdot \sg_{-i+1}\cdots \sg_{-1}\cdots $, where $\sm_i:= \sa_{-i-1}\cdot \sg_{-i}\cdot \sg_{-i+1}\cdots \sg_{-1}$ is a regular, sequence, which is suitable from $S^{t\! t}$ and  so that: 
\[h^{-1}\circ (f^\Lambda)^{-i}(z)\in Y_{\sm_i}\;.\]
Furthermore it is complete since $\sg_{-1}$ is complete as every element in $\sS$.  Therefore, by Proposition  \ref{completearepuzzlepiece}, it is a puzzle piece of $S^{t\!t}$. This means that the segment 
$S^{t\! t}_{\sm_i}=Y_{\sm_i}\cap S^{t\! t}$ is sent by $f^{n_{\sm_i}}$ onto the flat stretched curve $S^{t\! t\cdot \sm_i}$.

By Proposition \ref{geoYg}.4, there exists a curve $\mathcal C$ included in $Y_{\sm_i}$ which passes through $z$, intersects the flat stretched curve $S^{t\! t}$ and which is $\theta^{n_{\sm_j}}$-contracted by $f^{n_{\sm_j}}$. 

Consequently, the point $z$ is $\theta^{n_{\sm_j}}$-close to the curve $S^{t\!t\cdot \sm_i}$.

Also, a trivial consequence of Proposition \ref{5.17} is the following:
\begin{lemm}
If $ \sg$ is suitable from $S^{t\! t}$ and 
 $\sg'$ is suitable from 
$S^{t\! t\cdot \sg}$, then the curves $S^{t\! t\cdot \sg'}$ and $S^{t\! t\cdot \sg\cdot \sg'}$ are $\theta^{n_{\sg'}}$-close for the $C^{1+Lip}$-topology.
\end{lemm}
This implies that $(S^{t\! t\cdot \sm_i})_i$ converges to a  curve $\gamma ^u(\overline \sb)$ in the $C^{1+Lip}$-topology and that $z$ belongs to $\gamma^u(\overline \sb)$. 

   Moreover by hyperbolicty of simple and parabolic pieces, 
the vectors tangent to the curve $S^{t\!t\cdot \sm_i}$ are uniformly $e^{-\frac c3 k}$-contracted by $f^{-k}$ for every $k\le n_{\sm_i}$. Thus the vectors tangent to the curve $\gamma^u(z)$ are uniformly $e^{-\frac c3 k}$-contracted by $f^{-k}$ for every $k\ge 0$. 
\end{proof}

\paragraph{Product structure on $\Lambda$}
We recall that $\sL^s\subset \sS^\mathbb N$,  $\sL^u\subset \sS^{\mathbb Z^-}$ and $\sL\subset \sS^\mathbb Z$. 
Actually these inclusions are equalities up to the subsets  made by sequences eventually equal to $\ss_-$:
\[\sA^{(\mathbb N)}\cdot \{\ss_-\}^\mathbb N:=\{(\sa_i)_i\in \sA^{\mathbb N}\colon \exists N\ge 0,\; \sa_i=\ss_-,\; \forall i\ge N\}\; ,\]
\[\sA^{(\mathbb Z)}\cdot \{\ss_-\}^\mathbb N:=\{(\sa_i)_i\in \sA^{\mathbb Z}\colon \exists N\ge 0,\; \sa_i=\ss_-,\; \forall i\ge N\}\; .\]
\begin{prop}\label{markovstructure}\label{Markov}
The following equalities hold:
 \[\sL^s= \sS^\mathbb N\setminus \sA^{(\mathbb N)}\cdot \{\ss_-\}^\mathbb N\quad \mathrm{ and }\quad \sL= \sS^\mathbb Z\setminus \sA^{(\mathbb Z)}\cdot \{\ss_-\}^\mathbb N\quad \mathrm{ and }\quad \sL^u= \sS^{\mathbb Z^-}\; .\]
 \end{prop}
 \begin{proof}
 We remark that if $\sL^s= \sS^\mathbb N\setminus \sA^{(\mathbb N)}\cdot \{\ss_-\}^\mathbb N$ then the inverse limit $\sL$ of $\sL^s $ is equal to the inverse limit $\sS^\mathbb Z\setminus \sA^{(\mathbb Z)}\cdot \{\ss_-\}^\mathbb N$ of $\sS^\mathbb N\setminus \sA^{(\mathbb N)}\cdot \{\ss_-\}^\mathbb N$. This implies also that  $\sL^u= \sS^{\mathbb Z^-}$. 
 
To prove that  $\sL^s= \sS^\mathbb N\setminus \sA^{(\mathbb N)}\cdot \{\ss_-\}^\mathbb N$, it suffices to prove that for every sequence  $(\sg_i)_{i\ge 0} \in \sS^\mathbb N$, there exists $\ss\in \sY_0$ so that $\ss\cdot \sg_0\cdots \sg_N$ is  regular from $S^{t\! t}$  for every $N$; in particular $Y_{ \ss\cdot \sg_0\cdots \sg_N}$ is non empty. Taking the limit $N\to \infty$, we see that it contains a regular point in $Y_\se\setminus W^s(A)$, if the sequence $(\sg_i)_i$ is not eventually equal to $\ss_-$.

First, we notice that by definition of $\sS$, there exists $\ss\in \sY_0$ so that  $\ss\cdot \sg_0$ is regular from $S^{t\! t}$.  As $\sg_0$ is complete, $\ss\cdot \sg_0$ is also complete, and so, by Proposition \ref{completearepuzzlepiece}, it defines a puzzle piece of   $S^{t\! t}$. Then we assume by induction on $N\ge 0$, that $\ss\cdot \sg_0\cdots \sg_N$ is complete and regular from $S^{t\! t}$. In particular  $\ss\cdot \sg_0\cdots \sg_N$ defines a puzzle piece of $S^{t\!t}$. Thus $f^{n_{\ss\cdot \sg_0\cdots \sg_N}}(S^{t\! t}_{\ss\cdot \sg_0\cdots \sg_N})$ stretches across $Y_\se$. By Proposition \ref{geoYg}, it  stretches also across $Y_{\sg_{N+1}}$, since $\sg_{N+1}$ is weakly regular. 
Lemma \ref{TestRadmissible} below implies that $\sg_{N+1}$ is suitable from $S^{t\! t\cdots \ss\cdot \sg_0\cdots \sg_N}$ and so $\ss\cdot \sg_0\cdots \sg_{N+1}$ is suitable from $S^{t\! t}$. Clearly it is also a complete chain. The regularity condition on the orders is straight forward. 
 \end{proof}
\begin{lemm}\label{TestRadmissible}
Every $(\sa_i)_{i\ge 0}\in \sL^s$ is suitable from every $S^t$, with $t\in T^*$ complete (that is   
$t=(\sa_i)_{i\le -1}$ with $\sa_{-1}\in \sY_0$). 
\end{lemm}
 This Lemma will be shown as Corollary \ref{coroporuProp3.2}. \begin{Claim}\label{pourwsa}
Every single symbol in $\sY_0$ is in $\sS$, whereas every non trivial $\sY_0$-chain is NOT in $\sS$:
\[\sY_0\subset \sS\quad \text{and}\quad \cup_{n\ge 2} Y_0^n\cap \sS=\varnothing\; .\]
In particular, for every $k\ge 2$, the word $\sg=(\ss_-)_{0\le i<k}$ does not belong to $\sS$.
\end{Claim}
\begin{proof} 
Let $k\ge 1$, and let   $(\ss_i)_{i=1}^k\in  Y_0^k\cap \sS$.
Then there exist $\ss_0\in Y_0$ and $\underline \sa=(\sa_i)_{i\ge 0}\in \sR$ such that $\sg\cdot \underline \sa\in \sR$ with $
\sg= (\ss_i)_{i=0}^{k-1}$, and $k$ is the first return time of $\sg\cdot \underline \sa$ in $\sR$. However the sequence $\ss_{k-1}
\cdot \underline \sa$ satisfies equality $(\dagger)$. Also $\tilde \sigma(\underline \sa)$ belongs to $\sL^s$. By the above Lemma, as $t\! t\cdot \ss_{k-1}\cdot \sa_0\in T_0\subset T^*$,  it 
The chain $\tilde \sigma(\underline \sa)$ is suitable from $S^{t\! t\cdot \ss_{k-1}\cdot \sa_0}$. Consequently $\ss_{k-1}\cdot \underline \sa$ is suitable from $S^{t
\! t}$ and so regular. Thus $\tilde \sigma^{k-1}(\sg\cdot \underline \sa)\in \sR$ and, by definition of the first return time, it holds $k=1$.
 \end{proof}

For every $\underline \sb\in \sL^s$  and every $\overline \sb\in \sL^u$, by Proposition \ref{markovstructure}, the sequence $\overline \sb\cdot \underline \sb$ belongs to $\ola \sL$. From the geometries of $\gamma^u(\overline \sb)$ and $\gamma^s(\underline \sb)$, there exists a unique intersection point in $\gamma^u(\overline \sb)\cap \gamma^s(\underline \sb)$
which must be the preimage by 
$\ola \sb$ of  $\overline \sb\cdot \underline \sb$.  
This proves that: 
\[\bigcup_{\sL^s} \gamma^s \cap \bigcup_{\sL^u}\gamma^u= \Lambda\]

%
%
%
%
%

\paragraph{Markov Structure $(Y_1)$}
For every $\sg\in \sS$, we define the sets:
\[\sL^s_\sg:=\sg\cdot \sL^s\subset \sL^s\quad \text{and}\quad \sL^u_\sg := \sL^u\cdot \sg\subset \sL^u\;,\]
which means that $\sL^s_\sg$ is formed by the sequences $\underline \sb'$ which begin with $\sg$ and continue with a certain $\underline \sb'\in \sL^s$, and similarly for the presequences in $\sL^u_\sg$.

 Let us define
 \[\Lambda_\sg^s:= \bigcup_{\underline \sb\in \sL^s_\sg} \gamma^s(\underline \sb)\cap \Lambda\quad \text{and}\quad 
\Lambda_\sg^u:= \bigcup_{\overline \sb\in \sL^u_\sg} \gamma^u(\overline \sb)\cap \Lambda\;.\]
\begin{Claim}\label{Y1} For every $\sg\in \sS$, for every $z\in \Lambda_{\sg_0}^s$ it holds:
\begin{enumerate}[(a)]
\item $f^\Lambda(\gamma^s(z))\subset \gamma^s(f^\Lambda(z))$ and $f^\Lambda(\gamma^u(z))\supset \gamma^u(f^\Lambda(z))$ ,
\item  $f^\Lambda(\gamma^s(z)\cap \Lambda)=\gamma^s(f^\Lambda (z))\cap \Lambda_{\sg_0}^u)$ and 
$f^\Lambda(\gamma^u(z)\cap \Lambda_{\sg_0}^s)= \gamma^u(f^\Lambda (z))\cap \Lambda$. 
\end{enumerate}
\end{Claim}
\begin{proof}
a) Let $\underline \sb(z)=\sg_0\cdots \sg_n\cdots $. We recall that $\underline \sb(f^\Lambda(z))=\sg_1\cdots \sg_n\cdots $ and:
\[\gamma^s(z):=\bigcap_{k\ge 0} Y_{\sg_0\cdots \sg_k}=\bigcap_{k\ge 1} Y_{\sg_0}\cap  f^{-n_{\sg_1\cdots \sg_k}}(Y_{\sg_1\cdots \sg_k})\]
which is clearly sent by $f^\Lambda=f^{n_{\sg_0}}$ into
\[\gamma^s(f^\Lambda (z)):=\cap_{k\ge 1}Y_{\sg_1\cdots \sg_k}\;.\]

We have:
\[\gamma^u(z):=\bigcap_{k\ge 1}f^{n_{\sg_{-k}}+\cdots +n_{\sg_{-1}}}(Y_{\sg_{-k}\cdots \sg_{-1}})\supset \bigcap_{k\ge 1}f^{n_{\sg_{-k}}+\cdots +n_{\sg_{-1}}}(Y_{\sg_{-k}\cdots \sg_{-1}})\cap Y_{\sg_0}\; .\]
Hence its image by $f^{\Lambda}$ contains 
\[\bigcap_{k\ge 0}f^{n_{\sg_{-k}}+\cdots +n_{\sg_{-1}}+n_{\sg_{0}}}(Y_{\sg_{-k}\cdots \sg_{-1}\cdot \sg_0})=\gamma^u(f^\Lambda(z))\; .\]
b) 
By the conjugacy $\ola \sb$, it is sufficient to see that $\ola \sb (\gamma^s(z)\cap\Lambda)= \sL^u\cdot (\sg_0 \cdot \sg_1\cdots )$ is shifted to $\ola \sb (\gamma^s(f^\Lambda (z))\cap \Lambda_{\sg_0}^u )= (\sL^u\cdot \sg_0)\cdot (\sg_1\cdots\sg_k\cdots )$ by $\tilde \sigma^\sL$, to prove the first equality.

By the conjugacy $\ola \sb$, it is sufficient to remark that $\overline \sb(z)\cdot  \sL_{\sg_0}^s=\overline \sb(z)\cdot  ({\sg_0}\cdot \sL^s) $  is shifted by $\tilde \sigma^{\sL}$ to  $(\overline \sb(z)\cdot  {\sg_0})\cdot \sL^s$, and that $\overline \sb(z)\cdot {\sg_0}=\overline \sb(f^\Lambda(z))$, to prove the second equality.
\end{proof}

\paragraph{Measures of the closures $(Y_2)$}
From Proposition \ref{markovstructure}, the following equality holds:
\[\sL^s = \bigsqcup_{\sg\in \sS} \sL ^s_\sg\Rightarrow 
\Lambda = \bigsqcup_{\sg\in \sS} \Lambda^s_\sg\;.\] 

This proposition will be shown in section \ref{setE}:
\begin{prop}\label{prop313}
For every $\overline \sb\in \Lambda^u$, 
$cl(\Lambda\cap  \gamma^u(\overline\sb))$ is the union of 
$\Lambda\cap  \gamma^u(\overline\sb)$ with a set of Hausdorff dimension at most $1/\sqrt{M}$. 
\end{prop}
This implies immediately the following:
\begin{Claim}\label{Y2}
For every $\overline \sb\in \Lambda^u$, 
 with $\leb_{\gamma^u(\bar \sb)}$ the Lebesgue measure on the curve $\gamma^u(\bar \sb)$, it holds:
 $$\leb_{\gamma^u(\overline\sb)}(
cl(\Lambda\cap  \gamma^u(\overline\sb))\setminus \Lambda) =0\; .$$
\end{Claim}
\paragraph{Pesin manifolds $(Y_3)$}
We recall that by Claims \ref{gammas} and \ref{gammau},
for every $z\in \Lambda$, it holds for every $n\ge 0$:
\[ \frac 1n \|Df^{-n}|T\gamma^u(z)\|\le {-c/3}\quad
\text{and}\quad 
 \frac 1n \|Df^{n}|T\gamma^s(z)\|\le \log \theta\;.\]
This implies that for every invariant probability  measure $\mu$, for $\mu$- a.e. $z$, the curves $\gamma^u(z)$ and $\gamma^s(z) $ are  respectively Pesin unstable and stable manifolds.  
%
%

\paragraph{SRB measure $(Y_5)$}
We recall that every strongly regular map leaves invariant a unique, physical, ergodic,  SRB  measure $\mu$, as proved in \cite{berhen}.  The unstable Lyapunov exponent $\lambda^u$ of $\mu$ must be greater than $c/3$. By Ledrappier-Young formula \cite{LYI85}, the entropy of $h_\mu$ of $\mu$  is equal to $\lambda^u$.   
 
We will see in Claim \ref{ClaimRE} and Corollary \ref{EntropysurE}, that the complement of the orbit of $\Lambda$ does not support measure with such a large entropy. Hence $\mu$ must be supported by the orbit of $\Lambda$:
\[\mu(\Lambda)>0\;.\]
As $\mu$ is an SRB measure, it is absolutely continuous w.r.t. the Pesin unstable manifold of the point in $\Lambda$. We recall that by $(Y_3)$, the unstable curves $(\gamma^u(\overline \sb))_{\sb\in \sL^u}$ are Pesin local unstable manifolds. Hence there exists a probability measure $\nu$ on $\sL^u$ so that for every Borelian $A\subset \Lambda$:
\[ \mu(A)  = \int_{\overline \sb\in \sL^u} d\leb_{\gamma^u(\overline \sb)}(A\cap  \gamma^u(\overline \sb))d\nu\;.\]
Thus, by the first return time property,  it holds:
\[ \mu(\cup_{n\ge 0} f^n(\Lambda))  = \int_{\overline \sb\in \sL^u} \int_{z\in \gamma^u(\overline \sb)\cap \Lambda}  N_\Lambda(z) d\leb_{\gamma^u(\overline \sb)}  d\nu\;.\]
This proves:
\begin{Claim}\label{Y5}
There exists $\overline \sb\in \sL^u$ so that:
\[0<\int_{z\in \gamma^u(\overline \sb)\cap \Lambda}  N_\Lambda(z) d\leb_{\gamma^u(\overline \sb)}<\infty \; . \]
\end{Claim}

\subsection{Conjugacy of $f|\cup_n f^n(R)$ with a strongly positive recurrent shift}

We can now conjugate the dynamics on $\cup_{n\ge 0} f^n (R)$ with Markov countable, mixing shift without the stable set of a 2-periodic point (corresponding to the stable set of the fixed point $A$). 

Let us recall that a countable shift is defined by a  graph $G$ with vertices $V$ and arrows $\Pi\subset V^2$.  
Let $\Omega_G$ be the set of infinite two-sided sequences $(v_n)_n\in V^\mathbb Z$ such that $(v_n,v_{n+1})\in \Pi$ for every $n$.
The shift map of $\Omega_G$ is denoted by  $\sigma$.\index{$\sigma$} \index{$\Omega_G$}

The  product structure of $\sL=\sS^\mathbb Z$ invites us to consider: 

\[V:= \{(\sg, i) :\; \sg \in \sS,\; 1\le  i\le n_\sg-1\}\sqcup\{\se\}\; ,\]
\[\Pi:= \big\{[\se, (\sg,1)]: \; \sg \in \sS\big\}\sqcup \big\{[(\sg,i),(\sg,i+1)]:\; 1\le i\le n_\sg-2: \; \sg\in \sS\big\}\sqcup\big\{[(\sg,n_\sg-1),\se)]: \; \sg \in \sS\big\}\; .
\]
\begin{Claim}\label{Bij3}
There is a canonical bijection between $\{(v_i)_i\in  \Omega_G:v_0=\se\}$ and $\sS^\mathbb Z$.
\end{Claim}
\begin{proof}
 Indeed, given such a $(v_i)_{i}$,  let $(i_k)_{k\in \mathbb Z}$ be the increasing sequence of integers so that $i_0=0$ and $v_{i_k}=\se$ for every $k$.  As the order of each element of $\sS$ is at least 2, it comes that $v_{i_k+1}$ is of the form $(\sg_k,1)$.  This defines a canonical map:
\[  (v_i)_{i}\in \{(v_i)_i\in  \Omega_G: v_0=\se\}\mapsto (\sg_k)_k\in \sS^\mathbb Z\; . \]
One easily checks that this map is a bijection.
\end{proof}
\begin{rema}
It could have been more natural to consider the canonical graph made by the vertices  $\{(\sg, i) :\; \sg \in \sS,\; 0\le  i\le n_\sg-1\}$. We replaced the vertices $\{(\sg, 0); \sg\in \sG\}$ of this graph by $\se$, in order to make the radius $R^*$ (defined and computed below)  much smaller. 
\end{rema}

Let $\hat A$ be the 2-periodic orbit in $\Omega_G$  corresponding to the $2$-periodic sequence $\cdots \se \cdot (\ss_-,1)\cdot\se \cdot (\ss_-,1)\cdots \in \Omega_G$. The stable set $W^s(\hat A)$ corresponds to the sequence which are eventually equal to $\se \cdot (\ss_-,1)\cdot \se\cdot (\ss_-,1)\cdots $. Put:
$$\Omega'_G:= \Omega_G\setminus W^s(\hat A)\; .$$
\index{$\Omega'_G$}

 By Claim \ref{pourwsa}, every $\ola v=(v_i)_{i\in \mathbb Z}\in \Omega'_G$ is canonically associated to a sequence $\ola \sg= (\sg_i)_i\in \sS^\mathbb Z$ whose concatenation is not  eventually equal to $s_-$.
 
  Hence it is in $\ola \sL$, and so there exists $z\in \Lambda$ so that   $\ola \sg=\ola \sb(z)$. 
\[i(\ola v)=z\; .\]

For every $\ola v\in \Omega'_G$, there exists $k\ge 0$ so that $v_k=\se$. We put:
\[i(\ola v )=f^k\circ i\circ \sigma^{-k}(\ola v )\;.\]

We observe that the following diagram commutes, with  $\sigma$ the shift dynamics of $\Omega_G$:
\begin{displaymath}  \xymatrix{
\Omega'_G\ar[d]_\sigma \ar[r]^i & \mathcal O({\Lambda})\ar[d]_f\\
\Omega'_G              \ar[r]^i & \mathcal O({\Lambda})
}
\end{displaymath}

We notice that the shift $\sigma\colon \Omega_G\to \Omega_G$ is mixing. 

\begin{prop}
The map $i$ is a bijection. 
\end{prop}
\begin{proof} 
It suffices to prove that $\{(v_i)_i\in  \Omega'_G:v_0=\se\}$ is sent bijectively onto $\Lambda$ by $i$.

The restriction of $i$ to this set is the composition of two bijections. 
The set $\{(v_i)_i\in  \Omega'_G:\sg_0=\se\}$ is sent bijectively to $\sS^\mathbb Z\setminus\sA^{(\mathbb Z)}\cdot \{s_-\}^\mathbb N$, by Claims \ref{Bij3} and \ref{pourwsa}. 
The set $\sS^\mathbb Z\setminus \sA^{(\mathbb Z)}\cdot \{s_-\}^\mathbb N$ is equal to $\sL$ by Proposition \ref{Markov}, and $\ola \sb^{-1}$ sents it bijectively onto $\Lambda$ by Proposition \ref{Bij2}. 
\end{proof}

To study the ergodic properties of $\sigma$, we regard for $n\ge 0$: 
\begin{itemize}
\item   the number $Z_n$   of loops from $\se$ in the graph $G$ of length $n$.
\item the number $Z_n^*$  of loops in the graph $G$ of length $n$  passing by $\se$ exactly once.
\end{itemize}

Let $R^*_{G}$ and $R_{G}$ be the convergence radii of the series $\sum_n Z_n^* X^n$ and $\sum_n Z_n X^n$. We remark that $R_G\le R^*_G$.
\begin{defi} \label{defiSPR}
The shift $(\sigma,\Omega_G)$ is \emph{strongly positive recurrent} if $R_G<R^*_G$.
\end{defi}

In Proposition \ref{noentropyailleurs}, we will see that the complement of the orbit of $\Lambda$ does not support any measure of high entropy. 
Since the map $x\mapsto x^2-2$ contains a horseshoe\footnote{For instance the horseshoe encoded by the symbols in $\sY_0$\; .} of entropy close $\log 2$, the same occurs for $f$ (for $M$ large and then $b$ small). Thus  the topological entropy of $f$ is at least close to $\log 2$, and by the Variational Principle, there is a measure of entropy at least close to 2 supported by the obit of $\Lambda$. Looking at the pull back by $i$ of this measure, it follows that the shift leaves invariant a probability measure of entropy at least close to $\log 2$. Since $i$ is a bijection (which is bi-measurable by Claim \ref{homeo} below), the entropy of $\sigma$ is bounded. 
Consequently  by \cite{Gu1969}, the radius $R_G$ is at most equal to $e^{-h_{top}}$ and so at most close to $1/2$.

On the other hand, with $\epsilon=1/\sqrt{M}$, we prove in the sequel:
\begin{prop}\label{propSPR} The convergence radius $R^*_G$ is greater than $e^{-2\epsilon}$. \end{prop}
As the entropy of the shift is at least close to $\log 2$, the latter proposition implies:
\begin{coro}The shift $\sigma$ is strongly positive recurrent.\end{coro}
\begin{proof}[Proof of Proposition \ref{propSPR}]
To bound $R_G^*$ from below, we are going to show that:
\begin{equation}\label{majorationdeZn} Z_m^*\le 2 e^{2\epsilon m},\quad \forall m\ge 2.\end{equation}

We notice that:
\[Z^*_m=\text{Card}\,\{\sg\in \sS:\; n_\sg=m\}.\]

Below we will prove the following upper bound on the number of
 suitable, prime, complete  chains of symbols from $S^t$ of order $m$, among $t\in T^*$:
\begin{lemm}\label{upperboundonP_n}
$$P_m:= \sup_{t\in    T ^*  } \text{Card }\{(\sa_1\cdots \sa_j)\cdot \sa_{j+1}\in (\sA\setminus \sY_0)^{(\mathbb N)}\times \sY_0:\; t\cdot \sa_1\cdots \sa_j\cdot \sa_{j+1}\in    T   ^*   ,\; n_{\sa_1\cdots \sa_j\cdot \sa_{j+1}}=m\}\le 2 e^{\epsilon m}
$$
\end{lemm}

Let $\sg=(\sa_i)_{i=1}^N\in \sS$ be such that $n_\sg=m$. 
We recall that $\sa_N$ belongs to $\sY_0$. 

There are two possibilities. 

Either $\sa_i$ does not belong to $\sY_0$ for every $i<N$. Then  the cardinality of such a possibility is at most~$P_m$. 

Either there exists $i_0<N$ maximal such that $\sa_{i_0}$ belongs to $\sY_0$. 
Put $\sg':= \sa_1\cdots \sa_{i_0}$ and $\underline \sa:= \sa_{i_0+1}\cdots \sa_N$. We remark that 
$\sg'\cdot \ss_+\cdots \ss_+\cdots $ belongs to $\sL^s$, and so 
$\sg'$ belongs to $\sS$. The cardinality of such $\sg'$ is bounded by $Z^* _{n_{\sg'}}$ while the cardinality of such $\underline \sa$ is given by $P_{n_{\underline \sa}}$. The cardinality of such a possibility is bounded from above by $Z^* _{n_{\sg'}}\cdot P_{m-n_{\sg'}}$. By Claim \ref{pourwsa}, the word $\underline \sa \in \sS$ is not a concatenation of symbols in $\sY_0$. Hence $n_{\underline \sa}>M$. Thus by induction:
\[Z_m^* \le P_m+\sum_{k=0}^{m-M-1} Z^*_k \cdot P_{m-k}\le P_m+\sum_{k=0}^{m-M-1} 4 e^{2 \epsilon k} e^{\epsilon(m-k)}\le  2e^{\epsilon m}+4 \frac{e^{\epsilon (m-M)+\epsilon m}}\epsilon.\]
As ${e^{-\epsilon M}}/\epsilon$ is very small and as $2e^{\epsilon m}$ is small w.r.t. $2^{2\epsilon m}$,it comes that 
 $Z_m^*$ is small with respect to $2 e^{2\epsilon m}$, for $m\ge M+1$.\end{proof}
\begin{proof}[Proof of Lemma \ref{upperboundonP_n}]
 The last symbol has an order at most $M$, while the other symbols have an order at least $M+1$. Also, given $(\sa_i)_{i=1}^j$  suitable from $S^{t\! t}$, for each $k\ge 2$, there are at most two symbols $\sa_{j+1}$ in $\sA$ of order $k$ such that $(\sa_i)_{i=1}^{j+1}$ is suitable from $S^{t\! t}$ (this is clear when $\sa_{j+1}$ is parabolic, when it is simple it follows from definitions \ref{Psimple} and \ref{henonsimple}). 

Consequently, it holds $P_m=2$ for $2\le m\le M$ and for $m\ge M+1$:
\[P_m \le 2\sum_{k=M+1}^{m-2} P_{m-k}=2\sum_{k=2}^{m-M-1} P_{k}.\]
Thus if  $M+1\le m\le 2M+1$, these two inequalities imply $P_m\le 4 (m-M-2)\le 2e^{\epsilon m}$. If $m\ge 2M+2$, the induction gives:
 \[P_m \le 2\sum_{2}^{m-M} 2 e^{\epsilon k}\le 4 e^{2\epsilon} \frac{1-e^{\epsilon(m-M-1)}}{1-e^{\epsilon}}\le 4 \frac {e^{\epsilon(m-M+1)}}{\epsilon}\]

This proves that  $P_m$ is less than $2e^{\epsilon m}$, since $m\ge 2M+1$ implies that $e^{\epsilon m}$ is much larger than $M$  
$4 e^{-\epsilon M}/\epsilon$ is very small, by Remark \ref{violent}.
\end{proof}
 Hence -- as explained in the introduction -- it suffices  to prove the holder continuity of $i$ and the continuity of $\underline \sb$ to accomplish the proof that $f|\cup_n f^n(R)$ satisfies all the conclusions of Theorem \ref{Main}.

 \subsection{Hölder continuity of $i$ and continuity of its inverse}\label{Holder continuity of i} 
 
 Given $\underline w=(w_i)_i$, $\underline w'=(w'_i)_i\in \Omega_G'$, let 
 $$\underline w\wedge \underline w' = \sup\{n\ge 0:\; w_i=w_i',\; \forall i\in [-n,n]\}\;,$$
with the convention $\sup\varnothing =-\infty$.  This defines the following metric on $\Omega_G'$:
 \[ d(\underline w, \underline w')=2^{-\underline w\wedge \underline w'}\;.\]
 
\begin{Claim} \label{Holder}
The map $i$ is $\frac c{3\log 2}$-Hölder for the metric $d$.
\end{Claim}
\begin{proof}
 Let $\underline w=(w_i)_i$, $\underline w'=(w'_i)_i\in \Omega_G'$ and put $n:= \underline w\wedge \underline w'$.  The Claim is obvious when $n=\pm\infty$. Let us suppose $n\ge 0$.

Let $m\ge n$ be minimal such that $w_i=w'_i$ for every $-m\le  i\le 0$ and $w_{-m}=w'_{-m}=\se$. 
 
We notice that  there exists $\underline \sa\in \sS^{(\mathbb N)}$ so that  $f^{-m}\circ i(\underline w)$ and $f^{-m}\circ i(\underline w')$
 belong to $Y_{\underline \sa}$ and $n_{\underline \sa}\ge n+m$ .
 
By Lemma \ref{TestRadmissible}, the chain $\underline \sa$ is suitable from $S^{t\!t}$, in particular the pair  $(S^{t\!t}_{\underline \sa},n_{\underline \sa})$
 is well defined and hyperbolic.  This implies that the length of 
 $f^m(S^{t\!t}_{\underline \sa})$ is smaller than $|Y_\se| e^{-(n_{\underline \sa} -m) c/3}\le |Y_\se| e^{-n c/3}$, where $|Y_\se|$ denote the \emph{width of} $Y_\se$ that is the maximal length of a flat stretched curve. 

Furthermore, by Proposition \ref{geoYg}, $f^{-m}\circ i(\underline w)$ belongs to a curve $\mathcal C\subset Y_{\underline a}$ 
and 
$f^{-m}\circ i(\underline w')$ belongs to a curve $\mathcal C'\subset Y_{\underline a}$ 
such that:
\begin{itemize}
\item the length of $f^{k}(\mathcal C)$ and $f^{k}(\mathcal C')$ are smaller than $\theta^k$ for every $k\le m+n$,
\item the curves $\mathcal C$ and $\mathcal C'$ intersect $S^{t\!t}$ at points $z,z'\in Y_{\underline a}$.
\end{itemize}

Thus  the distance between $f^{m}(z)$ and $i(\underline w)$ is smaller than $\theta^m$ and the distance between $f^{m}(z')$ and $i(\underline w')$ is smaller than $\theta^m$.
Also, the distance between  $f^{m}(z)$ and $f^{m}(z')$  is smaller than the length of  $f^m(S^{t\!t}_{\underline \sa})$. Thus:
\[d(i(\underline w),i(\underline w'))\le  2\theta^n+|Y_\se|e^{-n c/3}\le 2 d(\underline w, \underline w')^{\frac{c}{3\log 2}}.\]
\end{proof}
\begin{Claim}\label{homeo} The map $i$ is a homeomorphism from $\Omega_G'$ onto $\cup_n f^n(\Lambda)$.
\end{Claim}
\begin{proof}
By Claim \ref{pouraC0}, the map $\underline \sa\colon \mathcal R\to \sR$ is continuous.  Moreover the map $\tilde \sigma$ is continuous from $\sR$ into $\sL^s$ (since $\tilde \sigma\in C^0(\sA^\N,\sA^\N))$. Also the map $h$ is continuous from $Y_\se\setminus \cup_{\ss\in \sY_0} Y_\ss\supset Y_\se \setminus W^s(A;f)$ into $Y_\se$. Consequently the composition $\underline \sb = \tilde \sigma\circ \underline \sa\circ h^{-1}$ is continuous from $\Lambda$ into $\sL^s$. 


Hence for an elementary closed set of the form $C:=\prod_{i<0} V\times 
\prod_{i=0}^m  V_i\times \prod_{i>m} V\cap \Omega_G'$, with $m>0$ and $(V_i)_i$ closed (finite) subsets of $V$, it comes from the continuity of $\underline \sb$ that
$i(C)$ is a closed subset of $\Lambda$. It is also a closed subset of $\cup_{n\ge 0} f^n(\Lambda)$ by the following Claim:
\begin{Claim}
$\Lambda$ is closed in $\cup_{n\ge 0} f^n(\Lambda)$.\end{Claim}
\begin{proof} It is sufficient to prove that $R$ is closed in   $\cup_{n\ge 0} f^n(R)$ by the continuity properties of $h$. For this end, we recall that the closure of $R$ is included in $\tilde {\mathcal R}\cup W^s(A)$, by Claim \ref{pour scs}. Hence by Proposition \ref{lift} the intersection $cl(R)\cap f^k(R)$ is included in $R$ for every $k\ge 0$.\end{proof}

Also for an elementary closed set of the form $C:=\prod_{i<-m} V\times 
\prod_{i=-m}^m  V_i\times \prod_{i>m} V\cap \Omega_G'$, with $m>0$ and $(V_i)_i$ closed (finite) subsets of $V$, it comes from the continuity of $\sigma^{-1}$ that 
$\sigma^m(C)$ is a closed subset of $\Omega_G'$, which is of the latter form, and so 
$i\circ \sigma^m(C)$ is a closed set. As $f$ is a diffeomorphism; 
$f^{-m}\circ i\circ \sigma^m(C)$ is a closed subset of $\cup_n f^n(\Lambda)$. By commutativity of the diagram, $i(C)$ is a closed subset of $\cup_n f^n(\Lambda)$. Thus $i$ is closed, and so its inverse is continuous. 
\end{proof} 
 
The proof of Proposition \ref{lift} is combinatorial and geometric. It needs a few notions.
\subsection{Definition and properties of the division}

A useful tool introduced in \cite{berhen}  is the right division on the words of the puzzle algebra  $\sG$ (not to be mistaken with $\sS$).

First let us recall that by $(SR_2)$, for every $t\in T^*$ and every $j\ge 0$, 
the common puzzle piece $\sc_j^t(S^{t\!t})$ of depth $j$ is given by a concatenation of complete $\sA$-chains $\underline \sc_j^t= \underline \sa_1  \cdots  \underline \sa_j$.  


The chain $\underline \sc_j^t$ is complete,  suitable from $S^{t\! t}$ and even regular. The word $\underline \sc_j^t$ is called the \emph{$\sA$-spelling} of the puzzle piece $\sc_j ^t(S^{t\!t})$. 
 

We say that $\underline \sa\in \sG$ is \emph{(right) divisible}\index{Right divisibility $/$} by $\underline \sa'\in \sG$ and we wright $ \underline \sa/\underline \sa'$ if  one of the following conditions hold:
\begin{itemize}  
\item[$(D_1)$] ${ \underline \sa}=\underline \sa'$ or $\underline \sa'=\se$, 
\item[$(D_2)$] $ \underline \sa$ is of the form $ \square_\pm (\sc_l-\sc_{l+1})$ and satisfies $ \underline \sc_l/\underline \sa'$, with $ \underline \sc_l$ the $\sA$-spelling of $\sc_l$,
\item[$(D_3)$] there are splittings $ \underline \sa=\underline \sa_3\cdot  \sa_2\cdot \underline \sa_1$ and $\underline \sa'=\underline \sa'_2\cdot \underline \sa_1$ into words $\underline \sa_1, \underline \sa'_2, \underline \sa_3\in\sA^{(\mathbb N)}$ and  $\sa_2\in \sA$,  such that $\sa_2/\underline \sa_2'$ and $n_{\underline \sa_1}+n_{\underline \sa_3}\ge 1$.   
\end{itemize}  

The two last conditions are recursive but the recursion decreases the order $n_{ \underline \sa}$. Thus the right divisibility is well defined by induction on $n_{ \underline \sa}$. 


In Proposition 5.14 of \cite{berhen}, we showed:

\begin{prop}\label{propdediv} The right divisibility  $/$ is an order relation on $\sG$. Moreover for all $\underline \sa$, $\underline \sa'$, $\underline \sa''\in \sG$:
\begin{enumerate}
\item If $\underline \sa/\underline \sa'$ then $n_{\underline \sa}\ge n_{\underline \sa'}$, with equality iff $\underline \sa=\underline \sa'$.
\item If $\underline \sa/\underline \sa'$, $\underline \sa/\underline \sa''$ and $n_{\underline \sa'}\ge n_{\underline \sa''}$ then $\underline \sa'/\underline \sa''$.
\end{enumerate}
\end{prop} 
This allows us to define:
\begin{defi}
The greatest common divisor of $\underline \sa$  and $\underline \sa'$ is the element 
$\underline \sd\in \sG$ dividing both $\underline \sa$ and $\underline \sa'$ with maximal order. 

We write $\underline \sd=: \underline \sa\wedge \underline \sa'$. 
For all $\sa\in \sG$, we put $\nu(\sa,\sa)=n_{\underline \sa\wedge \underline \sa'}$.
\end{defi}

As any $t,t'\in T^*$ are presenquences $t = (\sa_i)_{i\le -1},t' = (\sa'_i)_{i\le -1}\in \sA^{\mathbb Z^-}$, we can define:
\[\nu(t,t')=\sup_{m\le -1} \nu((\sa_i)_{m\le i\le -1},(\sa_i)_{m\le i\le -1})\in\mathbb N\cup\{\infty\}\;.\]

The first application is the following:
\begin{prop}[\cite{berhen} Lemm. 6.1, Prop. 5.17]\label{5.17}
The $C^{1+Lip}$-distance between two curves $S^t, S^{t'}$ is dominated by $b^{\nu(t,t')/4}$, for all $t,t'\in T^*$.

Moreover it holds $\underline \sc_j^t = \underline \sc_j^{t'}$ for every  $j\le \Xi(M+1+\nu(t,t'))$. 
\end{prop}



An application of this proposition is the following:
\begin{coro}\label{coroporuProp3.2}
For every $t\cdot \underline \sg\in T^*$, with $t\in T_0$ and $\underline \sg\in \sA^{(\mathbb N)}$ suitable and complete from $S^t$, for every $\underline  \sa=(\sa_i)_{i\ge 0}\in \tilde \sR$, for every $k\ge 0$, both $\sg \cdot (\sa_i)_{i= 0}^k$ and $\sg \cdot (\sa_{i+1})_{i= 0}^k$ are suitable from $S^{t}$.
In particular, both  $(\sa_i)_{i= 0}^k$ and $(\sa_{i+1})_{i= 0}^k$ are suitable from $S^{t\cdot \sg}$.
\end{coro}
\begin{proof}
First let us observe that  by Proposition \ref{completearepuzzlepiece},  $\underline \sg(S^t)$ is a puzzle piece, thus   $Y_{\sa_1\cdots \sa_k}$ and $Y_{\sa_0\cdots \sa_k}$ intersect $f^{n_\sg}(S^t_{\sg})$ at a non trivial segment. Hence it suffices to show that $(\sa_i)_{i= 0}^k$ and $(\sa_{i+1})_{i= 0}^k$ are suitable from $S^{t\cdot \sg}$. Put ${t'}= t\cdot \sg$.

By definition of the regular sequences, $n_{\sa_0}\le M$ and $n_{\sa_1}$ is of order at most $M+M\Xi$ which is small w.r.t. $\Xi(M+1)\le \Xi(M+1+\nu(t,t\! t\cdot \sa_0))$.  
 Hence $\sa_0$ is in $\sY_0$ and so is suitable from $S^{t'}$, and $\sa_1$ is either in $\sY_0$ (and so is suitable from $S^{t'}$)
  or $\sa_1$ is of the form  $\square_\pm   (\sc^{t\! t\cdot \sa_0}   _i  -\sc^{t\! t\cdot \sa_0}_{i+1})$ with $i+1\le \Xi(M+1+\nu(t,t\! t\cdot \sa_1))$. By Proposition  \ref{5.17}, it comes that $\sc^{t\! t\cdot \sa_0}_i$ and $\sc^{t\! t\cdot \sa_0}_{i+1}$ are equal to respectively $\sc^{{t'}\cdot \sa_0}_i$ and $\sc^{{t'}\cdot \sa_0}_{i+1}$. Thus $\sa_1$ is suitable from $S^{t'}$. 
  
  By induction on $k\ge 0$, we assume that both $(\sa_i)_{i= 0}^k$ and $(\sa_i)_{i= 1}^k$
  are suitable from $S^t$. We recall that $n_{\sa_{k+1}}\le M+\Xi \sum_{i=0}^k n_{\sa_i}$. But $\nu({t'}\cdot \sa_1\cdots \sa_k , t\! t\cdot \sa_1\cdots \sa_k)$ and 
 $\nu({t'}\cdot \sa_0\cdots \sa_k , t\! t\cdot \sa_1\cdots \sa_k)$ are both greater than  
$\sum_{i=1}^k n_{\sa_i}$. Thus applying again Proposition \ref{5.17}, we get that both $(\sa_i)_{i= 0}^{k+1}$ and $(\sa_i)_{i= 1}^{k+1}$  are suitable from $S^{t'}$.
  \end{proof}

\begin{rema}\label{remap40}
The same proof shows that if $\underline \sa\in \sG$ is suitable from a certain $S^t$, $t\in T^*$, and satisfies $(\dag)$, then $\underline \sa$ belongs to $\tilde {\sR}$.
\end{rema}
 A geometric consequence of the division is the following Lemma:

\begin{lemm}\label{premier lemm de division}
For any regular chains $\underline \sa, \underline \sb\in \sA^{(\mathbb N)}$,  
if $\underline \sa/\underline \sb$, then $f^{n_{\underline \sa}}(Y_{\underline \sa})\subset f^{n_{\underline \sb}}(Y_{\underline \sb})$.
\end{lemm}

\begin{proof}We proceed by induction on $n_{\underline \sa}$. 
If $\underline a\in \sY_0$ or $\underline \sb=\se$ then $\underline \sb\in \{\underline \sa,\se\}$. The inclusion is clear.

Otherwise $\underline \sa=\underline \sa_1\cdot \sa_2$, with $\underline \sa_1$ a regular non empty chain and $\sa_2\in \sA$. Let $\underline \sb=\underline \sb_1\cdot \sb_2$ with $\underline \sb_1$ possibly equals to $\se$ and $\sb_2\in \sA$ (and so $\not = \se$). 
As $\underline \sa$ has at least two letters, the rule $(D_2)$ cannot apply directly.   

If $\sa_2\not=\sb_2$, then $(D_1)$ cannot apply directly. It remains only $(D_3)$ with an empty last letter. It implies  $\sa_2/\underline \sb$. As $\underline \sb$ is not $\se$ nor $\sa_2$, it comes that $ \sa_2$ is of the form $\square_\pm(\sc_i-\sc_{i+1})$. At this step we can only apply $(D_2)$ which gives $\underline \sc_i/\underline \sb$. Thus $\underline \sc_i$ and $\underline \sb$ are regular chains, by induction $f^{n_{\underline \sc_i}}( Y_{\underline \sc_i})\subset f^{n_{\sb}} (Y_{\underline \sb})$. As $f^{M+1+n_{\underline a_1}}( Y_{\underline \sa})\subset Y_{\underline \sc_i}$, it holds $f^{n_{\underline a}}(Y_{\underline \sa})\subset f^{n_{\underline \sb}} (Y_{\underline \sb})$.

If $\sa_2=\sb_2$, then $\underline \sb_1$ is a regular chain and by definition of the division (third item), it comes that   $\underline \sa_1/\underline \sb_1$.
 Thus by induction $f^{n_{\underline \sa_1}}(Y_{\underline \sa_1})\subset f^{n_{\underline \sb_1}}(Y_{\underline \sb_1})$. Looking at the image by $f^{n_{\sa_2}}|Y_{\sa_2}=f^{n_{\sb_2}}|Y_{\sb_2}$ of this inclusion, we get the requested inclusion.
\end{proof}

\subsection{Proof of the first return time property (Proposition \ref{lift})}
\label{Proofof proplift}
Let $x\in \cap_{n\ge 0} (f^\mathcal R)^n( {\mathcal R})= R$. Let $N$ be the first return time of $x\in R$ in $\tilde{\mathcal R}$. We recall that $N_R(x)\ge N$ denote the symbolic return time. We want to show $N=N_R(x)$.

Let $\underline \sa(x)= \sb_0\cdots \sb_i\cdots$. Let $i+1$ be minimal such that with $\underline \sb= \sb_0\cdots \sb_{i+1}$, it holds $n_{\underline \sb}\ge N$. We have $N\in (  \sb_0\cdots \sb_{i}, \sb_0\cdots \sb_{i+1}]$.

 We notice that if $N=n_{\underline \sb}$, then by Corollary \ref{coroporuProp3.2}, $\underline \sa(x)= \underline \sb\cdot \underline \sa(f^N(x))$ and so $n=N_R(x)$. 

Hence we can suppose for the sake of a contradiction that $N\in (  n_{\sb_0\cdots \sb_{i}}, n_{\sb_0\cdots \sb_{i+1}})$.


 Put $x'= f^N(x)$. Let $x''\in R$ be a symbolic  backward  return of $x'$: there exists a regular chain $\underline \sa$ so that  
 $x':=f^{n_{\underline \sa}}(x'')$. As there are infinitely many such backward returns, we can suppose $n_{\underline \sa}\ge N$. Put $n:= n_{\underline \sa}-N\ge 0$. 
 
 Observe that $f^n(Y_{\underline \sa}) \cap Y_{\underline \sb}\ni x$. Also $x\notin W^s(A)$ is not in the boundary of $Y_{\underline \sb}$.

Thus we can apply the following lemma:
\begin{lemm}\label{lemme qui torche} Let $\underline  \sa,\; \underline \sb$ be regular chains  and let $n\in [0,n_{\underline \sa}]$ be such that  $f^n(Y_{\underline \sa}) \cap int\,Y_{\underline \sb}\not= \varnothing$ and $n_{\underline \sb}+n\ge  n_{\underline \sa}$. Then 
there exists a (possibly empty) regular chain $\underline \sa'$   such that: 
\[\underline \sa/\underline \sa' \quad\mathrm{and}\quad n+n_{\underline  \sa'}=n_{\underline \sa},\]
 and $\underline \sb$ starts by $\underline \sa'$ (i.e. $\underline \sb=\underline \sa'\cdot \underline \sb'$, with  $\underline \sb'\in \sA^{(\mathbb N)}\cup \{\se\}$).  
\end{lemm}
Therefore $\underline \sa'$ is of the form $\sb_0\cdots \sb_k$, with 
 $n+n_{\sb_0}+\cdots +n_{\sb_k} =n_{\underline \sa}$.
 But $n= n_{\underline \sa}-N$. Thus $n_{\sb_0}+\cdots +n_{\sb_k}=N$. A contradiction.

\begin{proof}[Proof of Lemma \ref{lemme qui torche}]
We proceed by induction on $n_{\underline \sa}$ to show the existence of such an $\underline a'\in \sR$. 

If $\underline \sa\in \sY_0$ then either $n=0$ or $n=n_{\underline \sa}$; take $\underline \sa'=\underline \sa$ or $\se$ respectively. 

 The case $\underline \sa=\square_\pm (\sc_i-\sc_{i+1})$ does not occur since $\underline \sa\in \sR$. 

Let $\underline \sa=\underline \sa_1\cdot \sa_2$ be with $\sa_2\in \sA$ and $\underline \sa_1\in \sA^{(\mathbb N)}$. If $n>n_{\underline \sa_1}$, then 
either $\sa_2\in \sY_0$ and $n=n_{\underline \sa}$; either $\sa_2=\square_\pm (\sc_i-\sc_{i+1})$ and we can use the induction hypothesis on $\underline \sc_i$ which  is regular.

If $n\le n_{\underline \sa_1}$, then by induction there exists a regular chain $\underline \sa_1'$ such that $n+n_{\underline \sa_1'}=n_{\underline \sa_1}$ and $\sa_1/ \sa_1'$. Furthermore, the regular chain $\underline \sb$ starts by $\underline \sa_1'$: there exists $\underline \sb'\in \sA^{(\mathbb N)}$ such that $\underline \sb= \underline \sa_1'\cdot \underline \sb'$. 

Let $\sb_2\in \sA$ be the first letter of $\underline \sb'$. We want to show that $\sb_2=\sa_2$ since $\underline \sa'= \underline \sa_1'\cdot \sa_2$ divides $\underline \sa$ and satisfies $n+n_{\underline \sa'}=n_{\underline \sa}$. 

As $\underline \sb:= \underline \sa_1'\cdot \sb_2\cdots $ is regular, the $\sA$-letter $\sb_2$ is suitable from  $S^{t\!t\cdot \underline \sa_1'}$ and $n_{\sb_2}\le M+\Xi n_{\underline a_1'}$.

As $\underline \sa_1/\underline \sa_1'$, by Proposition \ref{5.17}, the parabolic pieces of $S^{t\!t\cdot \underline \sa_1}$ and $S^{t\!t\cdot \underline \sa_1'}$ with order less than $M+ \Xi n_{\underline \sa_1'}$  are pairwise the same. In particular, $\sb_2$ is suitable from  $S^{t\!t\cdot \underline \sa_1}$ and $\sb_2\in \sP(t\! t \cdot \underline \sa_1)$.

Also  $f^n(Y_{\underline \sa})$ intersects the interior of $Y_{\underline \sb}\subset Y_{{\underline \sa'_1}\cdot  \sb_2}$ and  
\[ f^{n}(Y_{\underline \sa})\cap Y_{{\underline \sa'_1}\cdot  \sb_2}=f^n(Y_{\underline \sa_1}\cap f^{-n_{\underline \sa_1}}(Y_{\sa_2})) \cap Y_{\underline \sa_1'}\cap f^{-n_{\underline \sa_1'}}(Y_{\sb_2})\subset f^{-n_{\underline \sa_1'}}(Y_{\sa_2}\cap Y_{\sb_2}).\]
Thus the interior of $Y_{\sa_2}\cap Y_{\sb_2}$ is not empty and so the partition property of $\mathcal P(t\!t\cdot \underline \sa'_1)$ implies $\sa_2=\sb_2$.\end{proof}

\section{Entropies and Hausdorff dimensions of exceptional sets}
\label{exception}

\subsection{The set $K_\square$}
\label{Ksquare}
We already defined the compact set $K_\square$ as the hyperbolic continuity of the following invariant set of $P^{M+1}$:
\[\check  K_\square:= \bigcap_{n\ge 0} P^{-n(M+1)}(\R_{\square_-(\se-\sc_1)}\cup \R_{\square_+(\se-\sc_1)}). \]
Following the position of the segment $R_{\sc_1}$ with respect to $\R_{\square}$, there are two possibilities: 
\begin{itemize}
\item Either $P^{M+1}|\R_{\square}$ does not intersect $\R_{\square}$, and then $\check  K_\square$ is empty.
\item Either $P^{M+1}|\R_{\square}$ intersects $\R_{\square}$. Then $P^{M+1}|\R_{\square_-(e-\sc_1)}$  (resp. $P^{M+1}|\R_{\square_-(e-\sc_1)}$) is a bijection onto its image which contains $\R_\square\supset \R_{\square_-(e-\sc_1)}\sqcup \R_{\square_+(e-\sc_1)}$.
\end{itemize}
By uniform expansion, in the latter case the entropy of  $P^{M+1}|\check  K_\square$ is equal to $\log 2$, and the number of fixed points of $P^{(M+1)n}|\check  K_\square$ is $2^n$, for every $n\ge 0$. 

By hyperbolic continuity, the same holds for $f^{M+1}|K_\square$. This implies that the following properties for the $f$-invariant set $\hat K_\square:= \cup_{n=0}^{M}f^n(K_\square)$:
\begin{prop}\label{hKsquare}
The entropy of $f$ restricted to $\hat K_\square$ is at most $\log(2)/(M+1)$. The number of fixed points of $f^n$ in $\hat K_\square$ is at most $(M+1)2^{n/(M+1)}$, for every $n\ge 0$.
\end{prop}

By the variational principle, every ergodic probability measure supported by $\hat K_\square$ has an entropy smaller than $\log(2)/(M+1)$, which is small since $M$ is large. 
Hence such measures do not need to be studied to prove Theorem \ref{Main}.

\subsection{The exceptional set $\mathcal E$}
\label{setE}

By Propositions \ref{eventually regu} and \ref{hKsquare}, it comes:
\begin{Claim}\label{ClaimRE}
For every invariant, ergodic  measure $\mu$ with entropy greater than $\log(2)/(M+1)$, it holds  $\mu(\mathcal R)>0$ or $\mu(\mathcal E)>0$.
\end{Claim}

In order to show that the invariant probability measures satisfying $\mu(\mathcal E)>0$ have small entropy we show in \textsection \ref{appendixA1}:
\begin{prop}\label{HDexceptionel}
For every curve $\gamma$ transverse to the stable direction of $\mathcal E$, the Hausdorff dimension of $\gamma \cap E$ is at most $1/\sqrt{M}$ .
 \end{prop}

\begin{coro}\label{EntropysurE} Any ergodic probability measure $\mu$ so that $\mu(\mathcal E)>0$ has an entropy at most $\log 4/ \sqrt M$ 
\end{coro}
\begin{proof}[Proof of Corollary \ref{EntropysurE}] 
%

By ergodicity, the support of $\mu$ is supported by the orbit of $\mathcal E$.


The idea is to use 
 Ledrappier-Young formula. 
First, as the measure $\mu$ is hyperbolic, for $\mu$-almost every $x$, we can define the conditional measure $\mu_x$ associated to a Pesin local unstable manifold of $x$. Then Ledrappier-Young formula (Thm C  \cite{LYII85}) states that for $\mu$-almost every $x$, 
with $W^u_{\epsilon}(x)$ the Pesin unstable local manifold of diameter $\epsilon$ and $\lambda_u$ the unstable Lyapunov exponent of $\mu$, it holds:
$$h_\mu =\limsup_{\epsilon\to 0}  \frac{\log \mu_x( W^u_{\epsilon}(x))}{\log \epsilon} \cdot \lambda_u\, .$$

As Lebesgue differentiation Theorem holds for every finite Borelian measure (for a proof see for instance \cite{LYI85}), for $\mu$-a.e. $x\in \mathcal E$, the measure $\mu(\mathcal E\cap W^u_{\epsilon}(x)))$ is equivalent to $\mu( W^u_{\epsilon}(x))$ as $\epsilon\to 0$. Hence for $\mu$ a.e. $x\in \mathcal E$, it holds: 
 $$h_\mu =\limsup_{\epsilon\to 0}  \frac{\log \mu( W^u_{\epsilon}(x)\cap \mathcal E)}{\log \epsilon} \cdot \lambda_u\, .$$

By Prop 7.3.1 \cite{LYII85},  $\limsup_{\epsilon\to 0}  \frac{\log \mu( W^u_{\epsilon}(x)\cap \mathcal E)}{\log \epsilon}$ is bounded by the Hausdorff dimension of 
$W^u_{\epsilon}(x)\cap \mathcal E$. Consequently, it comes that:
$$
h_\mu \le d_{HD}(W^u_{\epsilon}(x)\cap \mathcal E) \cdot \lambda_u\, .$$
To achieve the proof of the Corollary, it suffices to bound the right hand terms of this inequality. 
First we recall that $\lambda_u$ is at most $\log 4$.  Also $\tilde {\mathcal R}$ is foliated by Pesin unstable manifolds, which are $\mu$ a.e. transverse to the Pesin stable manifold $W^u_\epsilon(x)$. Hence we deduce from Proposition \ref{HDexceptionel} that $d_{HD}(W^u_{\epsilon}(x)\cap \mathcal E)\le 1/\sqrt M$. Consequently $h_\mu \le \log 4/\sqrt M$.

%
%

\end{proof} 

Furthermore, in \textsection \ref{appendixA2}, we will show the following:
\begin{prop}\label{cardperE}
The number of fixed points of $f^n$ in $\cup_{n\ge 0} f^n(\mathcal E)$ is at most $ne^{n/\sqrt M}$.
\end{prop}


We are now ready to prove Proposition \ref{prop313}.  Let us show that for every $\overline \sb\in \Lambda^u$, 
$cl(\Lambda\cap  \gamma^u(\overline\sb))$ is the union of 
$\Lambda\cap  \gamma^u(\overline\sb)$ with a set of Hausdorff dimension at most $1/\sqrt{M}$. 
 \begin{proof}[Proof of Proposition \ref{prop313}]
 We recall that $R=h^{-1}(\Lambda)$ is included in $\tilde {\mathcal R}= \mathcal R\cup \mathcal E$. Moreover, by Claim \ref{pour scs}, 
 the closure of $\tilde {\mathcal R}$ is included in 
$\tilde {\mathcal R}\cup \bigcup_{\underline \sa\; \text{regular chain}}\partial^s Y_{\underline \sa}\; .$
 
As $h$ is a closed map from $\cup_{\ss\in \sY_0} Y_\ss\setminus \partial^sY_\ss$ into $Y_\se\setminus \partial^s Y_\se$, we have:
  \[cl(\Lambda) \subset h( {\mathcal R})\cup  h( {\mathcal E})\cup   \bigcup_{\underline \sa\; \text{weakly regular chain}}\partial^s Y_{\underline \sa}\; .\] 
By the geometry of the box of weakly regular chains (Prop. \ref{geoYg}), the set    $\bigcup_{\underline \sa\; \text{weakly regular chain}}\partial^s Y_{\underline \sa}$ is a countable union of curves which intersects the flat stretched curve $\gamma^u(\overline \sb)$ at a countable set. Hence this union has Hausdorff dimension equal to 0. 

On the other hand, $h( {\mathcal E})$ is an union of curves of the form $W^s_{\underline \sb}$, with $\underline \sb$ weakly regular. By the same Proposition, these curves intersect $\gamma^u(\overline\sb))$ uniformly transversally.  Hence $h^{-1}( \gamma^u(\overline\sb))$ is uniformly transverse to  $ {\mathcal E}$. Hence,  by Proposition \ref{HDexceptionel}, the Hausdorff dimension of $ {\mathcal E}\cap  h^{-1}(\gamma^u(\overline\sb))$ and so 
$h( {\mathcal E})\cap  \gamma^u(\overline\sb)$ is at most $1/\sqrt M$. 
   \end{proof}
\subsection{Proof of Theorem \ref{Main}}
Since the map $x\mapsto x^2-2$ contains a horseshoe of entropy close to $\log 2$, the same occurs for $f$ (for $M$ large and then $b$ small). Hence the number of fixed points of $f^n$ is at least $(2-\eta)^n$, for $\eta>0$ small. With respect to this quantity, the number of periodic points which intersect the exceptional set $\mathcal E$ or $K_\square$ is small (see Prop.  
\ref{hKsquare} and \ref{cardperE}). Using the bounds by respectively $\frac{\log  4}{\sqrt M}$ and $\frac{\log 2}{M+1}$ on the entropy of ergodic probability measures supported partially by $\mathcal E$ or $K_\square$ (see Prop. \ref{hKsquare} and \ref{EntropysurE}),  the dichotomy of Proposition \ref{EgaliteracLambda} implies:

\begin{prop}\label{noentropyailleurs}
 Every ergodic probability measure of entropy greater than $\log 4/\sqrt M$ is supported by 
 $\cup_{n\ge 0} f^n(\Lambda)$.
 
  Moreover the following sequence of atomic measures converges to $0$: 
$$\frac 1{Card \, Fix(f^n)} \sum_{z\in Fix \, f^n,\; z\notin \cup_{n\ge 0} f^n(\Lambda)} \delta_z\to 0$$ 
\end{prop}

As explained in the introduction, this accomplishes the proof of  Theorem \ref{Main}.

\subsection{Proof of the upper bound on the dimension of $\mathcal E$}
\label{appendixA1}
In this section we prove Proposition \ref{HDexceptionel} which bounds from above the Hausdorff dimension of the intersection of $\mathcal E$ with a transverse curve. We use a similar method to the one of \cite{Se03} in the quadratic map context.

We recall that $\sE:= \tilde \sR\setminus \sR$ and $\mathcal E:= \cup_{\sc\in \sE} W^s_{\sc}$.


We recall that $\cup_{\sc\in \tilde \sR} W^s_\sc$ is a disjoint union of Pesin stable manifolds, that we call \emph{long local stable manifolds}. 

A same proof as for Proposition 3.9 of \cite{berhen} shows that the tangent space to the stable manifolds of $( W^s_\sc)_{\sc\in \sR}$ is a $3$-Lipschitz function of $z\in \cup_{\sc\in \tilde \sR} W^s_\sc$. This implies that the holonomy along this lamination is Lipschitz. 

Hence to prove Proposition \ref{HDexceptionel}, it suffices to prove that the Hausdorff dimension of $S^t\cap \mathcal E$ is smaller than $1/\sqrt M$ for any flat stretched curve $S^t$, $t\in T^*$. To fixe the idea, we take $t=t\!t$ and define $K_*:=S^{t\!t}\cap \mathcal E$. 
 
To show that the Hausdorff dimension of $\mathcal E$ is small, we work with a family of nice coverings of $K_*$.  This covering is given by $\sA$-chains which are suitable for $S^{t\!t}$. For this end we define: 
\begin{defi} For every $t\in T^*$, let $\sD(t)$ be the set of $\sA$-chains $\underline \sa$ which are suitable from $S^t$ and so that $t\cdot \underline \sa$ belongs to $T^*$:
\[ \sD(t)=\{\underline \sg\in \sA^{(\mathbb N)}:\quad t\cdot \underline \sg\in T^*\}\; .\]
\end{defi}

It will be usefull to bound from above $\#_N:= \sup_{t\in T^*} Card\{ \underline \sg\in \sD(t):\; n_{\underline \sg} =  N\}$.

\begin{lemm}\label{card de Pj} For every $N\ge 0$, it holds $\#_N\le 2^N$.
\end{lemm}
\begin{proof}[Proof of Lemma \ref{card de Pj}] 
We proceed by induction on $N$. It follows from the describition of the simple pieces that $\#_1=0$, $\#_2=2$ and $\#_3=2$.

Let us suppose $N\ge 4$. As for all $t\in    T^*$ and $n\ge 0$, there are at most two letters in $\sA$ suitable from $S^t$ and with  order $n$, an induction gives:
\[\#_N\le 2+\sum_{n=2}^{N-2} 2 \#_{N-n}\le 2+\sum_{n=2}^{N-2} 2^{n+1}\le 2^{N}\; . \] 
\end{proof}

We recall that $|Y_\se|$ denote the maximal length of a flat stretched curve.

\begin{prop}\label{calculdHD} For every $N$, there exists $\sC_N\subset \sD(t\! t)$ such that:
\begin{itemize} 
\item[$(i)$] $ \{S^{t\! t}_{\underline \sa}:\; \underline \sa\in \sC_N\}$ covers $K_*$,
\item[$(ii)$] every  $\underline \sa\in \sC_N$ has its order  $n_{\underline \sa}\ge N$ and so $|S^{t\! t}_{\underline \sa}|\le |Y_\se|e^{-n_{\underline \sa}c/3}\le |Y_\se | e^{-c N/3 }$,
\item[$(iii)$] $ \sum_{\underline \sa\in \sC_N} e^{-s n_{\underline \sa}c/3}< \lambda^N$, with $\lambda:=e^{- M^{1/4}}$ and $s:=1/\sqrt{M}$.
\end{itemize} 
\end{prop}

An immediate consequence is the following:
\begin{coro}\label{p3}
The set $K_*$ has Hausdorff dimension smaller than $s=1/\sqrt M$ (which  is small for $M$ large).
\end{coro}
\begin{rema}Actually the same estimate holds for every  $S^{t}\cap \mathcal E$, among $t\in    T^*$.
\end{rema}

To find the set of words $\sC_N$ given by Proposition  \ref{calculdHD}, we consider the following subset of words of $ \sD(t)$, for $t\in T^*$:
\[{\sM}(t):= \bigcup_m\{\sa_1\cdots \sa_m\in\sD(t):\; n_{\sa_m}> M+\Xi \sum_{k< m} n_{\sa_k}\}.\]
We define also:
\[{\sM}:= \bigcup_{t\in T^*} \sM(t)\; .\]\index{$\sM$}

\begin{lemm}\label{mauvaiseorthographe}
For every  $\underline \sa\in \sE$,  
there exist $\underline  \sa'\in \sD(t\!t)$ and 
$(m_i)_{i\ge 0}\in {\sM}^\mathbb N$ such that:
\[\underline \sa= \underline \sa' \cdot \sm_1\cdots \sm_n\cdots\]
\end{lemm}
\begin{proof}[Proof of Lemma \ref{mauvaiseorthographe}] If $\underline \sa$ does not come back infinitely many times into $\tilde \sR$ by the shift $\tilde \sigma\colon \sA^{(\mathbb N)}\to \sA^{(\mathbb N)}$, then for every $n$ large, let 
$N\ge 0$ be such that for every $n\ge N$,  $\tilde \sigma^n( \underline \sa( z))$ does not belong to $\tilde \sR$. 
Let $\underline \sa'$ be the word formed by the $N-1$ first $\sA$-letters of $\underline \sa$. Let $\underline \sa_0\in \sA^\N$ be such that $\underline \sa= \underline \sa'\cdot \underline \sa_0$. We remark that $\tilde \sigma^n(\underline \sa_0)\notin \tilde \sR$ for every $n\ge 0$. By Remark \ref{remap40}, there exists $\sm_1\in {\sM}$  such that $\underline \sa_0:= \sm_1\cdot \underline \sa_1$, for a certain $\underline \sa_1\notin \tilde \sR$. 
Moreover $\tilde \sigma^{n}(\underline \sa_1)\notin \tilde \sR$, for every $n\ge 0$, and so we can  write $\underline \sa_1:= \sm_2\cdot \underline \sa_2$ with $\sm_2\in {\sM}$  and $\underline \sa_2\notin \tilde \sR$. And so on the proposition follows by induction. \end{proof}

  
\begin{proof}[Proof of Proposition \ref{calculdHD}]

For every $N$, let:
\[\sC'_N:= \{\underline \sa'\cdot \sm_1\cdots \sm_N\in \sE:\; \sm_i\in {\sM},\; \underline \sa'\in \sD(t\!t),\; n_{\underline \sa'}\le N\}.\]  
We remark that $\sC_N:= \cup_{N'\ge N} \sC_{N'}'$ satisfies $(i)$ by Lemma \ref{mauvaiseorthographe}. Property $(ii)$ holds by the hyperbolicity of $\underline \sa(S^t)$, $\underline \sa\in \mathcal D(t)$ (see Prop. \ref{geoYg}).

 Let us show $(iii)$. For $s>0$, put
\[\Psi_N(s):= \sum_{\underline \sa\in \sC'_N} e^{-s n_{\underline \sa} c/3}.\]

For $t\in    T^*$, let ${\sM}(t)$ be the set of words $\sm\in {\sM}$ so that $t\cdots \sm \in T^*$: $\sM(t)=\sM\cap \sD(t)$. For $N\ge 1$, we put $\sM^N(t)=\sM^N\cap \sD(t)$.

\begin{lemm}\label{CardMt}
We have for every $N\ge 1$ :
\[\sum_{\sm\in {\sM}^N(t)} e^{- n_{\sm} \frac c {3\sqrt M}}\le { M^N}e^{ - N \sqrt M \frac c 3}.\]
\end{lemm}
\begin{proof}[Proof of Lemma \ref{CardMt}]
 The word $\sm$ can be of the form $\square_\pm (\sc_i-\sc_{i+1})$ or ${\underline \sa'} \cdot \square_\pm (\sc_i-\sc_{i+1})$ with $n_{ \sc_i}\ge \Xi n_{{\underline \sa'}}$. In both cases the order is at least $M+1$. In the first case or in the second case with ${\underline \sa'}$ fixed,  there are only two possible parabolic pieces for each order. Consequently:

\[\sum_{\sm\in {\sM}(t)} e^{-s n_{\sm}\frac c 3}\le 
\sum_{j\ge M+1} 2 e^{-s j \frac c 3}+\sum_{j\ge 1} 2 \text{Card}\{{\underline \sa'}\in \sD(t):\; n_{{\underline \sa'}}=j\} \sum_{k\ge \Xi j} e^{- s (k+j)  \frac c 3}.\]

By Lemma \ref{card de Pj}, it comes:
\[\sum_{\sm\in {\sM}(t)} e^{-s n_{\sm}\frac c 3}\le \frac{ 2e^{-s (M+1) \frac c 3}}{1- e^{-s \frac c 3}}+ \sum_{j\ge 1} 2^{j+1} e^{-s j\frac c3} \frac{e^{- s\Xi j \frac c 3}}{1-e^{-s \frac c3}}.\]

For $s= M^{-1/2}$, since $M$ is large and $\Xi= e^{\sqrt M}$, we get:
\[\sum_{\sm\in {\sM}(t)} e^{- n_{\sm} \frac c {3\sqrt M}}\le { M}e^{ -  \sqrt M \frac c 3}.\]
This proves the Lemma for $N=1$. Also for $N=2$:
\[\sum_{\sm_1\cdot \sm_2\in {\sM^2}(t)} e^{- n_{\sm} \frac c {3\sqrt M}}\le 
\sum_{\sm_1\in {\sM}(t)} e^{- n_{\sm_1} \frac c {3\sqrt M}}
\sum_{\sm_2\in {\sM}(t\cdot \sm_2)} e^{- n_{\sm_2} \frac c {3\sqrt M}}
\le 
({ M}e^{ -  \sqrt M \frac c 3})^2.\]
And similarly we get the Lemma for any $N\ge 1$.  
\end{proof}
Hence: 
\[\Psi_N(s) \le \text{Card}\{{\underline \sa'}\in \sD(t\!t):\; n_{{\underline \sa'}}\le N\} \cdot {M^{N}}  e^{ - N \sqrt M \frac c 3}.\]

By Lemma \ref{card de Pj}, it comes: 
\[\Psi_N(s)\le  N2^{N}  {M^{N}}   e^{ - N  \sqrt M \frac c 3}.\]
And so:
\[ \sum_{\underline \sa\in \sC_N}  e^{- s n_{\underline \sa} \frac c3}\le \sum_{n\ge N} n2^{n}  {M^{n}}  e^{ - n \sqrt M \frac c 3}\le e^{- M^{1/4} N}.\]
\end{proof} 
\subsection{Cardinality of periodic cycle which intersect $\mathcal E$ }
\label{appendixA2}

Proposition \ref{cardperE} follows from an encoding of the set of periodic orbits intersecting $\mathcal E$ with the set $\sM ^{(\mathbb N)}$. 

We split the proof into several Lemmas and Propositions.
\begin{prop} 
If $x\in \tilde {\mathcal R}$ is periodic, then $\underline \sa(x)=:(\sa_i)_i\in \sA^\N$ is preperiodic.
\end{prop}
\begin{proof}

If there exists $j$ such that $p=n_{\sa_0\cdots \sa_j}$ then
$\sa_{n+j}=\sa_n$ for every $n$. Thus $\underline \sa(x)$ is periodic and $x$ belongs to $\mathcal R$ and even in $R$. 

Otherwise there exists $j\ge 0$ such that $p\in ( n_{\sa_0\cdots \sa_{j-1}}, n_{\sa_0\cdots \sa_{j}})$. By Lemma \ref{lemme qui torche}, there exists $i< j$ such that:
\[ \sa_0\cdots \sa_{j}/\sa_0\cdots \sa_{i}\quad \mathrm {and}\quad p+n_{\sa_0\cdots \sa_{i}}= n_{\sa_0\cdots \sa_{j}}\; .\]
As $ n_{\sa_0\cdots \sa_{j-1}}<p<n_{\sa_0\cdots \sa_{j}}$, it holds $n_{\sa_j}\ge n_{\sa_0\cdots \sa_{i}}$. As $\sa_0\cdots \sa_{j}/ \sa_j$, by Proposition \ref{propdediv}, we have  $\sa_{j}/\sa_0\cdots \sa_{i}$. 

Using the same Lemma, it holds that there exists $i'>i$ such that 
\[ \sa_0\cdots \sa_{j+1}/\sa_0\cdots \sa_{i'}\quad \mathrm {and}\quad p+n_{\sa_0\cdots \sa_{i'}}= n_{\sa_0\cdots \sa_{j+1}}\]
 \[\Rightarrow  \sa_{j+1}/\sa_{i+1}\cdots \sa_{i'} \quad \mathrm{and} \quad 
 n_{\sa_{i+1}\cdots \sa_{i'}}=n_{\sa_{j+1}}\]
By Proposition \ref{propdediv}, it comes that  $\sa_{i+1}\cdots \sa_{i'}={\sa_{j+1}}$. By uniqueness of the $\sA$-spelling, $\sa_{j+1}=\sa_{i+1}$.

And so on, $\sa_{i+k}=\sa_{j+k}$, for every $k\ge 1$. From this it comes that $(\sa_i)_{i\ge 1}$ is equal to the preperiodic sequence:
\[\sa_0\cdot \sa_1 \cdots \sa_i\cdot\underline \sm\cdot\underline \sm\cdots \underline \sm\cdots ,\quad \text{with } \underline \sm:= \sa_{i+1}\cdot \sa_{i+2} \cdots \sa_{j-1}\cdot \sa_{j}\]
\end{proof}

The above proof showed that for every periodic $x\in \mathcal E$, it holds that $\underline \sa(x) = \sa_0\cdot \sa_1 \cdots \sa_i \cdot \underline \sm \cdot \underline \sm \cdots \underline \sm \cdots $, with $\sa_j$ the last letter of $\underline \sm$ and:
\[n_{\underline \sm}=p\qquad \text{and}\qquad  \underline \sm /\sa_j/\sa_0\cdot \sa_1 \cdots \sa_i\; .\]

  The chain $\underline \sm$ is very irregular:


\begin{lemm} There exists $\sm_1,\dots ,  \sm_k\in {\sM}$ so that equal to $\underline \sm = \sm_1\cdots   \sm_k$. \end{lemm}
\begin{proof} The word $\underline \sm$ cannot be regular since otherwise $x$ would be in $\mathcal R$. Let $\sm_1\in \sA^{(\N)}$ be the minimal world which is not regular and so that $\exists \underline \sm'\in  \sA^{(\N)}$ satisfying  
 $\underline \sm=\sm_1\cdot \underline \sm'$. 

Note that $\underline \sm'$ can be empty but $\sm_1$ cannot. If $\underline \sm'$ is empty then we are done: $\underline \sm=\sm_1\in \sM$.  By remark \ref{remap40}, $\sm_1$ belongs to $\sM$.

Suppose that $\underline \sm'$ is regular. We recall that $\sa_{j}/\sa_0\cdot \sa_1 \cdots \sa_i$ and so $n_{\underline \sm'}\ge n_{\sa_{j}}\ge n_{\sa_0\cdot \sa_1 \cdots \sa_i}$.
As, by regularity, $n_{\sa_{i+1}}$ is at most $M+\Xi n_{\sa_0\cdot \sa_1 \cdots \sa_i}\le M+\Xi n_{\underline \sm'}$.

Furthermore $\sa_{j+1}=\sa_{i+1}$, and so $\underline \sm'\cdot  \sa_{i+1}$ is regular.
The same argument implies that $\underline \sm'\cdot  \sa_{i+1}\cdots \sa_m$ is regular. Thus $\underline \sm'\cdot  \sa_{i+1}\cdots \sa_m\cdots=\underline \sm'\cdot  \underline \sm\cdots \underline \sm\cdots$ belongs to $\tilde \sR$ and $\underline \sa(x)$ belongs to $\sR$. A contradiction.

Thus, by remark \ref{remap40},  we can write $\underline \sm$ in the form $\sm_1\cdot \sm_2\cdot \underline \sm''$ with $\sm_1$, $\sm_2\in \sM$.  Again by the full argument, we show that  $\underline \sm''$ is not regular, and so on, it follows that $\underline \sm$ belongs to ${\sM}^{(\mathbb N)}$.
\end{proof}

The above Lemma defines the following canonical map: 
\[\underline \sm\colon x\in  Per_f \cap \mathcal E\mapsto \sm_1 \cdots \sm_k\in {\sM}^{(\mathbb N)}.\]
Let us prove the following:
\begin{lemm} If $\underline \sm(x)=\underline \sm(x')$ for $x,x'\in Per_f\cap \tilde{\mathcal  R}\setminus {\mathcal  R}$, then the periodic orbits of $x$ and $x'$ are equal. \end{lemm}
\begin{proof}
We have $\underline \sa(x)$ and   $\underline \sa(x')$ of the form:
\[\underline \sa(x)=\sg \cdot \underline \sm \cdot \underline \sm \cdots \underline \sm \cdots \quad \text{and}\quad \underline \sa(x')=\sg' \cdot \underline \sm \cdot \underline \sm \cdots \underline \sm \cdots \qquad 
\text{with }\underline \sm / \sg\text{ and }\underline \sm / \sg'\; .\] 

Let us suppose for instance that $n_ \sg\ge n_ {\sg'}$. Then by 
Proposition \ref{propdediv}, it holds $\sg/\sg'$. By Lemma \ref{premier lemm de division}, It comes that $f^{n_\sg}(Y_\sg)\subset f^{n_{\sg'}}(Y_{\sg'})$. Thus $f^{n_\sg}(W^s_{\underline \sa(x)})\subset f^{n_{\sg'}}(W^s_{\underline \sa(x')})$. 
Consequently $f^{n_{\sg}-n_{\sg'}}(x)$ belongs to the stable manifold  $W^s_{\underline \sa(x')}$ of $x'$. As $x$ and $x'$ are periodic points their orbits are equal. 
 \end{proof}

Hence Proposition \ref{cardperE} is a consequence of the following:
\begin{prop}
\[Card \, \{\underline \sb\in  \bigcup_{t\in T^*,\; N\ge 0} {\sM}^{N}(t):\;n_{\underline \sb}=p\}\le e^{-\frac p{\sqrt M}}\; .\]
\end{prop}
\begin{proof}
First let us notice that a direct consequence of Lemma \ref{CardMt} is:
\begin{Claim}\label{418}
For every $t\in T^*$, it holds: 
\[Card\{\underline \sb\in  \bigcup_{N\ge 1} {\sM}^{N}(t):\;n_{\underline \sb}=p\}\le e^{p\frac c{3\sqrt M}}\sum_{N\ge 1} M^N e^{-N\sqrt M \frac c3}\;.\]
\end{Claim}

There are infinitely many $t\in T^*$, but most of them have the same $\sA$-letter of order $p$. 

We recall that Proposition \ref{5.17}, for every $t,t'\in T^*$, 
we have $\underline \sc^t_j= \underline \sc^{t'}_j$ for every $j\le \Xi (M+1+\nu(t,t'))$.
 Thus the parabolic pieces $\square_\pm(\sc^t_j-\sc^t_{j+1})$ and $\square_\pm(\sc^{t'}_j-\sc^{t'}_{j+1})$ are equal for every $j\le \Xi (M+1+\nu(t,t'))-1$. This proves :
\begin{Claim}
For all $t,t'\in T^*$, if $m\le \Xi(M+1+\nu(t;t'))+M$, then the following sets are equal:
\[\{ \underline \sa \in \mathcal D(t):n_{\underline \sa}\le m\}
=
\{\underline \sa \in \mathcal D(t'):n_{\underline \sa}\le m\}\; .\]

\end{Claim}
An immediate consequence of the latter Claim and the division rules $(D_3)$ is: 
\begin{Claim}
For all $t,t'\in T^*$, for every $N\le \Xi (M+1+\nu(t,t'))+M$, the sets $\{\underline \sa\in \mathcal D(t):\; n_{\underline \sa}\le N\}$ and $\{\underline \sa\in \mathcal D(t'):\; n_{\underline \sa}\le N\}$ are equal.  
\end{Claim}
Thus $\{\underline \sb\in  \bigcup_{N\ge 1} {\sM}^{N}(t):\;n_{\underline \sb}=p\}$ and $\{\underline \sb\in  \bigcup_{N\ge 1} {\sM}^{N}(t'):\;n_{\underline \sb}=p\}$ are equal if $$\nu(t,t')\ge \max(0, (p-M)/\Xi -M-1)=:\beth(p)\; .$$

By Proposition 6.5 \cite{berhen}, for every $k\ge 1$, \[P_{M^2 k}:=\{t\! t\cdot \sg\in T^*:\; n_{\sg}\le M^2k\}\]
satisfies that for every $t\in T^*$ there exists $t'\in  P_{M^2 k}$ so that $\nu(t,t')\ge k$. 

Note that $Card\, P_{M^2 k}\le M^2 k 2^{M^2 k }$ by Lemma 
\ref{card de Pj}. 
Consequently, by Claim \ref{418}, the cardinality of the $p$-periodic orbits is at most: 
\[Card \, \{\underline \sb\in  \bigcup_{t\in T^*,\; N\ge 1} {\sM}^{N}(t):\;n_{\underline \sb}=p\}\le \sup_{t\in P_{M^2 \beth(p)}} Card\{\underline \sb\in  \bigcup_{N\ge 1} {\sM}^{N}(t):\;n_{\underline \sb}=p\} Card\, P_{M^2 \beth(p)}\]
\[\le  e^{p\frac c{3\sqrt M}}\sum_{N\ge 1} M^N e^{-N\sqrt M \frac c3}  M^2 \beth(p) 2^{M^2 \beth(p)  }\le e^{\frac p{\sqrt M}}\; .\]
\end{proof} 
\begin{appendix}
\section{Proofs involving the existence of a Lyapunov exponent}
\label{prinvLyapExp}

\begin{proof}[Proof of Proposition \ref{evenregudim1}]
We saw that every invariant probability measure $\nu$ has a Lyapunov exponent at least equal to $c/3$ in Proposition \ref{lyapinvmeasure}. 

For the sake of a contradiction, assume  that $z$ is not eventually regular and that $\sa_i(z)\not=\square$ for some $i$. 

To simplify, we denote by $(\sa_i)_i$ the sequence $(\sa_i(z))_i$ associated to $z$. By replacing $z$ by an iterate, we can suppose that $\sa_1=\square $ and $\sa_2\not = \square$. We recall that $n_\square=M+1$. 

Let $(i_j)_{j\ge 1}$ be the increasing sequence of integers defined by $\sa_{k}=\square$ iff $k=i_j$. Note that $i_1=1$.

Put $N_1=0$ and for $j\ge 2$, put $N_j:= \sum_{i_{j-1}\le l< i_{j}} n_{\sa_l}$. Let $n_j:= N_1+\cdots + N_j$ be the $j^{th}$-irregular return  time and let   $z_{n_j}:= P^{n_j}(z)$ be the $j^{th}$ irregular return of $z$.
We prove below the following:
\begin{lemm}\label{tempsexpo}
For every $j$, the point $z_{n_j}$ belongs to $\R_\sa$ with $n_\sa\ge \frac {\Xi}{M} n_j$ with $\sa\in \sA$.
\end{lemm}
Hence $\sa$ must be of the form $\sa=\square_\pm (\sc_k-\sc_{k+1})$. The segment $\R_\sa$ has length at most $2 e^{c n_{\sa}/3}$. The segment joining $\R_\sa$ to 0 is filled by segments of the
 form $\R_{\square_\pm (\sc_l-\sc_{l+1})}$. Hence the modules of the points in $\R_\sa$ is at most $\sum_{m\ge n_\sa}  2 e^{-c \cdot m/3}=:Cst \cdot e^{-c\cdot  n_\sa/3}$.
 
Hence at $z_j:= P^{n_j} (z)$: 
\[\log \frac{|\partial_x P(z_j)|}{2Cst}< -\frac c3 \cdot n_\sa\le - \frac c3  \frac {\Xi}{M}n_j\;. \]
%
%
On the other hand, 
\[\log (|\partial_x P^{n_j} (z)|)\le  4 \cdot n_j\]
Consequently $|\partial_x P^{n_j+1} (z)|$ is very small, and since $n_j$ is arbitrarily large, this contradicts the fact that the Lyapunov exponent of $\nu$ is at least $c/3$.   
\end{proof}
\begin{proof}[Proof of Lemma \ref{tempsexpo}]
First let us show by  induction that for every $j$, $\sa_{i_{j}+1}$ is not $\square$. First we recall that $\sa_2 = \sa_{i_1+1} \not = \square$. Let $j\ge 2$. By induction we assume that $\sa_{i_{j-1}+1}\cdots \sa_{i_j-1}$ is a non-empty regular chain. Thus $z_{n_j} =0$ or $z_{n_j}\in \R_{\square_\pm (\sc_k-\sc_{k+1})}$ with $$n_{\sc_k}+M+1>M+\Xi n_{\sa_{i_{j-1}+1}\cdots \sa_{i_{j}-1}}=M+\Xi(N_j-M-1)\; .$$

If $z_{n_j} =0$ then $f^{M+1}(z_{n_j})$ belongs to a common piece of arbitrarily high order by $(SR_1)$. Hence $f^{M+1}(z_{n_j})$ is regular, and so $z$ is eventually regular which is a contradiction.
  
If $z_{n_j}\in \R_{\square_\pm (\sc_k-\sc_{k+1})}$ with $n_{\sc_k}+M+1>\Xi$ (and so $k>1$), then $f^{M+1}(z_{n_j})$ belongs to a simple piece. Thus the symbol $\sa_{i_{j}+1}$ is not $\square$. This proves the induction. 

Moreover, for every $j$, the point $z_{n_j}$ belongs to a set $\R_\sa$, with $\sa$ of the form ${\square_\delta (\sc_k-\sc_{k+1})}$ satisfying:
\[ n_{\sa}\ge M+1+\Xi (N_j-M-1)\; .\]
As $N_j-M-1\ge 2$, by the mean value theorem, it comes $n_\sa\ge (\frac{\Xi}{M}+1) N_j$. 

By condition $(\star)$ on the common sequence, the puzzle piece $\sc_k$ is a product of simple and parabolic pieces $(\sb_i)_{i=1}^m$ which forms a regular chain with $m\ge k$. Hence $(\sb_i)_{i=1}^m$ is equal to the first symbols of $\sa_{i_j+1}\cdots \sa_{i_{j+1}-1}$.  Thus it comes 
$$N_{j+1}=M+1+n_{\sa_{i_j+1}\cdots \sa_{i_{j+1}-1}}\ge M+1+n_{\sb_1\cdots \sb_m}= n_{\sa}$$
 and:
\begin{equation}\label{majorationgeo2}
N_{j+1}\ge   n_{\sa}\ge \Big(\frac {\Xi}{M}+1\Big) N_j\ge  \frac {\Xi}{M}N_{j}+N_{j}\ge \frac {\Xi}{M}\sum_{l=1}^{{j}} N_l=\frac {\Xi}{M} n_j\; .\end{equation}
\end{proof}

\label{sectionPropeventuallyregu}
\begin{proof}[{\bf Proof of Proposition \ref{eventually regu}}]
Let $\mu$ be an ergodic measure with support off $\{A,A'\}$.
By Lemma \ref{lyappos},  $\mu$ has one non-negative  Lyapunov exponent.

 This implies that for $\mu$-almost every point $z\in Y_\se$,  there exists a unit vector $u$ such that:
\begin{equation}\label{weak hyper}
\| D_zf^n (u)\|\ge e^{-c^+ n/2}, \quad \text{ for every $n\ge 0$ large enough}.	
\end{equation}
Suppose, for the sake of a contradiction, that $z$ is not eventually regular and $\sa_i(z)\not=\square$ for some $i$. 

To simplify, we denote by $(\sa_i)_{i\ge 0}$ the sequence $\underline \sa (z) =(\sa_i(z))_i$. By replacing $z$ by an iterate, we can suppose that $\sa_0=\square $ and $\sa_1\not = \square$. 

Let $(i_j)_{j\ge 1}$ be the increasing sequence of integers defined by $\sa_{k}=\square$ iff $k=i_j$. Note that $i_1=0$.

Put $N_0=0$ and for $j\ge 1$, put $N_j:= \sum_{i_{j-1}\le l< i_{j}} n_{\sa_l}$. Let $n_j:= N_1+\cdots + N_j$ be the $j^{th}$-irregular return  time and let   $z_{n_j}:= f^{n_j}(z)$ be the $j^{th}$ irregular return of $z$. 
 
By the same proof as Lemma \ref{tempsexpo} (therein one just replaces $\R$ by $Y$), it comes that $f^{n_j}(z)$ belongs to 
$(f^{M+1}|Y_\square )^{-1}(Y_{\sc_k})$, with $\sc_k=\sc_k^{t_j}$ the common piece of depth $k$ of $S^{t_j}$ with $t_j= t\! t\cdot \sa_{i_{j-1}+1}\cdots \sa_{i_{j}-1}$ and satisfies:
%
%
%
%
%
%
\begin{equation}\label{majorationgeo}
N_{j+1}\ge  M+1+n_{\sc_k}\ge \Big(\frac \Xi M +1\Big) N_j\ge \frac \Xi M n_j
 \end{equation}
 
Basically, the idea of the rest of the proof is the same as for Proposition \ref{evenregudim1}: we are going to show that $\|Df^{2n_j}(u)\|< e^{-c^+ n_j}$ which is a contradiction with Inequality (\ref{weak hyper}).

Neverthesless it is slightly more complicated since we deal with the two dimensional case. The idea is to compare the expansion of $Df^{n_j}(w)$,
with $w:= Df^{n_j}(u)/\|Df^{n_j}(u)\|$,  to the contraction at a point $\zeta$  which is the two dimensional equivalent of the quadratic critical point in dimension 1.

First let us notice that by $(SR_1)$, the curve 
$f^{M+1}(S^{t_j}\cap Y_\square)$ is tangent to $W^s_{\sc^{t_j}}$. 
Let $\zeta\in S^{t_j}\cap Y_\square $ be the preimage by $f^{M+1}|Y_\square$ of the tangency point.  We notice that 
$\zeta$ is in $S^{t_j}_{\square \sc_k}$ where:
\[S^{t_j}_{\square \sc_k}= (f^{M+1}|S^{t_j}_\square)^{-1}(Y_{\sc_k })=
 cl(\cup_{m\ge k} S^{t_j}_{\square_\pm \sc_m^{t_j}})\;.\] 
By hyperbolicity of the parabolic pieces, the length of 
$S^{t_j}_{\square \sc_k}$ is at most $Cst\cdot e^{-\frac c3 (M+1+ n_{\sc_k})}$, where $Cst=|Y_\se|/(1-e^{-\frac c3 })$ is a real number independent to $n_j$. 

By definition, a unit vector $v$ which is tangent to $S^{t_j}$ at $\xi$ is sent by $f^{M+1}$ to a tangent vector to $W^s_{\sc^{t_j}}$. Since $W^s_{\sc^{t_j}}$ is $\theta^m$-contracted  by $f^m$ for every $m\ge 0$, it comes:
\[\|D_{\xi}f^{n_j}(v)\|\le \theta^{n_j-M-1} e^{c^+(M+1)}\;.\]
Since $n_j$ is very large, this upperbound is very small, in particular it holds:
 \begin{equation}\label{expansioncritique}
  \|D_{\xi}f^{n_j}(v)\|\le e^{- 2c^+n_j}/2\;.
  \end{equation}

In order to compare $\|D_{\xi}f^{n_j}(v)\|$ to $\|D_{z_{n_j}}f^{n_j}(w)\|$ we prove in the sequel the following lemma:

\begin{lemm}\label{lemmapospone}
For every $j\ge 2+N$, there exist $z'_{n_j}\in S^{t_{{j}}}$ and a unit vector $w'\in T_{z'_{n_j}} S^{t_{j}}$ (i.e. tangent to $S^{t_{{j}}}$ at $z'_{n_j}$) such that the angle between $w$ and $w'$ is $\theta^{n_j/5}$-small and furthermore $z_{n_j}$ and $z'_{n_j}$ are $\theta^{n_{j}/2}$-close.
\end{lemm}

As $z_{n_j}$ is in $(f^{M+1}|Y_\square)^{-1}(Y_{\sc_k })$, it comes that $z_{n_j}'\in S^{t_j}$ is $2 \theta^{n_{j}/2}$ close to 
$S^{t_j}_{\square \sc_k} := (f^{M+1}|S^{t_j}_\square)^{-1}(Y_{\sc_k })$. Since $\zeta$ belongs to $S^{t_j}_{\square \sc_k}$, using the estimate on the length of  $S^{t_j}_{\square \sc_k}$ it comes that the distance  between $z'_{n_j}$ and $\zeta$ is at most $2 \theta^{n_{j}/2} +Cst\cdot e^{-\frac c3 (M+1+ n_{\sc_k})}$.
By the flatness of the curve $S^{t_j}$ and Lemma \ref{lemmapospone} it comes:
\begin{Claim} The distance  between $z_{n_j}$ and $\zeta$ is at most $3 \theta^{n_{j}/2} +Cst\cdot e^{-\frac c3 (M+1+ n_{\sc_k})}$. 

The angle between $w$ and $v$ is at most $2\theta^{n_j/5}+  e^{-\frac c3 (M+1+ n_{\sc_k})}$.
\end{Claim}

A classical computation gives:
\[\|D_{z_{n_j}} f^{n_j}(w)-D_{\zeta} f^{n_j}(v)\|\le \|D_{z_{n_j}} f^{n_j}-D_{\zeta} f^{n_j}\|+\|D_{z_{n_j}} f^{n_j}\|\cdot \|w-w'\|\]
\[\le n_j \| Df\|^{2(n_j-1)} \|D^2f\| d(z_{n_j}, \zeta) + \|Df\|^{n_j}\cdot \|w-w'\|\le (n_j e^{2n_jc^+}+e^{c^+ n_j})   (5\theta^{n_j/5}+  2e^{-\frac c3 (M+1+ n_{\sc_k})} )\]

Then it follows from Inequality (\ref{majorationgeo}) and then (\ref{expansioncritique}) that:
\[\|D_{z_{n_j}} f^{n_j}(w)-D_{\zeta} f^{n_j}(v)\|\le \frac 12 e^{-3 c^+ n_j}\Rightarrow \|D_{z_{n_j}} f^{n_j}(w)\|\le e^{-2 c^+ n_j}\]
Consequently, since $\|D_{z} f^{n_j}(u)\|\le e^{c^+ n_j}$, it holds $ \|D_{z} f^{2n_j}(u)\|\le e^{- c^+ n_j}$.
A contradiction with (\ref{weak hyper}).
\end{proof}

\begin{proof}[Proof of Lemma \ref{lemmapospone}]
The point $z_{n_{j-1}+M+1}:=f^{n_{j-1}+1}(z)$ is $i_j-i_{j-1}-1$-regular. 
Indeed,  $\sg:= \sa_{i_{j-1}+1}\cdot \sa_{i_{j-1}+2}\cdots \sa_{i_{j}-1}(z)$ is regular and consists of the first letters of $\underline \sa(z_{n_{j-1}+M+1})$.
In Proposition \ref{geoYg}.4, we saw that $z_{n_{j-1}+M+1}$ belongs to a curve $\mathcal C$ which satisfies the following properties:
\begin{itemize}
\item[(i)] For every $k\le n_\sg=  N_j-M-1$, $\diam f^k(\mathcal C)\le \theta^k$.
\item[(ii)] The curve $\mathcal C$ intersects every flat stretched curve.
\item[(iii)] The curve $\mathcal C$ is included in $Y_\sg$.   
\end{itemize} 

By (ii), there exists a point $z'\in \mathcal C\cap S^{t\! t}$. By (iii), the point $z'$ belongs to $S^{t\!t}_\sg$. Thus $z'_{n_j}:=f^{n_\sg}(z')$ belongs to $S^{t_{j}}$.  By (i), the distance between $z'_{n_j}$ and $z_{n_j}$ is less than $\theta^{n_\sg}$. By (\ref{majorationgeo}), we have:
\[(1+\frac{M}{\Xi}) N_j \ge   n_j.\]

As $j\ge 2$, by (\ref{majorationgeo}),  $N_j$ and so $n_j$ is large with respect to $M$, thus:
\begin{equation}\label{ngnj}   n_\sg =N_j -M-1> n_j/2. \end{equation}
   It follows that the distance between $z'_{n_j}$ and $z_{n_j}$ is less than $\theta^{n_j/2}$.
 
 Take a unitary vector $u'$ tangent at $z'$ to $S^{t\! t}$. 
For every $k\ge 0$, put $u_k:= D_zf^k(u)$ and  $u'_k:= D_{z'}f^k(u')$.
We notice that $u'_{n_j}:= D_{z'}f^{n_\sg}(u')$. 
 To evaluate the angle between $u'_{n_j}$ and $u_{n_j}$, we regard the formula:
\[|\sin \angle(u'_{n_j},u_{n_j})|= \frac{\|u'_{n_j}\times u_{n_j}\|}{ \|u'_{n_j}\|\cdot \|u_{n_j}\|}\; .\]
Put $n':= [n_\sg/2]+1$. Let $x:= f^{-n'}(z_{n_j})$ and $x':= f^{-n'}(z'_{n_j})$. 
We have
\[|\sin \angle(u'_{n_j},u_{n_j})|
\le \frac{|\det (D_{x} f^{n'})|\cdot \|u'_{n_j-n'}\times u_{n_j-n'}\|}{\|u'_{n_j}\|\cdot \|u_{n_j}\|}
+ \frac{ \|D_{x} f^{n'}-D_{x'} f^{n'}\|\cdot \|u'_{n_j-n'}\|}
{\|u'_{n_j}\|}.\]
Let us study the first term of this sum. Since the determinant is less than $b$, $|\det (D_{x} f^{n'})|\le b^{n'}$.
By $h$-times property of $Y_\sg$ (Proposition \ref{geoYg}.3), $\|u'_{n_j-n'}\|/\|u'_{n_j}\|\le e^{- n'c/3}$. 
By Inequality (\ref{weak hyper}), as $j$ and so $n_j$  are large enough, it comes:
\[\|u_{n_j}\|\ge e^{-c^+ n_j/2}\|u_0\|;\quad \|u_{n_{j}-n'}\|\le e^{ c^+ (n_{j}-n')}\|u_{0}\|\Rightarrow\frac{\|u_{n_j-n'}\|}{\|u_{n_{j}}\|}\le  e^{c^+ n_j/2+c^+ (n_{j}-n')}\le e^{2 c^+ n_j}.\]
Consequently: 
\[\frac{|\det (D_{x} f^{n'})|\cdot \|u'_{n_j-n'}\times u_{n_j-n'}\|}{\|u'_{n_j}\|\cdot \|u_{n_j}\|}\le
b^{n'}  e^{2 c^+ n_j}  e^{- n'c/3}.\]
Using again $h$-times property, the second term is bounded from above by $ \| D_{x} f^{n'}-D_{x'} f^{n'}\| e^{-n'c/3}$.
A classical computation gives, $\|D_{x} f^{n'}-D_{x'} f^{n'}\|\le (n'+1) e^{2n'c^+} \theta^{n'}$.  Therefore:
\[|\sin \angle(u'_{n_j},u_{n_j})|\le b^{n'}  e^{2 c^+ n_j}  e^{- n'c/3} + (n'+1)e^{n'c^+ - n'c/3} \theta^{n'}.\]
By (\ref{ngnj}), we have $n_j> n_\sg >n_j/2$ it comes:
\[|\sin \angle(u'_{n_j},u_{n_j})|\le b^{n_j/4}  e^{2 c^+ n_j} + (n_j+1)e^{n_jc^+ - n_j c/12} \theta^{n_j/4}\]
Hence the angle between $w= u_{n_j}/\|u_{n_j}\|$ and $w'= u'_{n_j}/\|u'_{n_j}\|$ is smaller than $\theta^{n_j/5}$.
\end{proof}

\begin{proof}[{\bf Proof of Proposition \ref{EgaliterdesR}}]
We are going to prove that for every invariant ergodic measure $\mu$, the subsets $\cup_{n\ge 0} f^n(  R )=\cup_{n\ge 0} f^n( \cap_m (f^\mathcal R)^m(\mathcal  R ))$ and $\cup_{N\ge 0} \cap_{n\ge N} f^n(\mathcal R)$  are equal $\mu$-almost everywhere. 

The first subset is contained clearly in the second one.  By ergodicity of $\mu$ and invariance of $\cup_{N\ge 0} \cap_{n\ge N} f^n(\mathcal R)$, we can suppose that the latter has full measure. This implies that the measure of $\mathcal R$ is positive. 

First let us notice that there are points in $\cap_{N\ge 0} \cup_{n\ge N} f^n(\mathcal R)$ which are not in $\cup_{n\ge 0} f^n(  R )$. 
This is the case for instance of a point in $\cap_{n\ge 0} \mathcal R\cap f^{n_{\sg_i}}(Y_{\sg_i})$ with $\sg_i=\sc_1\square_-(\sc_1-\sc_2)\cdots \square_-(\sc_1-\sc_2)$. We notice that this point does not even belong to $\cup_{n\ge 0}f^n((f^{\mathcal R})^2(\mathcal R))$.

We shall first prove the set of that such points has measure $0$.
\begin{Claim}\label{preaprouver}
For every $q\ge 0$, the subset $\mathcal R\cap \cup_{n\ge 0} f^n((f^{\mathcal R})^{q}(\mathcal R))$ is off full measure in $\mathcal R$. 
\end{Claim}
\begin{proof} 
Let us recall that every $\underline \sg\in \sR$ is canonically split as a concatenation of regular chains $(\sg_i)_{\ge 0}$: $\underline \sg = \sg_0\cdots \sg_i\cdots$.
For every $q$, there are countably many $q$-upplets of regular chains $\sg_0\cdots \sg_q$. This defines a countable partition of $\sR$: $\sR =\bigsqcup_{\sg_0,\dots ,\sg_q}\sR_{\sg_0\cdots \sg_q}$, where $\sR_{\sg_0\cdots \sg_q}$ is the set of $\underline \sg\in \sA^{\mathbb N}$ which begins with $\sg_0\cdots \sg_q$. This defines also a countable partition of $\mathcal R$. 
\[\mathcal R =\bigsqcup_{\sg_0,\dots ,\sg_q}\mathcal R_{\sg_0\cdots \sg_q},\]
where $\mathcal R_{\sg_0\cdots \sg_q}$ is formed by points $x$ in $\mathcal R$ so that $\underline \sa(x)\in \sR_{\sg_0\cdots \sg_q}$.  As this union is countable, there exists a certain $q$-upplets $\sg_0\cdots \sg_q$ so that $\mathcal R_{\sg_0\cdots \sg_q}$ has positive measure. 
By Poincaré recurrence Theorem, for $\mu$-a.e. $x\in \mathcal R$, there exists infinitely many $n\ge 0$ so that $f^{-n}(x)$ belongs to $\mathcal R_{\sg_0\cdots \sg_q}$. This implies the Claim. \end{proof}
Unfortunately it is not sufficient to conclude that $\mu$-a.e. point $x\in \mathcal R$ are in  $\cup_{n\ge 0} f^n(\cap_{q\ge 0}(f^\mathcal R)^q(\mathcal R))= \cup_{n\ge 0} f^n(R)$. 
To get such an inequality we shall prove the following Claim:
\begin{Claim}\label{Aprouver}
There exists $N'\ge 0$  so that $\cup_{j=0}^{N'}f^{-j}(R)$ has a positive $\mu$-measure.
\end{Claim}
This implies that $R$ has positive measure and so that its orbit has full measure by ergodicity. This is the statement of Proposition \ref{EgaliterdesR}. 
\end{proof}
\begin{proof}[Proof of Claim \ref{Aprouver}]
To prove the Claim we use the following lemma shown below:
\begin{lemm}\label{A6}
There exists a measurable function $N\colon \mathcal R\to \N$ to so that for $\mu$-a.e. $x\in \mathcal R$, for every $\sb=(\sb_i)_i\in \sR$ and $m$ so that  $x\in f^m(W^s_{\sb})$, the following bound holds.
\[n_{\sb_p}\le N(x)\quad\text{ where $p$ is so that }\quad 
n_{\sb_0\cdot \sb_1\cdots \sb_{p-1}}
< m\le n_{\sb_0\cdots \sb_{p}}\;.\]
\end{lemm}
%

By Lemma \ref{lemme qui torche} 
with $\underline \sa(x)=(\sa_i)_{i\ge 0}$, there exists $i_0$ so that 
\[\sb_0\cdots \sb_{p}/ \sa_0\cdots \sa_{i_0},\quad n_{\sa_0\cdots \sa_{i_0}}=n_{\sb_0\cdots \sb_{p}}-m\qquad \text{and}\qquad \sb_{p+k}= \sa_{i_0+k}\quad \forall k\ge 0\; .\]
Thus by the above Lemma, $n_{\sa_0\cdots \sa_{i_0}}\le N(x)$. 

Let $i_1\ge 0$ be maximal such that $n_{\sa_0\cdots \sa_{i_1}}\le N(x)$. We notice that $i_1$ does not depend on $\underline \sb$ but only on $\underline \sa(x)$ and $N(x)$, and so only on $x$ as a measurable function. Let $i_2\ge i_1$ be minimal so that $\tilde \sigma^{i_2}(\underline \sa)\in \sR$. 
We notice that $\hat N(x):= n_{ \sa_0\cdots \sa_{i_2}}$ depends only on $x$; it is a measurable function of $x$. 

We proved that for every $\underline \sb\in \sR$ and $m\ge 0$ satisfying
$x\in f^m(W^s_{\underline \sb})$, there exist $k\le \hat N(x)$ and $l\ge 0$ so that $f^k(x)\in (f^\mathcal R)^l(W^s_{\underline \sb})$. By Claim \ref{preaprouver}, the integer $l$ can be supposed arbitrarily large. Thus:
\[x\in \bigcup_{k\le \hat N(x)} f^{-k}((f^\mathcal R)^q(\mathcal R) )\quad \forall q\ge 0\Rightarrow x\in \bigcup_{k\le \hat N(x)} f^{-k}(R) )\; .\]
Since $\hat N$ is measurable, there exists $N'$ large enough so that $ \bigcup_{k\le N'} f^{-k}(R)$ has positive measure.  

\end{proof}
\begin{proof}[Proof of Lemma \ref{A6}]
We are going to use an argument based on the convergence of the Lyapunonv exponent. Put $m':=m-n_{\sb_0\cdots \sb_{p-1}}$. We notice that $m'\le n_{\sb_p}$.  

 With $\sigma\ge c/3$ the Lyapunov exponent of $\mu$ and $E^u_x$ the unstable direction at $x$, for every $\eta>0$ small, there exists $C(x)>0$ so that for every $j \in \mathbb Z$:
\[-\frac{C(x)}{|j|} +\sigma -\eta\le  \frac1j \log \|Df^j |E^u_x\|\le \frac{C(x)}{|j|} +\sigma +\eta\; .\]
Thus 
\[-C(x) -(\sigma +\eta) m'\le   \log \|Df^{-m'} |E^u_x\|\le C(x) -(\sigma-\eta) m' \; . \]
\[-C(x) +(\sigma -\eta) m\le   - \log \|Df^{-m} |E^u_x\|\le C(x) +(\sigma+\eta) m \; . \]
Consequently:
\begin{equation}\label{eqA6}
-2C(x) + \sigma (m-m')  -\eta(m+ m')\le   \log \|Df^{m-m'} |E^u_{x_{-m}}\|\le 2C(x)+ \sigma (m-m') +\eta(m+ m') \; . 
\end{equation}
 Let us assume that $n_{\sb_p}\ge 101(M+1)$ (otherwise $N=101(M+1)$ is a suitable bound) and so that $\sb_p$ is a parabolic symbol of the form $\square_\pm(\sc_k-\sc_{k+1})$, with $n_{\sc_k}\ge 100 M$. Also we suppose $m$ large:  $m\ge 2M+4$. 
 
With $m'':= M+1+n_{\sc_k}/M$, we have $M m''\ge n_{\sb_p} \ge m'$. Similarly to (\ref{eqA6}), with $\eta'= (M+1)\eta$, it holds:
\begin{equation}\label{eqA7}
\left\{\begin{array}{rl}
-2C(x) +(\sigma -\eta)m'' \le   \log \|Df^{m''} |E^u_{x_{-m'}}\|
& \text{if } m''-m'\ge 0.\\
-2C(x) +(\sigma-\eta') m''\le -2C(x) +\sigma m'' -\eta (m'+m'')\le   \log \|Df^{m''} |E^u_{x_{-m'}}\|
& \text{if } m''-m'\le 0.\\
\end{array}\right.
\end{equation}

From the regularity property $(\star)$, we have:
\[m'\le n_{\sb_p}\le M+1+\Xi n_{\sb_0\cdots \sb_{p-1}}, \quad 
\text{and}\quad 
m\le n_{\sb_0\cdots \sb_{p}}\le M+1+(\Xi+1) n_{\sb_0\cdots \sb_{p-1}}\;.\]
By the latter inequality:
\[m-m'=n_{\sb_0\cdots \sb_{p-1}}\ge \frac{m-M-1}{\Xi+1}\ge \frac{m}{2\Xi}\quad \text{and so}\quad  \frac{m+m'}{m-m'}\ge 2\Xi. \]

Thus with $\eta<\sigma/ (2\Xi)$ it holds $\sigma (m-m')  -\eta(m+ m')\ge0$ and so by (\ref{eqA6}):
\begin{equation}\label{eqA8}
-2C(x) \le   \log \|Df^{m-m'} |E^u_{x_{-m}}\| \; . 
\end{equation}


Let $x'_{-m}$ be the intersection point between $W^s_{\sb}$ and $S^{t\!t}$  and $u'_{-n}$ be a unit vector tangent to $S^{t\! t}$ at $x'_{-m}$. For $i\in \mathbb Z$, let $x_{i}:=f^{i}(x)$, $x'_{i}:=f^{m+i}(x'_{-m})$,   $u_{i}:= D_{x}f^{i}(u)$ and $u'_{i}:= D_{x'_{-m}}f^{m+i}(u')$. 
\begin{sublemm}\label{lesublemma}
We have the following bounds on the distance and the angle: 
\[d(x_{-m'},x'_{-m'})\le \theta^{m-m'}\quad \text{and} \quad 
|\angle(u_{-m'},u'_{-m'})|\le (e^{4c^+}\theta)^{(m-m')/2}e^{2C(x)}\;.\]
\end{sublemm}
\begin{proof}By $\theta^j$-contraction of $W^s_{\underline \sb}$ by $f^j$ (see Prop. \ref{geoYg}), it comes  for every $ i\le m$:
\[d(x_{-i},x'_{-i})\le \theta^{m-i}\Rightarrow d(x_{-m'},x'_{-m'})\le \theta^{m-m'} \; .\]
By Lemma 14.10 of \cite{berhen}, we have the following bound:
\[\|u_{-m'}\times  u'_{-m'}\| \le 4(e^{4c^+}\theta)^{(m-m')/2}\|u_{-m}\|\cdot \|u'_{-m}\|\;.\]
Consequently, the angle between $u_{-m'}$ and $u'_{-m'}$ is dominated by:
\[(e^{4c^+}\theta)^{(m-m')/2}\frac{\|u_{-m}\|\cdot \|u'_{-m}\|}{\|u_{-m'}\|\cdot \|u'_{-m'}\|}\; .\]
By the $h$-time property of $S^{t\!t}_{\sb_0\cdots \sb_{p-1}}$ we have $\|u'_{-m}\|\le \|u'_{-m'}\|
 \; .$
By (\ref{eqA8}), we have $\|u_{-m}\|\le e^{2C(x)}\|u_{-m'}\|
 \; .$
%
\end{proof}
Let $\zeta$ be the point of $S^{t\!t \cdot \sb_0\cdots \sb_{p-1}}$ which is sent by $f^{M+1}|Y_\square$ to a tangency point with $W^s_{\sc^{t\!t \cdot \sb_0\cdots \sb_{p-1}}   }$. 
Let $w$ be a unit vector tangent to $S^{t\!t \cdot \sb_0\cdots \sb_{p-1}}$ at $\zeta$. 

We are going to compare the expansions: 
\[
\frac{\|D_{x_{-m}} f^{M+1+n_{\sc_{k}}/M}(u_{-m'})\|}{\|u_{-m'}\|}
\quad ,\quad 
\frac{\|D_{x'_{-m}} f^{M+1+n_{\sc_{k}}/M}(u'_{-m'})\|}{\|u'_{-m'}\|}\quad ,\quad 
\|D_{\zeta} f^{M+1+n_{\sc_{k}}/M}(w)\|\; .
\]

We notice that for every $j\ge 0$:
\[\|D_\zeta f^{M+1+ j}(w)\|\le \theta^j e^{c^+(M+1)}\; .\]
In particular 
\[\|D_\zeta f^{M+1+m'/M}(w)\|\le \theta^{m'/M } e^{c^+(M+1)}\; .\]

On the other hand the distance between $\zeta$ and $x'_{-m'}$ is at most $Cst \cdot
e^{-n_{\sc_k }c/3}$ and likewise for the angle between $w$ and $u_{-m'}$. 
Then a classical computation gives:
\[\|D_\zeta f^{M+1+j}(w) -D_{x'_{-m}} f^{M+1+j}(u'_{-m'}/\|u'_{-m'} \|)\|\le Cst \cdot (M+2+j) e^{2c^+(M+1+j) }e^{-n_{\sc_k }c/3}\; .\]
On the other hand, Sub-Lemma \ref{lesublemma} implies:
\[  \|D_{x_{-m}} f^{M+1+j}(\frac{u_{-m'}}{\|u_{-m'} \|}) -D_{x'_{-m}} f^{M+1+j}(\frac{u'_{-m'}}{\|u'_{-m'} \|})\|\le
 Cst \cdot (M+2+j) e^{2c^+(M+1+j) } e^{2C(x)}(e^{4c+}\theta)^{(m-m')/2}\; .\]

Thus:
\[\|D_{x_{-m}} f^{M+1+n_{\sc_k}/M}(\frac{u_{-m'}}{\|u_{-m'} \|})\|\le 
\theta^{n_{\sc_k}/M} e^{c^+(M+1)} +e^{3c^+(M+1+n_{\sc_k}/M)}(e^{-n_{\sc_k}c/3} + e^{2C(x)} (e^{4c+}\theta)^{(m-m')/2})\]
As $n_{\sc_k}\le \Xi (m-m')$ and since $\theta^{1/\Xi}$ and $\theta^{1/M}$ are much smaller than $e^{-c^+}$, it comes:
\[\|D_{x_{-m}} f^{M+1+n_{\sc_k}/M}(\frac{u_{-m'}}{\|u_{-m'} \|})\|\le e^{-c n_{\sc_k}/4}+ e^{2C(x)} \theta^{n_{\sc_k}/(3\Xi)} \; .\]
Thus  by (\ref{eqA7}), the integer $n_{\sc_k}$ is bounded by a function of $C(x)$. 
%
%
%
%
%
%
%
%
%

\end{proof}

\section{Proofs on the geometry of the regular boxes and their hyperbolic properties}

\label{sectionpreuvegeoYg}
\begin{proof}[ Proof of Proposition \ref{geoYg}]
%

The proof if classical (compare with Prop. 3.6 \cite{berhen}), thus we proceed quickly.

Let $\chi$ the cone field  on $Y_\se$ defined by:
\[ \chi:=\{u=(u.x,u.y): \; |u.y|\le \theta|u.x|\}\;.\]
 
Let $\sg=\sa_1\cdots \sa_i$ be a  sequence of $\sA$-symbols which is regular (resp. weakly regular).

\paragraph{Third item} By proceeding as for Proposition 3.6 of \cite{berhen} (see \textsection 14.2, equations (44) and (45)), we can prove that for every $z\in Y_\sg$, for every $j<k$, every unit vector $u\in \chi$ is sent by $D_zf^{n_{\sa_1\cdots \sa_j}}$ into a small cone field $\chi_{\sa_{j+1}}\subset \chi $, whose vectors are $e^{c n_{\sa_{j+1}}/3}$-expanded by $Df^{ n_{\sa_{j+1}}}$. From this we deduce immediately the first inequality of the third item of the Proposition.
The second inequality uses the same expansion property plus the linear bound $(\star)$ on the order of $n_{\sa_i}$ given by the regularity definition.  
\paragraph{Fifth  item} Actually every unit  vector $u$ in $\chi_{\sa_{j}}$ satisfies the $h$-time property (up to time $n_{\sa_j}$) by Proposition 5.9 of \cite{berhen}. This means that $\|Df^{n_{\sa_j}}(u)\|\ge e^{c k/3} \|Df^{n_{\sa_j}-k}(u)\|$ for every $k\le n_{\sa_j}$. By the mapping property of the cone fields, this implies that for every unit vector $u\in \chi$ and $z\in Y_\sg$:
\[ \|D_zf^{n_\sg}(u)\|\ge e^{c k/3} \|D_zf^{n_\sg-k}(u)\|,\quad \forall k\le n_\sg\; .\]
Then Lemma 2.4 of \cite{YW} implies that the curvature of the curves $f^{n_\sg}(\partial^uY_\sg)$ is $\theta$-small. As these curves have their tangent space in $\chi$, they are flat. 

The length estimate of  $\partial^s Y^\sg$ is given by the second item.

\paragraph{Fourth item} This is a classical statement called binding by Benedicks-Carleson. 
 
The second inequality for the third item implies, by using  Corollary 2.1 of \cite{YW}  that the most contracted direction $e_{n_\sg}$ of $Df^{n_\sg}$ is well defined, of class $C^1$, moreover it is $\theta^k $-contracted by $f^k$ for every  $k\le n_\sg$, and it is $\theta$-$C^1$-close to the most contracted direction of $Df$ which is close to $(1,-P'(x))$. Integrating this vector field, we get curves 
which are $C^2$-close to arcs of parabolas and $\theta^k$-contracted. Such arcs have a small length. Nevertheless it might exit from $Y_\sg$ by $\partial^sY_\sg$ instead of  $\partial^uY_\sg$ (and so might do not intersect a flat stretched curve). If it is the case, we concatenate canonically it with a segment of $\partial^sY_\sg$ so that the new curve has the requested property by the first and second items. 

\paragraph{First and second items}
By the same classical lemmas as for the fourth item, it is a consequence of the following inequality:
\begin{equation}\label{PPCE1}\forall z\in \partial^s Y_\sg,\; \forall u\in \chi;\; \forall k\ge 0,
\quad \| D_zf^k(u)\|\ge e^{-M\Xi k}\|u\|\;.\end{equation}

By the mapping cone property of the parabolic piece, it is sufficient to 
prove that for every parabolic piece $\sp= \square_\pm (\sc_i-\sc_{i+1})$, we have for every $u\in \chi_{\sp}$ and $z\in \partial^s Y_\sp$:

\begin{equation}\label{PPCE2}\|D_zf^{n_{\sp}+k}(u)\|\ge e^{-M\Xi k}\|u\|\quad  \forall k\ge 0.\end{equation}
Then it is clear that   $\partial^s Y_\sp$ is formed by two segments close to a parabola.  Also by the fifth item, an induction gives that $\partial^sY_\sg$ is made by two curves, which are included in $\cup_{k=0}^m f^{-n_{\sa_0\cdots \sa_k}} (\partial^s Y_{\sa_{k+1}})$.  Also the third item and (\ref{PPCE2}) imply (\ref{PPCE1}).

In order to prove (\ref{PPCE2}), we put  $\sc_{i+1}=\sc_i\star \sb$.

In the proof of Proposition 5.9 of \cite{berhen} done in \textsection 14.2, equations (44) and (45) imply that 
for $z\in \partial^s Y_\sp$ and $u\in \chi_\sp$, the vector $Tf^{n_\sp}(u)$ belongs to the `` cone field of the canonical extension of $\sb$''. This cone field satisfies the $h$-times property by Proposition 5.9 of \cite{berhen} and is sent into $\chi$ by $Tf^{n_{\sb}}$. Consequently for every $z\in \partial^s Y_\sp$, $u\in \chi_\sp(z)$, $k\le n_{\sb}$:
\begin{equation}\label{pour pce2}f^{n_\sp+n_{\sb}}(z)\in \partial^s Y_\se,\quad D_z f^{n_\sp+n_{\sb}}(u)\in \chi\quad \text{and}\quad \|D_zf^{n_\sp+n_{\sb}}(u)\|\ge e^{(n_{\sb}-k)c/3} \|D_zf^{n_\sp+k}(u)\|.\end{equation}  
As $A$ is a repelling fixed points, for every vector $u'\in \chi$ and $z'\in \partial^s Y_\se$, it holds 
\[\|D_{z'}f^k(u')\|\ge  \|u'\|,\quad \forall k\ge 0,\]
and so it comes that for every $z\in \partial^s Y_\sp$, $u\in \chi_\sp(z)$:
\begin{equation}\label{pour pce3} \|D_zf^{k }(u)\|\ge e^{(n_\sp+n_{\sb})c/3} \|u\|, \quad \forall k\ge n_\sp+n_{\sb}.\end{equation}  

We recall that by definition of common sequences, $n_{\sb}\le M+  n_\sp/\Xi$. Thus from (\ref{pour pce2}) and (\ref{pour pce3}), for every $k\ge n_\sp$:
\begin{equation}\label{pour pce4}\|D_zf^{k}(u)\|\ge e^{-Mc^+k} \|u\|\end{equation}  
which implies (\ref{PPCE2}).

\end{proof}

\end{appendix}
{\addcontentsline{toc}{part}{Index}\printindex}

\bibliographystyle{alpha}
\bibliography{references}

\def\cprime{$'$}
\begin{thebibliography}{BBG06}

\bibitem[BBG06]{BBG06}
M.~Boyle, J.~Buzzi, and R.~G{\'o}mez.
\newblock Almost isomorphism for countable state.
\newblock {\em Journal fur die reine und angewandte Mathematik (Crelles
  Journal)}, 592:pp. 23--47, 2006.

\bibitem[BC85]{BC1}
M.~Benedicks and L.~Carleson.
\newblock On iterations of $1-ax^2$.
\newblock {\em Ann. Math.}, 122:1--25, 1985.

\bibitem[BC91]{BC2}
M.~Benedicks and L.~Carleson.
\newblock The dynamics of the {H}\'enon map.
\newblock {\em Ann. Math.}, 133:73--169, 1991.

\bibitem[Ber11]{berhen}
P.~Berger.
\newblock Abundance of one dimensional non uniformly hyperbolic attractors for
  surface endomorphisms.
\newblock {\em arXiv:0903.1473v2}, 2011.

\bibitem[BV01]{BV2}
Michael Benedicks and Marcelo Viana.
\newblock Solution of the basin problem for {H}\'enon-like attractors.
\newblock {\em Invent. Math.}, 143(2):375--434, 2001.

\bibitem[BY93]{BY}
Michael Benedicks and Lai-Sang Young.
\newblock Sina\u\i-{B}owen-{R}uelle measures for certain {H}\'enon maps.
\newblock {\em Invent. Math.}, 112(3):541--576, 1993.

\bibitem[CE83]{CE83}
P.~Collet and J.-P. Eckmann.
\newblock Positive {L}iapunov exponents and absolute continuity for maps of the
  interval.
\newblock {\em Ergodic Theory Dynam. Systems}, 3(1):13--46, 1983.

\bibitem[CS09]{sarig2}
Van Cyr and Omri Sarig.
\newblock Spectral gap and transience for {R}uelle operators on countable
  {M}arkov shifts.
\newblock {\em Comm. Math. Phys.}, 292(3):637--666, 2009.

\bibitem[Gur69]{Gu1969}
B.~M. Gurevi{\v{c}}.
\newblock Topological entropy of a countable {M}arkov chain.
\newblock {\em Dokl. Akad. Nauk SSSR}, 187:715--718, 1969.

\bibitem[Hof81]{H81}
Franz Hofbauer.
\newblock On intrinsic ergodicity of piecewise monotonic transformations with
  positive entropy. {II}.
\newblock {\em Israel J. Math.}, 38(1-2):107--115, 1981.

\bibitem[Jak81]{Ja81}
M.~V. Jakobson.
\newblock Absolutely continuous invariant measures for one-parameter families
  of one-dimensional maps.
\newblock {\em Comm. Math. Phys.}, 81(1):39--88, 1981.

\bibitem[KH95]{Katok-Haselblat}
A.~Katok and B.~Hasselblatt.
\newblock {\em Introduction to the modern theory of dynamical systems},
  volume~54 of {\em Encyclopedia of Mathematics and its Applications}.
\newblock Cambridge University Press, Cambridge, 1995.
\newblock With a supplementary chapter by Katok and L. Mendoza.

\bibitem[KS79]{KS79}
Michael Keane and Meir Smorodinsky.
\newblock Bernoulli schemes of the same entropy are finitarily isomorphic.
\newblock {\em Ann. of Math. (2)}, 109(2):397--406, 1979.

\bibitem[LY85a]{LYI85}
F.~Ledrappier and L.-S. Young.
\newblock The metric entropy of diffeomorphisms. {I}. {C}haracterization of
  measures satisfying {P}esin's entropy formula.
\newblock {\em Ann. of Math. (2)}, 122(3):509--539, 1985.

\bibitem[LY85b]{LYII85}
F.~Ledrappier and L.-S. Young.
\newblock The metric entropy of diffeomorphisms. {II}. {R}elations between
  entropy, exponents and dimension.
\newblock {\em Ann. of Math. (2)}, 122(3):540--574, 1985.

\bibitem[MV93]{MV93}
L.~Mora and M.~Viana.
\newblock Abundance of strange attractors.
\newblock {\em Acta. Math.}, 171:1--71, 1993.

\bibitem[New89]{NH}
Sheldon~E. Newhouse.
\newblock Continuity properties of entropy.
\newblock {\em Ann. of Math. (2)}, 129(2):215--235, 1989.

\bibitem[PSZ15]{Pe10}
Y.~Pesin, S.~Senti, and Ke~Zhang.
\newblock {T}hermodynamics of {T}owers of {H}yperbolic {T}ype.
\newblock {\em Trans. Amer. Math. Soc.}, (to appear), 2015.

\bibitem[PY09]{PY09}
Jacob Palis and Jean-Christophe Yoccoz.
\newblock Non-uniformly hyperbolic horseshoes arising from bifurcations of
  {P}oincar{\'e} heteroclinic cycles.
\newblock {\em Publ. Math. Inst. Hautes {\'E}tudes Sci.}, (110):1--217, 2009.

\bibitem[Rud82]{Ru82}
Daniel~J. Rudolph.
\newblock A mixing {M}arkov chain with exponentially decaying return times is
  finitarily {B}ernoulli.
\newblock {\em Ergodic Theory Dynamical Systems}, 2(1):85--97, 1982.

\bibitem[Sar13]{Sa10}
Omri~M. Sarig.
\newblock Symbolic dynamics for surface diffeomorphisms with positive entropy.
\newblock {\em J. Amer. Math. Soc.}, 26(2):341--426, 2013.

\bibitem[Sen03]{Se03}
Samuel Senti.
\newblock Dimension of weakly expanding points for quadratic maps.
\newblock {\em Bull. Soc. Math. France}, 131(3):399--420, 2003.

\bibitem[Tak11]{Ta11}
Hiroki Takahasi.
\newblock Abundance of non-uniform hyperbolicity in bifurcations of surface
  endomorphisms.
\newblock {\em Tokyo J. Math.}, 34(1):53--113, 2011.

\bibitem[Tak12]{Ta12}
Hiroki Takahasi.
\newblock Prevalent dynamics at the first bifurcation of {H}{\'e}non-like
  families.
\newblock {\em Comm. Math. Phys.}, 312(1):37--85, 2012.

\bibitem[Tak13]{Ta13}
H.~Takahasi.
\newblock Prevalence of non-uniform hyperbolicity at the first bifurcation of
  h\'enon-like families.
\newblock {\em arXiv:1308.4199}, 2013.

\bibitem[Tsu93]{Ts93}
Masato Tsujii.
\newblock A proof of {B}enedicks-{C}arleson-{J}acobson theorem.
\newblock {\em Tokyo J. Math.}, 16(2):295--310, 1993.

\bibitem[VJ67]{V-J67}
D.~Vere-Jones.
\newblock Ergodic properties of nonnegative matrices. {I}.
\newblock {\em Pacific J. Math.}, 22:361--386, 1967.

\bibitem[WY01]{YW}
Q.~D. Wang and L.~S. Young.
\newblock attractors with one direction of instability.
\newblock {\em Commun. Math. Phys.}, 218:1--97, 2001.

\bibitem[Yoc97]{Y}
J.-C. Yoccoz.
\newblock A proof of {J}akobson's theorem.
\newblock {\em Manuscript}, 1997.

\bibitem[You98]{Yo98}
Lai-Sang Young.
\newblock Statistical properties of dynamical systems with some hyperbolicity.
\newblock {\em Ann. of Math. (2)}, 147(3):585--650, 1998.

\end{thebibliography}

\end{document}